\documentclass[aos, preprint]{imsart}

\RequirePackage{amsthm,amsmath,amsfonts,amssymb}
\RequirePackage[authoryear]{natbib}
\RequirePackage[colorlinks,citecolor=blue,urlcolor=blue]{hyperref}
\RequirePackage{graphicx}

\usepackage[english]{babel}
\usepackage[utf8]{inputenc}

\usepackage{amsmath}
\usepackage{amsfonts}
\usepackage{amssymb}
\usepackage{amsthm}
\usepackage{mathtools}
\usepackage{paralist}
\usepackage{comment}
\usepackage{anyfontsize}

\usepackage[dvipsnames]{xcolor}
\usepackage{dsfont}
\usepackage{bm}

\usepackage{graphicx}
\usepackage{enumerate}
\usepackage{multirow}
\usepackage{booktabs}

\usepackage{array}
\usepackage{subcaption}

\theoremstyle{plain}
\newtheorem{theorem}{Theorem}[section]
\newtheorem{proposition}[theorem]{Proposition}
\newtheorem{lemma}[theorem]{Lemma}
\newtheorem{corollary}[theorem]{Corollary}

\theoremstyle{definition}

\newtheorem{assumption}[theorem]{Assumption}
\newtheorem{remark}[theorem]{Remark}
\newtheorem{example}[theorem]{Example}

\allowdisplaybreaks[4]

\numberwithin{equation}{section}

\newcommand{\R}{\mathbb{R}}

\newcommand{\N}{\mathbb{N}}
\newcommand{\eps}{\varepsilon}

\newcommand{\Ab}{\mathbb {A}}
\newcommand{\Bb}{\mathbb {B}}

\newcommand{\Db}{\mathbb {D}}
\newcommand{\Eb}{\mathbb {E}}

\newcommand{\Lb}{\mathbb {L}}

\newcommand{\Wb}{\mathbb {W}}

\newcommand{\Ac}{\mathcal{A}}

\newcommand{\Ic}{\mathcal{I}}

\newcommand{\Nc}{\mathcal{N}}

\newcommand{\weak}{\rightsquigarrow}

\newcommand{\Prob}{\mathbb{P}}    
\newcommand{\Exp}{\mathbb {E}}
\newcommand{\Var}{\operatorname{Var}}

\newcommand{\diff}{\mathrm d}

\newcommand{\oror}{\text{ or } \dots \text{ or }}

\renewcommand{\subset}{\subseteq}

\newcommand{\Wmtl}{W_m^0(T)}

\newcommand{\Kl}{K_L} 

\newcommand{\Lhat}{\widehat L} 
\newcommand{\Loracle}{\widetilde L} 
\newcommand{\Lpreasy}{\widetilde \mu} 
\newcommand{\Rhat}{\widehat R} 
\newcommand{\pdLhat}[1]{\widehat{\partial_{#1} L}}
\newcommand{\pdL}[1]{\partial_{#1} L}

\newcommand{\fmin}{f_{Q,L}} 
\newcommand{\nearmin}{\eta} 
\newcommand{\thetainterior}{\kappa} 
\newcommand{\mestimator}{\hat \theta_n} 
\newcommand{\bg}{\bm{g}}
\newcommand{\bpsi}{\bm{\psi}}
\newcommand{\Cholder}{C_{h}}
\newcommand{\gammaholder}{\gamma_{h}}

\newcommand{\gapp}{g}

\newcommand{\neunzig}{C_{s}} 
\newcommand{\CCCK}{C_{g}} 

\makeatletter
\newcommand*\rel@kern[1]{\kern#1\dimexpr\macc@kerna}
\newcommand*\widebar[1]{%
  \begingroup
  \def\mathaccent##1##2{%
    \rel@kern{0.8}%
    \overline{\rel@kern{-0.8}\macc@nucleus\rel@kern{0.2}}%
    \rel@kern{-0.2}%
  }%
  \macc@depth\@ne
  \let\math@bgroup\@empty \let\math@egroup\macc@set@skewchar
  \mathsurround\z@ \frozen@everymath{\mathgroup\macc@group\relax}%
  \macc@set@skewchar\relax
  \let\mathaccentV\macc@nested@a
  \macc@nested@a\relax111{#1}%
  \endgroup
}
\makeatother

\newcommand{\bx}{\boldsymbol{x}}
\renewcommand{\L}{\mathbb{L}}

\DeclarePairedDelimiter\autobracket{(}{)}
\newcommand{\br}[1]{\autobracket*{#1}}
\newcommand{\norm}[1]{\left\lVert#1\right\rVert}
\newcommand{\leftinv}[1]{#1^{\leftarrow}}
\newcommand{\abs}[1]{\left|#1\right|}
\newcommand{\matnorm}[1]{{\left\vert\kern-0.25ex\left\vert\kern-0.25ex\left\vert #1 
		\right\vert\kern-0.25ex\right\vert\kern-0.25ex\right\vert}}

\begin{document}

\begin{frontmatter}

\title{Empirical tail dependence functions in high dimensions:\\ uniform  linearizations and inference}
\runtitle{Empirical tail dependence functions}

\begin{aug} 
\author[A]{\fnms{Axel} \snm{B\"ucher}\ead[label=e1]{axel.buecher@rub.de}\orcid{0000-0002-1947-1617}}
\author[B]{\fnms{Yeonjoon} \snm{Choi}\ead[label=e2]{yeonjoon.choi@mail.utoronto.ca}}
\author[A]{\fnms{Katharina} \snm{Effertz}\ead[label=e3]{katharina.effertz@rub.de}}
\and
\author[B]{\fnms{Stanislav} \snm{Volgushev}\ead[label=e4]{stanislav.volgushev@utoronto.ca}}	

\address[A]{Ruhr-Universität Bochum, Fakultät für Mathematik\printead[presep={ ,\ }]{e1,e3}}
\address[B]{Department of Statistical Sciences, University of Toronto\printead[presep={ ,\ }]{e2,e4}}

\end{aug}

\begin{abstract} 
The analysis of extremal dependence in high dimensions is a key challenge in modern extreme-value statistics. Existing methodology primarily focuses on modeling and estimation of extremal dependence structures, often supported by concentration bounds for empirical tail quantities. However, comparatively little is known about general inferential procedures in high-dimensional extremes. In this paper, we develop foundational results that enable inference for rank-based empirical tail dependence coefficients, stable tail dependence functions, and functionals derived from them. We start by establishing finite-sample probability bounds that quantify the linearization error for such estimators uniformly over collections of coordinates. Moreover, we derive high-dimensional central limit theorems and establish the validity of multiplier bootstrap procedures for collections of empirical tail dependence statistics. Within an asymptotic framework, our results allow the dimension to grow exponentially with the effective sample size. We illustrate the usefulness of the results through two applications: uniform expansions and normal approximations for M-estimators of tail dependence parameters and inference for spatial isotropy based on collections of tail dependence functions.
\end{abstract}

\begin{keyword}[class=MSC]
	\kwd[Primary ]{62G20}
	\kwd{62G32}
	\kwd[; secondary ]{62G09}
\end{keyword}

\begin{keyword}
	\kwd{Extreme value statistics}
	\kwd{High dimensional statistics}
	\kwd{Multiplier bootstrap}
	\kwd{Tail dependence}
	\kwd{Tail correlation}
\end{keyword}

\end{frontmatter}

\section{Introduction}
\label{sec:introduction}

\addtocontents{toc}{\protect\setcounter{tocdepth}{-100}}

Extreme value theory studies the probabilistic behavior and statistical analysis of rare events, that is, realizations of a random sample occurring at unusually high (or low) levels \citep{BeiGoeSegTeu04, DehFer06}.
A central object of interest is \emph{tail dependence}, which describes the strength and structure of dependence between components of a random vector when some coordinates take extreme values. Understanding tail dependence is crucial for analyzing events driven or amplified by simultaneous extreme values accross multiple variables, with examples ranging from floods \citep{keef2009spatial, keef2013estimating} over climate extremes \citep{zscheischler2017dependence} to financial crises \citep{poon2004extreme, zhou2010banks}.  Mathematically, tail dependence can be characterized using various equivalent objects, including stable tail dependence functions (STDF) and tail copulas, exponent and spectral measures, and Pickands dependence functions; see Chapters~8 and~9 in \cite{BeiGoeSegTeu04} and Chapters~6 and~7 in \cite{DehFer06}.

Motivated by applications involving large spatial fields or high-dimensional financial data, there has been rapidly growing interest in modeling and analyzing high-dimensional extremes. In such settings, fully nonparametric approaches are often difficult to interpret and may be computationally infeasible.
Moreover, extreme value methods are particularly susceptible to the curse of dimensionality, as estimation relies solely on tail observations. These challenges have led to a variety of approaches that provide parsimonious and structured descriptions of tail dependence in high dimensions \citep{EngelkeIvanovs2021}. Popular approaches include clustering methods \citep{fomichov2023spherical, avellamedina2024spectral, boulin2024high, chen2025clustering}, principal component analysis \citep{drees2021principal, reinbott2026principal}, factor models \citep{boulin2026dimension}, graphical modeling and structure learning based on directed and undirected graphs \citep{EngelkeHitz2019, EngVol22, amendola2022conditional, wan2023graphical, lederer2023extremes, tran2024estimating, EngLalVol25} and vine copula constructions tailored to extremes \citep{kiriliouk2025x}.  

When it comes to a formal mathematical analysis of the methods, some of the above works explicitly allow the dimension to grow with the sample size, a setting that is arguably most relevant for many modern applications. However, the available theoretical guarantees in this regime remain limited: either the proposed methods lack a rigorous theoretical analysis altogether, or they rely predominantly on concentration inequalities. The latter have been established for empirical (rank-based) tail dependence quantities by \cite{GoiSabCle15}, with subsequent refinements in \cite{LhaSabSeg22,ClemenconJalalzaiLhautSabourinSegers2023} and \cite{EngLalVol25}. While such results provide non-asymptotic bounds that quantify stochastic fluctuations and thus yield useful performance guarantees, they do not deliver distributional approximations and are therefore inherently insufficient for non-conservative inference in the form of confidence intervals or hypothesis tests.

To the best of our knowledge, the few existing contributions that address inference for extremes in growing dimensions do not cover the problem of tail dependence. \cite{ChenZhou2026} develop tests for marginal tail parameters of high-dimensional random vectors, relying on techniques specific to univariate extremes. \cite{sasaki2024high} study a regression framework with high-dimensional predictors, focusing on the tail behavior of a univariate response conditional on covariates. Neither approach provides tools for inference on the extremal dependence structure.

The present paper develops tools for inference on tail dependence measures that comes with formal theoretical guarantees. Our focus is on STDFs and tail copulas, which are key building blocks in many modern methodologies for both low- and high-dimensional extremes. In fixed dimensions, the statistical properties of their empirical counterparts are well understood, typically through large-sample asymptotics in the form of (functional) central limit theorems. Foundational contributions were made by \cite{Hua92,DreHua98,draisma2004bivariate}; their results have been extended in various directions by \cite{EinKraSeg12,BucSegVol14,EinmahlSegers2021,lalancette2021rank}. Complementary bootstrap methods were developed in \cite{BucDet13}, and the resulting theory has been applied to parametric estimation in spatial models by \cite{einmahl2016m}.
A key challenge in this line of work is that the estimators are rank-based, which complicates the analysis as one must account for the stochastic fluctuations of empirical ranks in addition to those arising from the unknown tail dependence.\footnote{At the same time, rank-based methods are attractive because they avoid modeling marginal tails and can be more efficient than corresponding oracle procedures based on the true marginal distributions \citep{Bucher2014}.}
However, the established theoretical tools and results do not readily extend to growing dimensions. In particular, (functional) weak convergence is no longer meaningful when the dimension of the ambient space increases. Moreover, existing results provide no quantitative insight into how the dimension affects the accuracy of distributional approximations.

We overcome these challenges through a two-step approach. In the first step, we derive linear representations of the empirical estimators, where the leading term is expressed as a sum of independent random variables. We establish convergence rates and provide explicit finite-sample probability bounds for the remainder terms. In particular, we identify regimes in which the remainder is asymptotically negligible relative to the leading term, even as the dimension grows. Our approach is inspired by related developments for empirical copulas in \cite{BucherPakzad2025}, with a key application consisting of linearizations that hold uniformly over large collections of lower-dimensional margins, such as all bivariate margins. This type of result is particularly relevant for high-dimensional models characterized by pairwise dependence structures, including the Hüsler–Reiss model.
In the second step, we leverage recent advances in high-dimensional Gaussian approximation \citep{Che13,Che17,Che22}, combined with multiplier bootstrap techniques \citep{CheCheKat23}, to enable inference for the leading term. In this way, we extend bootstrap-based inferential methods for STDFs from the fixed-dimensional setting \citep{BucDet13} to the high-dimensional regime.

We illustrate the scope of the results in two applications. First, we study M-estimators for tail dependence parameters in the spirit of \cite{EinKraSeg08,EinKraSeg12} and derive uniform asymptotic expansions and normal approximation in high dimensions. Second, we consider testing isotropy in spatial extremal dependence structures, where the proposed multiplier bootstrap enables inference for large collections of tail dependence functions. Simulation experiments illustrate the finite-sample performance of the procedures.

The remaining parts of this paper are organized as follows. Section~\ref{sec:tail-dependence} introduces tail dependence functions and their empirical counterparts. Section~\ref{sec:linearizations} establishes the uniform linearization results that form the basis of our analysis. Section~\ref{sec:clts} derives high-dimensional central limit theorems and establishes the validity of multiplier bootstrap procedures. Section~\ref{sec:applications} discusses two applications, namely M-estimation for tail dependence parameters and testing spatial isotropy. Proofs of the main results are collected in Appendix~\ref{sec:proofs}, while auxiliary technical results are deferred to Appendix~\ref{sec:auxiliary-results}. 

\subsection{Notation}
\label{sec:notation}
For $d\in\N$, we write $[d]=\{1, \dots, d\}$. For a real-valued function $f$ defined on a set $B \subset \R^d$  and $\eps>0$, let
\begin{align} \label{eq:definition-modulus}
\omega_{f}(\eps;B) 
= 
\sup \big\{ |f(\bm u) - f(\bm v)|: \bm u, \bm v \in B, \| \bm u - \bm v\|_\infty \le \eps\big\}
\end{align}
denote the modulus of continuity with respect to the maximum norm on $\R^d$. 
For $\emptyset \ne I\subset [d]$ and $\bm x \in [-\infty, \infty]^d$ write $\bm x_I=(x_i)_{i \in I} \in [-\infty, \infty]^{I}$ for the vector made up by the coordinates of $\bm x$ that belong to $I$; note that we consider the vector to be indexed by $I$ and not by $\{1, \dots, |I|\}$. The same convention is applied for functions $f_I$ defined on a subset $B_I$ of $\R^I$. If existent, we denote the partial derivative of $f_I$ at $\bm x_I \in B_I$ with respect to the $j$th coordinate ($j\in I$) by $\partial_jf_I(\bm x_I)=\lim_{h \to 0} h^{-1} \{ f_I(\bm x_I + h \bm e_{I,j}) - f_I(\bm x_I)\}$, where $\bm e_{I,j} \in \R^I$ has coordinates $\bm 1(i=j)$ for $i \in I$. For a set $A \subset [0,\infty)^d$ and $\eps>0$, let $A^{\oplus \eps}= \{ \bm x \in [0, \infty)^d : \mathrm{dist}(\bm x, A) \le \eps\}$ denote the $\eps$-enlargement of $A$ in $[0,\infty)^d$, where $\mathrm{dist}(\bm x, A) := \inf\{\|\bm x - \bm y\|_\infty: \bm y \in A\}$ is based on maximum-norm $\| \cdot \|_\infty$ on $\R^d$. Finally, $\|\cdot\|_p$ denotes the $p$-norm, for $p \ge 1$, and $\mathcal N_d(\bm \mu, \bm \Sigma)$ denotes the $d$-variate normal distribution with mean $\bm \mu$ and variance matrix $\bm \Sigma$.

\section{Tail dependence functions and their empirical counterparts}
\label{sec:tail-dependence}

Let $\bm X=(X_1, \dots, \allowbreak X_d)^\top \in \R^d$ denote a $d$-variate random vector with common cumulative distribution function (cdf) $F$ and continuous marginal cdfs $F_1,\ldots,F_d$. As is standard in multivariate extremes, we assume that the dependence structure of $\bm X$ stabilizes in the tail. Formally, this can be characterized through the existence of the stable tail dependence function $L:[0,\infty)^d \to [0,\infty)$ or the tail copula $R:[0,\infty]^d \setminus \{ \bm \infty \} \to [0,\infty)$ of $\bm X$, which are defined by
\begin{align}
    L(\bm x) \label{eq:defL}
    &= 
    \lim_{t \to 0} t^{-1} \Prob\big( \exists j \in [d]: F_j(X_j)>1-tx_j\big), 
    \\ \label{eq:defR}
    R(\bm x) 
    &= 
    \lim_{t \to 0} t^{-1} \Prob\big( \forall j \in [d]: F_j(X_j)>1-tx_j\big), 
\end{align}
respectively. Both functions characterize the extremal dependence of $\bm X$, and by inclusion-exclusion, we have
\[
L(\bm x) = \sum_{\emptyset \ne I \subset [d]} (-1)^{|I|+1} R_I(\bm x_I), \qquad
R(\bm x) = \sum_{\emptyset \ne I \subset [d]} (-1)^{|I|+1} L_I(\bm x_I),
\]  
where $L_I(\bm x_I) = L(\bm x_I^0)$ and $R_I(\bm x_I)=R( \bm x_I^\infty)$ with $\bm x_I^a$ the vector having coordinates $x_j$ for $j \in I$ and $x_j = a$ for $j \in [d] \setminus I$, for $a \in \{0,\infty\}$. Note that 
\begin{align*}
    L_I(\bm x_I) 
    &= 
    \lim_{t \to 0} t^{-1} \Prob\big( \exists j \in I: F_j(X_j)>1-tx_j\big) 
    \\
    R_I(\bm x_I) 
    &= 
    \lim_{t \to 0} t^{-1} \Prob\big( \forall j \in I: F_j(X_j)>1-tx_j\big) 
\end{align*}
are nothing else than the stable tail dependence function and the tail copula of the sub-vector $\bm X_I=(X_i)_{i\in I}$, which are formally functions $L_I:[0,\infty)^I \to [0,\infty)$ and $R_I:[0,\infty]^I \setminus \{ \bm \infty \} \to [0,\infty)$.

Evaluating $L_I$ and $R_I$ at the $\bm 1$-vector, we obtain the extremal coefficient $\theta_I$ \citep{Sch03} and the joint tail coefficient $\chi_I$, that is, 
\begin{align} 
\label{eq:tail-coefficients}
\theta_I = L_I(\bm 1_I), \qquad \chi_I = R_I(\bm 1_I).
\end{align}
Note that $\chi_I=2-\theta_I = \lim_{t \to 0} \Prob(F_j(X_j) > 1-t \mid F_{j'}(X_{j'})>1-t)$ for $I=\{j,j'\}$ of cardinality $|I|=2$, which is also known as the upper tail dependence coefficient \citep{SchSta06} or the tail correlation. The matrix of pairwise tail correlations $(\chi_I)_{I \subset [d]: |I|=2}$ plays a fundamental role in multivariate extreme value analysis \citep{EngLalVol25}.

\begin{example}[H\"usler-Reiss distributions]\label{ex:HR1} 
The H\"usler-Reiss distribution has played a central role in recent developments on graphical modeling for extremes \citep{EngelkeHitz2019}. Its STDF is parametrized in terms of a $d$-dimensional symmetric, conditionally negative definite matrix $\Gamma=(\gamma_{j\ell})$ with non-negative entries satisfying $\gamma_{jj}=0$ for each $j \in [d]$, and is given by
\begin{align*}
L(\bm x; \Gamma)
=
\sum_{j=1}^d
x_j\,
\Phi_{d-1}\!\Big(
\Big(
\log\frac{x_j}{x_\ell}
+
\frac{\gamma_{\ell j}}{2}
\Big)_{\ell\neq j};
\,
\Sigma^{(j)}
\Big),
\end{align*}
where $\Phi_d(\cdot; \Sigma)$ is the cdf of the $(d-1)$-variate normal distribution with covariance matrix $\Sigma$ and where $\Sigma^{(j)} = (\Sigma^{(j)}_{\ell m})_{\ell,m \in [d] \setminus \{j\}}$ has entries $\Sigma^{(j)}_{\ell m} = (\gamma_{\ell j} + \gamma_{mj} - \gamma_{\ell m})/2$ \citep[Example~1]{HuPengSegers2024}. The bivariate marginal STDFs are given by 
\begin{align*}
L_{I}(x_j,x_\ell; \gamma_{j\ell})
=
x_j\,
\Phi\!\Big(
\frac{\log(x_j/x_\ell)}{\sqrt{\gamma_{j\ell}}}
+
\frac{\sqrt{\gamma_{j\ell}}}{2}
\Big)
+
x_\ell\,
\Phi\!\Big(
\frac{\log(x_\ell/x_j)}{\sqrt{\gamma_{j\ell}}}
+
\frac{\sqrt{\gamma_{j\ell}}}{2}
\Big), \qquad I = \{j,\ell\},
\end{align*}
which shows that the parameter matrix $\Gamma$ can be fully recovered from the bivariate margins only. Note that $\lim_{\gamma \to +\infty} L_{I}(x_j,x_\ell; \gamma)= x_j+x_\ell$ and $\lim_{\gamma\to 0} L_{I}(x_j,x_\ell;\gamma)= x_j \vee x_\ell$.
\end{example}

\begin{example}[Factor models and max-linear models] \label{ex:fac1}
As argued in \citet[Section 6, Lemma 6.1]{EinKraSeg12}, factor models with heavy-tailed factors and light-tailed noise lead to a STDF of the form 
\[
L(\bm x;B) = \sum_{j=1}^r \max_{\ell=1}^d (b_{j\ell} x_\ell), \qquad \bm x \in [0,\infty)^d,
\]
where $B=(b_{j\ell})_{j \in [r], \ell \in [d]} \in [0,1]^{r \times d}$ has column sums 1.
Such STDFs also arise in max-linear models on directed acyclic graphs which have recently gained popularity in modeling causal structural relationships in the tail \citep{gissibl2018max}. 
\end{example}

We next introduce empirical tail dependence functions.
Let $\bm X_1, \dots, \bm X_n$ denote an i.i.d.\ sample of $\bm X$, with $\bm X_i=(X_{i1}, \dots, X_{id})^\top$. For $j\in \{1, \dots, d\}$, let $R_{ij}$ denote the rank of $X_{ij}$ among $X_{1j}, \dots, X_{nj}$. The empirical stable tail dependence function and the empirical tail copula are defined as
\begin{align} \label{def:hatL}
\Lhat_{n}(\bm x) &:= 
\frac1k \sum_{i=1}^n 
\bm 1\big(\exists j \in [d]: R_{ij} > n+1-kx_j\big), 
\\ \label{def:hatR}
\Rhat_{n}(\bm x) &:= 
\frac1k \sum_{i=1}^n 
\bm 1\big(\forall j \in [d]: R_{ij} > n+1-kx_j\big),
\end{align}
where $k \in [n]$ denotes a parameter to be chosen by the statistician that controls the size of the presumed tail area. Note that those estimators can be interpreted as `plug-in' versions of the limiting relations in~\eqref{eq:defL} and~\eqref{eq:defR}. Indeed, replacing $t$ by $k/n$, $F_j$ by the marginal empirical CDF and probabilities by their empirical counterparts leads to expressions that are almost identical to~\eqref{def:hatL} and~\eqref{def:hatR}. In order to obtain consistent estimators for $L$ and $R$, one typically needs to select an intermediate sequence $k = k_n$ which satisfies $k_n \to \infty, k_n/n\to 0$. The challenges in analyzing the estimators $\Lhat_n, \Rhat_n$ are thus two-fold. First, taking ranks introduces dependence across all terms in the sum. Second, the sum is normalized by $1/k$ rather than $1/n$, and the distribution of the summands depends on $n$ and $k$.

In the finite-dimensional case where $d$ is a fixed integer, the asymptotic behavior of $\Lhat_n$ and $\Rhat_n$ is well-studied \citep{Hua92, EinKraSeg12, BucSegVol14}. We present one possible result in a way that is instructive for the developments in later sections. Let
\begin{align} 
\label{eq:l-process-r-process-def}
\mathbb L_n = \sqrt k(\Lhat_n - L ), \qquad 
\mathbb R_n = \sqrt k(\Rhat_n -  R )
\end{align}
denote the processes of rescaled estimation errors.

Let $\Lambda$ denote the measure on the Borel subsets of $\Eb_\infty := [0, \infty]^d \setminus \{ \bm \infty\}$ determined by $\Lambda(A(\bm x)) = L(\bm x)$ where
\[
A(\bm x) := \big\{ \bm y \in \Eb_\infty \mid \exists j \in [d]: y_j < x_j\big\}.
\]
Let $\mathbb W_\Lambda$ denote a zero-mean Gaussian process indexed by the Borel sets of $\Eb_\infty$ with covariance function $\Exp[\mathbb W_\Lambda(A) \mathbb W_\Lambda(B)] = \Lambda(A \cap B)$. The process shall be chosen in such a way that $[0,\infty)^d \to \R, \bm x \mapsto \Wb_L(\bm x) := \mathbb W_\Lambda(A(\bm x))$ is continuous almost surely. 
Finally, define $\bm V_i=(V_{i1}, \dots, V_{id})^\top$ with $V_{ij}=1-F_j(X_{ij})$ for $j\in[d]$ and $i\in[n]$, and let
\begin{align}
\label{eq:Ltilde}
\Loracle_n(\bm x) 
&= 
\frac1k \sum_{i=1}^n \bm 1\Big( \exists j \in [d]: V_{ij} < \frac{k}nx_j \Big) 
\\ \label{eq:mutilde}
\Lpreasy_n(\bm x) 
&= \frac{n}k \Prob\Big(\exists j \in [d]: V_{ij} < \frac{k}nx_j \Big)
\end{align}
and $\widetilde \Lb_n(\bm x) = \sqrt k\big\{ \Loracle_n(\bm x) - \Lpreasy_n(x) \big\}$. Note that $\widetilde \Lb_n(\bm x)$ has expectation zero. We then have the following result. 

\begin{theorem}[Linearization and weak convergence for fixed $d$, \citealp{EinKraSeg12}]
\label{theo:linear-finite-d}
Suppose that the following conditions are met:
\begin{compactenum}
\renewcommand{\theenumi}{(C1)}
\renewcommand{\labelenumi}{\theenumi}
\item \label{cond:second-order}
There exists $\alpha>0$ such that $\sup_{\bm x \in \Delta_{d-1}} \big| t^{-1} \Prob(F_1(X_1) > 1-tx_1 \oror F_d( X_d) > 1-tx_d) - L(\bm x ) \big| = O(t^\alpha)$ as $t\to0$, where $ \Delta_{d-1} = \{ \bm x \in [0,1]^d: x_1 + \dots + x_d=1\}$. 
\renewcommand{\theenumi}{(C2)}
\renewcommand{\labelenumi}{\theenumi}
\item \label{cond:k-rate}
$k\to\infty$ and $k=o(n^{2\alpha/(1+2\alpha)})$, with $\alpha$ from \ref{cond:second-order}.
\renewcommand{\theenumi}{(C3)}
\renewcommand{\labelenumi}{\theenumi}
\item \label{cond:smoothness-1}
For all $j\in[d]$, the first order partial derivative of $L$ with respect to $x_j$, say $\pdL{j}$, exists and is continuous on the set of points $\bm x$ such that $x_j>0$.
\end{compactenum}
Then, for any fixed $T \in \N$, we have
\begin{align}
\label{eq:linearization}
\sup_{\bm x [0,T]^d} \big| \mathbb L_n(\bm x) - \widebar {\mathbb L}_n(\bm x) \big| = o_\Prob(1),
\end{align}
where
\begin{align} \label{eq:definition-widebar-Ln}
\widebar {\mathbb L}_n(\bm x)  = \widetilde {\mathbb L}_n(\bm x)  - \sum_{j =1 }^d \pdL{j}(\bm x) \widetilde {\mathbb L}_{nj}(x_j).
\end{align}
Here, $\widetilde {\mathbb L}_{nj}(x_j)=\widetilde \Lb_n(0, \dots, 0, x_j, 0, \dots, 0)$, and $\pdL{j}(\bm x)$ is defined as the right-hand derivative at points $\bm x$ with $x_j=0$.
Moreover, we have
$\widetilde{\mathbb L}_n 
= 
\sqrt k (\Loracle_n - \Lpreasy_n) \weak \mathbb W_L$ in $\ell^\infty([0,T]^d),
$
and hence 
\begin{align} 
\label{eq:lnw}
\mathbb L_n = \sqrt k(\Lhat_n - L ) \weak \Bb_L \qquad \text{ in } \ell^\infty([0,T]^d), 
\end{align}
where the limit process $\Bb_L$ has the representation
\[
\mathbb B_L(\bm x) 
= 
\mathbb W_L(\bm x) - \sum_{j=1}^d \pdL{j}(\bm x) \mathbb W_{L,j}(x_j)
\]
with $\mathbb W_{L,j}(x_j) = \mathbb W_L(0, \dots, 0, x_j, 0 \dots, 0)$ for $x_j \ge 0$. 
\end{theorem}

While this result is not stated in any paper in this exact form, it can essentially be extracted from the proofs in \cite{EinKraSeg12}. Note that the weak convergence in~\eqref{eq:lnw} does not make sense if $d$ changes with $n$, whereas the representation in~\eqref{eq:linearization} can be reasonable. The proofs in \cite{EinKraSeg12} and related works, however, rely on the fact that the dimension $d$ is fixed. In the following section, we derive a quantitative version of \eqref{eq:linearization} that gives an explicit rate and tail bound for the difference in there and allows for increasing dimensions $d=d_n\to\infty$. Finally, we note that a simple calculation shows that Assumption \ref{cond:smoothness-1} holds if $L$ is the STDF of a H\"usler-Reiss distribution from Example~\ref{ex:HR1} but fails for the STDF corresponding to factor models in Example~\ref{ex:fac1}.

\section{Non-asymptotic linearization of empirical tail dependence functions}
\label{sec:linearizations}

The main results in this section are two theorems that derive linearizations of the empirical tail dependence process $\mathbb L_n$ under two different regularity assumptions on the partial derivatives of $L$.
For the first theorem, we fix an interesting set $A$, for instance $A=\{\bm 1\}$ to handle the extremal coefficient $\theta=\theta_{[d]}$ from \eqref{eq:tail-coefficients}, and then demand sufficient regularity of $L$ in a small extension of $A$.  For the second one, we start with $L$, and derive uniform linearizations on sets that are adapted to the regularity of $L$ and that are as large as possible. Either approach can be useful, depending on the application. For given $T \in \N, \delta \in (0,e^{-1})$ and $k\in\N$, let
\begin{align} \label{eq:definition-r}
r = r(\delta, T, k) = \sqrt{\frac{T}k \log\Big(\frac1\delta\Big)}.
\end{align}
Further, let 
\begin{align} \label{eq:bias}
    B_n(\bm x) = \sqrt k\big\{\Lpreasy_n(\bm x) - L(\bm x) \big\}, \qquad \bm x \in [0,\infty)^d.
\end{align}
denote the rescaled difference between the preasymptotic STDF and the STDF itself, and write
\begin{align} \label{eq:bias-A}
    B_{n,k}(L;S) := \sup_{\bm x \in S} |B_n(\bm x)|
\end{align}
for $S \subset [0,\infty)^d$.
Our first result will be stated under the following regularity assumption on the pair $(A,L)$.
\begin{enumerate}
\renewcommand{\theenumi}{(C4)}
\renewcommand{\labelenumi}{\theenumi}
\item \label{cond:smoothness-hoelder}
There exists $\kappa_L, \Kl \in(0,\infty)$ and $\alpha_L \in (0,1]$ such that  \vspace{-.2cm}
\begin{multline*}
\qquad \forall j \in [d],\forall \bm x \in A, \forall \bm y \in [0,\infty)^d \text{ with } \|\bm x - \bm y\|_\infty \le \kappa_L: \quad 
\\
\pdL{j}(\bm x), \pdL{j}(\bm y) \text{ exist and satisfy }
|\pdL{j}(\bm x)-\pdL{j}(\bm y)| \le \Kl \|\bm x - \bm y\|_\infty^{\alpha_L}.
\end{multline*}
\end{enumerate}

\begin{theorem} \label{theo:linearization-hoelder}
Let $L$ be a $d$-variate STDF and let $A \subset [0,T]^d$ (with $T \in \N$) be a fixed set such that the pair $(A,L)$ satisfies Assumption~\ref{cond:smoothness-hoelder}. Then, there exist constants $D_1=D_1(d), D_2=D_2(d)$ and $D_3=D_3(d,\Kl,\alpha_L)$ 
such that, for any $n\in\N, k \in [n], \delta \in (0, e^{-1})$ satisfying $\log(d/\delta)\le 2kT/7, n/k \ge T$ and $r \le \kappa_L/\neunzig$ with $\neunzig$
the universal constant from Lemma~\ref{lem:bound-on-order-statistics}, we have
\begin{align*}
\sup_{\bm x \in A } \big| \mathbb L_n(\bm x) -\widebar {\mathbb L}_n(\bm x) \big| 
&\le 
B_{n,k}(L ; A^{\oplus \kappa_L})
+
\frac{d}{\sqrt{k}}
+
D_{1} \sqrt{r \log\Big(\frac{TD_{2}}{\delta r}\Big)} +
D_3 r^{\alpha_L} \sqrt{T \log\Big(\frac1\delta\Big)}.
\end{align*}
with probability at least $1- (6d+5)\delta$, with $r$ from \eqref{eq:definition-r}. More specifically, the constant $D_1$ depends on $d$ via $d^{3/2}$,  while $D_2$ and $D_3$ depend linearly on $d$ (precisely, $D_3 = \neunzig^{1+\alpha_L} K_L d$).     
\end{theorem}

We provide an explicit discussion of the bias term, the smoothness condition~\ref{cond:smoothness-hoelder}, and the domain parameter $T$ in Remarks~\ref{rem:bias}, \ref{rem:comparison}, and \ref{rem:homogeneity}, respectively.

In contrast to Theorem~\ref{theo:linear-finite-d}, Theorem~\ref{theo:linearization-hoelder} provides non-asymptotic control of the error in approximating $\mathbb L_n$ by $\widebar {\mathbb L}_n$ and also explicitly characterizes the effect of the dimension $d$ on the approximation error. Another salient feature is that $\delta$ only enters the bound logarithmically. This is crucial for considering many estimators simultaneously since the maximum error is still controllable by using union bound type arguments.

The upper bound $d/\sqrt k$ prevents $d$ from being of the order $\sqrt k$ or larger. Much of the recent methodology for high-dimensional extremes does not attempt to estimate the entire joint tail of a large number of variables non-parametrically. For instance, the structure learning approaches in \cite{EngVol22, wan2023graphical, EngLalVol25} are based on a large number of estimators of bivariate tail dependence. To perform statistical inference in such settings, one needs results that hold uniformly in a growing number of low-dimensional estimators rather than one high-dimensional  estimator. Theorem~\ref{theo:linearization-hoelder} readily yields such results as we demonstrate next.

For $I\subset [d]$ with $|I|\ge 2$ and $\bm x_I =(x_i)_{i \in I} \in [0,\infty)^I$, let
\begin{align*}
\Lhat_{n,I}(\bm x_I) &= 
\frac1k \sum_{i=1}^n 
\bm 1\big(\exists j \in I: R_{ij} > n+1-kx_j\big) = \Lhat_n(\bm x_I^0) \\
\Loracle_{n,I}(\bm x_I)
&= 
\frac1k \sum_{i=1}^n \bm 1\Big( \exists j \in I: V_{ij} < \frac{k}nx_j \Big)  
= \Loracle_{n}(\bm x_I^0)
\\ 
\Lpreasy_{n,I}(\bm x_I) 
&= \frac{n}k \Prob\Big(\exists j \in I: V_{ij} < \frac{k}nx_j \Big)
= \Lpreasy_{n}(\bm x_I^0)
\end{align*}
denote the $I$-variate margin of $\Lhat_n$, $\Loracle_{n}$ and $\Lpreasy_{n}$, respectively. Recall that $\bm x_I^0$ has $x_j$ for $j \in I$ and $x_j = 0$ for $j \in [d] \setminus I$. Further, let $\mathbb L_{n,I}=\sqrt k(\Lhat_{n,I} - L_I )$, $\widetilde{ \mathbb L}_{n,I} = \sqrt k(\Loracle_{n,I} - \Lpreasy_I )$ and 
\begin{align} \label{eq:definition-widebar-LnI}
\widebar {\mathbb L}_{n,I}(\bm x_I)   = \widetilde {\mathbb L}_{n,I}(\bm x_I)   - \sum_{j\in I} \pdL{j}_I(\bm x_I) \widetilde {\mathbb L}_{nj}(x_j).
\end{align}
The following result shows that we obtain linearizations that are uniform over collections of margins. It follows from the union bound and Theorem~\ref{theo:linearization-hoelder} applied to each $(A_I,L_I)$.

\begin{corollary} \label{cor:linearization-hoelder}
Let $\mathcal I$ be a collection of index sets $I \subset [d]$ with $|I| \ge 2$, and write $m=\max_{I \in \mathcal I} |I|$. Fix $T\in\N$, let $(A_I)_{I \in \mathcal I}$ be a collection of sets with $A_I \subset [0,T]^I$, and suppose that, for each $I\in \mathcal I$, $\bm X_I$ has STDF $L_I$ such that \ref{cond:smoothness-hoelder} is met for $(A_I, L_I)$, with constants $\kappa_I, K_{I}$ and exponent $\alpha_I$. 
Then, with $\kappa_L = \min_{I \in \mathcal I} \kappa_I, \Kl=\max_{I \in \mathcal I} K_{I}$ and $\alpha_L=\min_{I \in \mathcal I}a_I$, there exist constants $D_1=D_1(m)$ and $ D_2=D_2(m)$ and $D_3=D_3(m,\Kl,\alpha_L)$ such that, for any $n\in\N, k \in [n], \delta \in (0, e^{-1})$ satisfying $\log(m/\delta)\le 2kT/7$, $n/k \ge T$ and $r \le \kappa_L/\neunzig$ with $\neunzig$ from Lemma~\ref{lem:bound-on-order-statistics}, we have
\begin{align*}
\max_{I \in \mathcal I} \sup_{\bm x \in A_I} \big| \mathbb L_{n,I}(\bm x) - \widebar {\mathbb L}_{n,I}(\bm x) \big| 
&\le 
\Big(\max_{I \in \mathcal I} B_{n,k}(L_I; A_I^{\oplus \kappa_L})\Big)
+
\frac{m}{\sqrt{k}}
\\
&\qquad 
+
D_{1} \sqrt{r\log\Big(\frac{TD_{2}}{\delta r}\Big)}  + D_3 r^{\alpha_L}\sqrt{T \log\Big(\frac{1}{\delta}\Big)}
\end{align*}
with probability at least $1-|\mathcal I|(6m + 5)\delta$, with $r$ from \eqref{eq:definition-r} and $B_{n,k}$ from \eqref{eq:bias-A}.
\end{corollary}

To see the power of this result in applications with large $|\Ic|$, let $T=1, \alpha_L=1/2$ and write $p$ for $m|\Ic|$ to lighten the notation. Picking $\delta = (9pk)^{-1}$ (recall that $m \ge 2$, such that $|\mathcal I|(6m+5) \le 9p$) shows that, with probability at least $1-k^{-1}$ 
\begin{align*}
\max_{I \in \mathcal I} \sup_{\bm x \in A_I} \big| \mathbb L_{n,I}(\bm x) - \widebar {\mathbb L}_{n,I}(\bm x) \big| 
&\lesssim 
\Big(\max_{I \in \mathcal I} B_{n,k}(L_I; A_I^{\oplus \kappa_L})\Big)
+ \Big( \frac{\log^3(pk)}{k}\Big)^{1/4},
\end{align*}
where the implicit constant in $\lesssim$ only depends on $m$ and $\Kl$ and where we have used that $r = \sqrt{k^{-1} \log(1/\delta)} \lesssim \sqrt{k^{-1} \log(pk)}$ and $\log(D_2 / \delta r) \lesssim \log (D_2\sqrt k/\delta) \lesssim \log(pk)$. In an asymptotic framework with $p=p_n, k= k_n, n \to \infty$ the upper bound vanishes provided that $\log p = o(k^{1/3})$, i.e. even when the number of estimators we consider grows faster than any polynomial of $k$. An important special case is $\mathcal I = \{I \subset[d]: |I|=2\}$ and $A_I=\{ \bm 1_I\}$, which corresponds to uniform linearizations for all bivariate empirical extremal coefficients $(\theta_I)_{|I|=2}$. 

For the next result, let $E_j = \{ \bm x \in [0,\infty)^d: x_j>0\}$, and for a $d$-variate STDF $L$, write 
\begin{align*}
    G_j^{(1)} &= \big\{ \bm x \in E_j \mid  \partial_{j} L (\bm x) \text{ exists  and is continuous} \big\},
    \\
    G_{j\ell}^{(2)} &=  \big\{ \bm x \in E_j \cap E_\ell \mid  \partial_{j\ell} L (\bm x) \text{ exists  and is continuous} \big\},
    \end{align*}
where $j,\ell \in [d]$.
Moreover, write $\mathfrak B_j^{(1)} = E_j \setminus G_j^{(1)}, \mathfrak B_{j\ell}^{(2)} = (E_j \cap E_\ell) \setminus G_{j\ell}^{(2)}$, and let
    \begin{equation}\label{eq:defbadB}
    \mathfrak B = \Big(\bigcup_{j \in [d]} \mathfrak B_j^{(1)}\Big) \cup  \Big(\bigcup_{j,\ell \in [d]} \mathfrak B_{j\ell}^{(2)}\Big)
    \end{equation}
denote a set of `bad points', where $L$ is not sufficiently regular. The next theorem provides uniform linearizations of $\mathbb L_n(\bm x)$ over collections of points $\bm x$ that are not too close to such 'bad' points. Consider the following smoothness condition on $L$.

\begin{enumerate}
\renewcommand{\theenumi}{(C5)}
\renewcommand{\labelenumi}{\theenumi}
\item \label{cond:smoothness-good}
There exists $\Kl>0$ such that
    \[
    \forall j,\ell\in[d], \forall \bm x \in G_{j\ell}^{(2)}: \quad |\partial_{j\ell} L(\bm x)| \le \Kl (x_j \vee x_\ell)^{-1}.
    \]
\end{enumerate}
A detailed comparison of this condition with condition~\ref{cond:smoothness-hoelder} is given in Remark~\ref{rem:comparison} below. 

\begin{theorem} \label{theo:linearization-good}
Let $L$ be a $d$-variate stable tail dependence function satisfying \ref{cond:smoothness-good}.
Fix $T\in \N$. Then, there exist constants $D_1=D_1(d,\Kl)$ and $ D_2=D_2(d,\Kl)$ such that, for any $n\in\N, k \in [n], \delta \in (0, e^{-1})$ satisfying $\log(d/\delta)\le 2kT/7$ and $n/k \ge 2T$, we have
\begin{align*}
    \sup_{\bm x \in [0,T]^d \setminus (\mathfrak B^{\oplus \neunzig r})} \big| \mathbb L_n(\bm x) - \widebar {\mathbb L}_n(\bm x) \big| 
    &\le 
    B_{n,k}(L; [0,T+\neunzig r]^d) +
    \frac{d}{\sqrt{k}}
    +
    D_{1} \sqrt{r\log\Big(\frac{TD_{2}}{\delta r}\Big)} 
\end{align*}
with probability at least $1- (6d + 5)\delta$, where $\neunzig$ is the universal constant from Lemma~\ref{lem:bound-on-order-statistics} and where $r$ is from \eqref{eq:definition-r}. Here, the constant $D_1$ depends quadratically on $d$,  while $D_2$ depends linearly on $d$. 
\end{theorem}

For many models, the set $\mathfrak B$ of bad points from \eqref{eq:defbadB} is actually empty. The derived linearization then holds uniformly on $[0,T]^d=[0,T]^d \setminus (\emptyset^{\oplus \neunzig r})$. Similar as for Theorem~\ref{theo:linearization-hoelder}, the upper bound $d /\sqrt k$ prevents $d$ from being exponentially large, which can be avoided by treating $m$-dimensional margins only. The following result follows by combining the tail bounds in Theorem~\ref{theo:linearization-good} with the union bound.

\begin{corollary} \label{cor:linearization-good}
Let $\mathcal I$ be a collection of index sets $I \subset [d]$ with $|I| \ge 2$, and write $m=\max_{I \in \mathcal I} |I|$. Suppose that, for each $I\in \mathcal I$, $\bm X_I$ has STDF $L_I$ satisfying  \ref{cond:smoothness-good}; denote the respective set of bad points from \eqref{eq:defbadB} by $\mathfrak B_{I}$. Fix $T\in\N$. Then, with $\Kl=\max_{I \in \mathcal I} K_{L_I}$, there exist constants $D_1=D_1(m,\Kl)$ and $D_2=D_2(m,\Kl)$ such that, for any $n\in\N, k \in [n], \delta \in (0, e^{-1})$ satisfying $\log(m/\delta)\le 2kT/7$ and $n/k\ge 2T$, we have
\begin{align*}
    &\max_{I \in \mathcal I} \sup_{\bm x \in [0,T]^I \setminus (\mathfrak B_I^{\oplus \neunzig r})} \big| \mathbb L_{n,I}(\bm x) - \widebar {\mathbb L}_{n,I}(\bm x) \big| 
    \\ &\hspace{3cm} \le 
    \Big(\max_{I \in \mathcal I} B_{n,k}(L_I; [0,T+\neunzig r]^I)\Big)
    +
    \frac{m}{\sqrt{k}}
    +
    D_{1}  \sqrt{r\log\Big(\frac{TD_{2}}{\delta r}\Big)} 
\end{align*}
with probability at least $1-|\mathcal I|(6m + 5)\delta$, where $\neunzig$ is from Lemma~\ref{lem:bound-on-order-statistics}, where $r$ is from \eqref{eq:definition-r} and  where $B_{n,k}$ is from \eqref{eq:bias-A}.
\end{corollary}

\begin{remark}[On the bias term] \label{rem:bias} 
Most of the literature that deals with inference for multivariate extremes is based on second order conditions which control the speed of convergence in~\eqref{eq:defL} or~\eqref{eq:defR}, see for instance \cite{EinKraSeg12, fougeres2015bias, EngVol22, EngLalVol25} among many others. For many typical models, the speed of convergence in~\eqref{eq:defL} or~\eqref{eq:defR} is a power of $t$. Consequently the bias $k^{-1/2} B_n(\bm x) = \Lpreasy_n(\bm x) - L(\bm x)$ from \eqref{eq:bias} is a power of $k/n$. In some settings, it is possible to establish the exact scaling and an exact asymptotic expansion for the bias, see Section~4 in \cite{fougeres2015bias} for details and further references.
\end{remark}

\begin{remark}[Comparison of \ref{cond:smoothness-hoelder} and \ref{cond:smoothness-good}]\label{rem:comparison}
Conditions \ref{cond:smoothness-hoelder} and \ref{cond:smoothness-good} are different in nature, and neither condition is weaker than the other. 
Condition \ref{cond:smoothness-hoelder} fails on sets of points that are not bounded away from zero, unless $L$ is the STDF corresponding to tail independence.  

Indeed, by homogeneity of $L$, i.e. $L(\lambda \bm x) = \lambda L(\bm x)$ for all $\bm x \in (0,\infty)^d$ and $\lambda>0$, we have $\pdL{j}(\lambda \bm x) = \pdL{j}(\bm x)$ for every $\bm x$ for which $\pdL{j}(\bm x)$ exists.
Suppose now that $A$ from \ref{cond:smoothness-hoelder} is not bounded away from zero. In that case, $A$ contains a null-sequence $\bm x_n$. If $\bm y_1, \bm y_2 \in [0,1]^d$ are arbitrary, then $\max_{i \in [2]}\| \bm x_n - \bm y_i/n\|_\infty \le \|\bm x_n\|_\infty + 1/n \le \kappa_L$ for sufficiently large $n$, and \ref{cond:smoothness-hoelder} then implies that 
\begin{align*}
\forall j \in [d]: \qquad |\pdL{j}(\bm y_1) - \pdL{j}(\bm y_2)| 
&= 
|\pdL{j}(\bm y_1/n) - \pdL{j}(\bm y_2/n)| 
\\&\le 
|\pdL{j}(\bm x_n) - \pdL{j}(\bm y_2/n)| + |\pdL{j}(\bm x_n) - \pdL{j}(\bm y_1/n)|
\\&\le 
2K_L \big(\|\bm x_n\|_\infty +1/n\big)^{\alpha_L} = o(1) \qquad (n \to \infty).
\end{align*}
Hence, $L$ must be linear on $[0,1]^d$, and the only linear STDF is the one corresponding to tail independence, $L(\bm x) = \sum_{j\in[d]} x_j$.

In contrast, condition~\ref{cond:smoothness-good} can often be verified with $\mathfrak B = \emptyset$, see Lemma~\ref{lem:bivariate-bound-pickands} for an example in the bivariate case. 
When $(0, \infty)^d \subset G_{jl}^{(2)}$, Condition \ref{cond:smoothness-good} implies Lipschitz continuity of the partial derivatives when all coordinates are away from zero, which is more restrictive than the H\"older assumption in \ref{cond:smoothness-hoelder}. Condition \ref{cond:smoothness-hoelder} is thus most useful for establishing expansions at individual points $\bm x$ with entries bounded away form zero under minimal assumptions, or on sets of such points. Important applications include the extremal coefficient or tail correlation. 
\end{remark}

We next discuss Condition~\ref{cond:smoothness-good}, which is related to Assumption 2 in \cite{EngLalVol25}. By homogeneity of $L$, that is, $L(\lambda \bm x) = \lambda L(\bm x)$ for all $\bm x \in [0,\infty)^d$ and $\lambda>0$, we have $\pdL{j}(\lambda \bm x) = \pdL{j}(\bm x)$ and $\pdL{j\ell}(\lambda \bm x) = \lambda^{-1} \pdL{j\ell}(\bm x)$ for all $j, \ell\in[d]$. 
It is hence sufficient to check the required bound for $\bm x \in G_{j\ell}^{(2)} \cap [0,1]^d$, as it then automatically holds for all $\bm x \in G_{j\ell}^{(2)}$ with the same constant $\Kl$. 
The following lemma provides a simple sufficient condition for the bivariate case.

\begin{lemma}\label{lem:bivariate-bound-pickands}
Suppose $L$ is a bivariate stable tail dependence function, and let $A(t) = L(1-t,t)$, $t \in [0,1]$, denote the associated Pickands dependence function. If $A$ is twice continuously differentiable on $(0,1)$ and if $A_{\infty} := \sup_{t\in(0,1)} t (1-t) A''(t) < \infty$, then Condition~\ref{cond:smoothness-good} is met for $L$, with $\mathfrak B=\emptyset$ and with $\Kl=A_{\infty}$. 
\end{lemma}

If, for instance, $L$ is the stable tail dependence function of the $d$-variate Hüsler-Reiss-copula with parameter matrix $\Gamma=(\gamma_{j\ell})_{j,\ell\in[d]}$ satisfying $\lambda_0:= \min_{j \ne \ell} \gamma_{j\ell} >0$ (i.e., the bivariate margins are bounded away from perfect dependence; see Example~\ref{ex:HR1}), then each bivariate marginal Pickands dependence function $A_I$ satisfies $A_{I,\infty}  \le C_A$ for some constant $C_A=C_A(\lambda_0)$  \citep[Example 2.6]{BucherPakzad2025}. As a consequence, Corollary~\ref{cor:linearization-good} is applicable with $\mathcal I=\{ I \subset[d]: |I|=2\}$, with $\mathfrak B_I = \emptyset$, and with $\Kl = \max_{|I|=2}A_{I,\infty}\le C_A$.

\begin{remark}[On the domain parameter $T$]
\label{rem:homogeneity} 
It is possible to derive Theorem~\ref{theo:linearization-hoelder} with general $T \in \N$ as stated from the version with $T=1$ only by utilizing certain homogeneity properties. To make explicit the dependence of the estimator $\Lhat_n$ on $k$, we will write $\Lhat_{n,k}$ throughout this remark. For example, for any $\eta > 0$ such that $\eta k$ is an integer, a straightforward calculation yields $\Lhat_{n,k}(\eta\bm x) = \eta \Lhat_{n,k\eta}(\bm x)$.
Together with homogeneity of $L$, this implies 
\[
\mathbb L_{n,k}(\eta \bm x) = \sqrt{k}\big(\Lhat_{n,k}(\eta\bm x) - L(\eta\bm x) \big) = \sqrt{\eta} \sqrt{k\eta} \big(\Lhat_{n,k\eta}(\bm x) - L(\bm x) \big) = \sqrt{\eta} \mathbb L_{n,k\eta}( \bm x).
\]
Similar computations show $ \widetilde{\mathbb L}_{n,k}(\eta \bm x) = \sqrt{\eta} \widetilde{\mathbb L}_{n,k\eta}(\bm x)$, $ \widebar{\mathbb L}_{n,k}(\eta \bm x) = \sqrt{\eta} \widebar{\mathbb L}_{n,k\eta}(\bm x)$. We still choose to state the version for general $T$ directly since the full conversion requires some tedious work. A similar comment applies to some of the other results in this section.
\end{remark}

\section{Gaussian approximations and bootstrap approximations}
\label{sec:clts}

Let $\mathcal I$ be a finite collection of index sets $I \subset [d]$ with $|I| \ge 2$, let $m = \max_{I \in \mathcal I} |I|$. For each $I \in \mathcal I$, assume that $L_I$ exists, let $A_I = \{\bm x_{I,1}, \dots, \bm x_{I,p_I}\}$ be a finite set of vectors in $(0,1]^{I}$, and let $p = \sum_{I \in \mathcal I}  p_I \ge |\mathcal I|$. Note that we  restrict ourselves to $T=1$, which is not restrictive by homogeneity of STDFs. Our goal is to derive Gaussian approximations for the $p$-dimensional random vector
\begin{align} \label{eq:Sn-clt}
\bm S_n = (\mathbb L_{n,I}(\bm x_{I, \ell}))_{I \in \mathcal I, \ell \in [p_I]}. 
\end{align}
Writing $\bm y_{I, \ell} = (\bm x_{I, \ell}, \bm 0_{I^c}) \in [0,1]^d$ and $A = \bigcup_{I \in \mathcal I} \{ \bm y_{I, \ell}: j \in [p_I]\}$, we can write
\[
\bm S_n = (\mathbb L_{n}(\bm y))_{y \in A} \in \R^p. 
\]
Such high-dimensional vectors arise naturally, for instance, when considering the extremal coefficient matrix with elements $\theta_I=L_I(\bm 1_I)$ for $I \subset [d]$ with $|I|=2$. The rescaled estimation error of the empirical counterpart is $\sqrt k(\hat \theta_I - \theta_I) = \mathbb L_{n,I}(\bm 1_I)$. Collecting these errors in a vector corresponds to considering $\mathcal I = \{ I \subset [d]: |I|=2\}$ and $A_I =\{\bm 1_I\}$, with $m=2$ and $p = d(d-1)/2$.

Let
\begin{align*} 
\bm G_n \sim \Nc_{p}(\bm 0, \Sigma_n), \quad \text{ where } \Sigma_n = \Var(\bm T_n) \text{ with }
\bm T_n  = (\widebar \Lb_{n,I}(\bm x_{I, \ell}))_{I \in \mathcal I, \ell \in [p_I]} \in \R^p. 
\end{align*}
and with $\widebar \Lb_{n,I}$ from \eqref{eq:definition-widebar-LnI}.
Specific formulas for the entries of $\Sigma_n$ are given in \eqref{eq:covariance-formula}. Write $\sigma_{n,q}^2$ for $q$th diagonal element of $\Sigma_n$.
For random vectors $\bm S$ and $\bm T$ of the same dimension $p\in\N$, let
\[
d_K(\bm S, \bm T) 
= 
\sup_{\bm x \in \R^{p}} \big| \Prob(\bm S \le \bm x) - \Prob(\bm T \le \bm x) \big|
\]
denote the Kolmogorov distance between $\bm S$ and $\bm T$. The following result provides a bound on $d_K(\bm S_n, \bm G_n)$ under a condition as in Corollary~\ref{cor:linearization-hoelder}; adaptations to the conditions of Corollary~\ref{cor:linearization-good} follow along similar lines and are omitted for the sake of brevity. The obtained upper bound has similar features as the bounds in classical high-dimensional Gaussian approximation results in \cite{CheCheKat23}. However, there is an additional bias term which is due to the fact that we do not directly observe data from $L$ but rather work with domain of attraction conditions. Note also that $n$ in the upper bound in \cite{CheCheKat23} is replaced by $k$ in our setting. Intuitively, this is because we effectively only use $k$ observations to compute $\Lhat$.

\begin{theorem} \label{theo:clt}
Let $\mathcal I$ and $(A_I)_{I \in \mathcal I}$ be as described in the beginning of Section~\ref{sec:clts} and suppose that the STDF $L_I$ of $\bm X_I$ exists for every $I \in \mathcal I$.  Assume that there exist $\kappa_L, \Kl \in(0,\infty)$ and $\alpha_L \in (1/2,1]$ such that 
\begin{align*}
\forall I \in \mathcal I, & \forall j \in I, \forall \bm x_I \in A_I, \forall \bm y_I \in [0,\infty)^I  \text{ with } \|\bm x_I - \bm y_I\|_\infty \le \kappa_L: \\
&\pdL{j}_I(\bm x_I), \pdL{j}_I(\bm y_I) \text{ exist and satisfy }
|\pdL{j}_I(\bm x_I)-\pdL{j}_I(\bm y_I)| \le \Kl \|\bm x_I - \bm y_I\|_\infty^{\alpha_L}.
\end{align*}
Moreover, assume that $m|\mathcal I| \ge 3, n \ge 2,p \ge 2$ 
and 
\begin{compactenum}[(i)]
\item $\sigma_{\min}^2 := \min_{q \in [p]} \sigma_{n,q}^2>0$. 
\item $\log(m^2 |\mathcal I|k^{1/4} ) \le 2k/7$.
\item 
$\log(m|\Ic|k^{1/4}) \le \kappa_L^2 k / \neunzig^2$ with $\neunzig$ from Lemma~\ref{lem:bound-on-order-statistics}. 
\end{compactenum}
Then there exists a constant $c = c(\sigma_{\min}^2, m, \Kl, \alpha_L) \ge 1$ such that 
\[
d_K(\bm S_n, \bm G_n) 
\le  
c \Big[ \sqrt{\log p}  \Big(\max_{I \in \mathcal I} B_{n,k}(L_I; A_I^{\oplus \kappa_L}) \Big) 
+
\Big( \frac{\log^5(pn)}{k} \Big)^{1/4} \Big].
\]
\end{theorem}

We briefly discuss the assumptions and the result. First, the smoothness condition on the collection $(L_I)_I$ essentially requires \ref{cond:smoothness-hoelder} to hold for each pair $(A_I, L_I)$, see also Corollary~\ref{cor:linearization-hoelder}. The assumptions $m|\mathcal I| \ge 3, n \ge 2,p \ge 2$ are very mild; they can be omitted at the cost of more technical arguments within the proof. The variance condition in (i) is required for high-dimensional CLTs as in \cite{Che22}; as shown in Remark~\ref{rem:variance-bound} below, it is a very mild and natural requirement if $m=2$. Finally, the conditions in (ii) and (iii) can best be interpreted in an asymptotic (triangular array) framework where $\mathcal I=\mathcal I_n$ and $k=k_n$ is allowed to depend on $n$: both conditions are satisfied for sufficiently large $n$ if $\log(|\mathcal I_n|) = o(k_n)$. In such an asymptotic framework, the upper bound on the Kolmogorov distance converges to zero if $\log^5(p_n) = o(k_n)$ and if the (uniform) bias term is of smaller order that $\sqrt{\log(p_n)}$. Finally, note that the factor $\sqrt{\log p}$ in front of the bias term is natural in view of Lemma 1 in \cite{CheCheKat23}.

\begin{remark}[Other possible versions] We note that the proof of Theorem~\ref{theo:clt} utilizes a particular version of a high-dimensional Gaussian approximation result from \cite{Che22}. Specifically, the proof proceeds by applying Theorem~\ref{theo:clt-new} to the collection $\bm T_n  = (\widebar \Lb_{n,I}(\bm x_{I, \ell}))_{I \in \mathcal I, \ell \in [p_I]}$ and controlling the error in approximating $\bm S_n$ by $\bm T_n$.
Depending on the assumptions, other versions of high-dimensional Gaussian approximation results can be applied. Here, we briefly mention two possible versions without going into details. First, \cite{lopes2022central, chernozhukov2023nearly} have shown that even better rates for the error are possible if the covariance matrix of the vector $\bm T_n$ has smallest eigenvalue bounded away form zero. In that case a rate of $n^{-1/2}$ up to poly-logarithmic factors can be achieved. Second, Theorem~\ref{theo:clt} requires a lower bound on the variance. This is because the smallest variance appears in both, Theorem~\ref{theo:clt-new} and in the proof via Theorem~\ref{theo:nazarov}. This restricts the set of admissible values for $\bm x$ away from the origin and also rules out asymptotic independence. At the cost of a slower rate and for a restricted Kolmogorov distance, it is possible to drop the minimum variance assumption by utilizing Lemma~7 and Theorem~3 in \cite{decker2025simultaneous}.
\end{remark}

\begin{remark}[On supremum statistics]
\label{rem:supremum-statistics}
The result in Theorem~\ref{theo:clt} is sufficiently strong to cover distributional approximations for supremum-statistics. It is instructive to study the bivariate case first, and more specifically,  we are then interested in 
approximations for the cdf of $\sup_{\bm x \in B}  \Lb_n(\bm x) $ with $B \subset [0,1]^2$. In view of the fact that $\Lhat_n$ is a piecewise constant function that is constant on intervals of the form $[\ell/k, (\ell+1)/k) \times [\ell'/k, (\ell'+1)/k)$, we have 
$
\sup_{\bm x \in B} \Lb_n(\bm x)
=
\max_{ \bm x \in B \cap G} \Lb_n(\bm x) ,
$
where $G$ contains all vectors in $[0,1]^2$ of the form $(\ell/k, \ell'/k)$ with $\ell, \ell' \in \N_0$. Note that $|G| \le (k+1)^2$. As a consequence, 
\[
\Prob\Big( \sup_{\bm x \in B}  \Lb_n(\bm x)  \le t \Big) 
=
\Prob\Big(\max_{ \bm x \in B \cap G}  \Lb_n(\bm x) \le t \Big)
=
\Prob\Big( (\Lb_n(\bm x))_{\bm x \in  B \cap G} \le \bm t \Big), 
\]
where $\bm t=(t,\dots, t)\in \R^{ B \cap G}$. We can hence apply Theorem~\ref{theo:clt} with $p=| B \cap G| \le (k+1)^2$, and the approach could easily be extended to the multivariate case, which each margin under consideration contribution at most $(k+1)^m$ to $p$. 
\end{remark}

\begin{remark}[On the variance condition] 
\label{rem:variance-bound}
A generic diagonal element $\sigma_{n,q}^2$ of $\Sigma_n$ can be written as 
$
\sigma_{n,I}^2(\bm x_I) = \Exp[\widebar \Lb_{n,I}^2(\bm x_I)]
$
for certain $I \in \mathcal I$ and $\bm x_I \in A_I$. 
A straightforward calculation, carried out in  Section~\ref{sec:proofs-clts}, shows that, if 
$I$ and $\bm x_I$ are fixed and if
$k=k_n$ satisfies $k_n =o(n)$ as $n \to \infty$,
\begin{align*}
\sigma_I^2 (\bm x_I) 
&=  \lim_{n \to \infty}\sigma_{n,I}^2(\bm x_I)
=
- L_I(\bm x_I) +  (\nabla L_I(\bm x_I))^\top \mathcal R_I(\bm x_I) (\nabla L_I(\bm x_I)) ,
\end{align*}
where $\nabla L_I(\bm x_I)= (\pdL{j}_I(\bm x_I))_{j \in I} \in \R^I$ and where $\mathcal R_I(\bm x_I) = (R_{\{j,\ell\}}(x_{I,j}, x_{I,\ell}))_{j,\ell \in I}$ is a $|I| \times |I|$ matrix, with diagonal entries $R_{\{j,j\}}(x_{I,j}, x_{I,j}) = x_{I,j}$ and with $R_{\{j, \ell\}}$ the tail copula of the bivariate sub-vector $X_{\{j,\ell\}}$ of $\bm X_I$. The variance condition in (i) of Theorem~\ref{theo:clt} would be satisfied for sufficiently large $n$ (more precisely, for sufficiently small $k/n$) if $\sigma_I^2 (\bm x_I)$ is bounded away from zero, uniformly in $I$ and $\bm x_I$. We show in Section~\ref{sec:proofs-clts} that, in the case $|I|=2$, $\sigma_I^2 (\bm x_I)$ is non-zero if and only if $R_I \notin\{ R_{{\text{ind}}}, R_{\text{pd}}\}$, where $R_{{\text{ind}}} \equiv 0$ and $R_{\text{pd}}(x,y) = x \wedge y$ correspond to tail independence and perfect tail dependence, respectively. As a consequence, (i) would be satisfied for sufficiently large $n$ if all $R_I$ are bounded away from these two extreme cases.
\end{remark}

Next, we derive bootstrap approximations, following the multiplier approach from \cite{BucDet13}, whose validity will be transferred to the high-dimensional setting by combining arguments from \cite{CheCheKat23} with a careful analysis of the impact of estimating the partial derivatives $\partial_j L$ in the bootstrap procedure. The presence of the latter means that the high-di\-men\-sional bootstrap result in Theorem 3 of \cite{CheCheKat23} is not directly applicable and additional arguments are needed. The approach requires suitable estimates of the partial derivatives of $L_I$, for which one may follow a simple finite-differencing technique: for
$\bm x_I \in (0,\infty)^I$, $j \in I$, and a bandwidth parameter $h>0$ such that $0< h < x_j$, define 
\[
\pdLhat{j}_I(\bm x_I) = \pdLhat{j}{}_{n,h,I}(\bm x) =  
 \min\Big\{ \frac{\Lhat_{n,I}(\bm x + h \bm e_j) - \Lhat_{n,I}(\bm x - h \bm e_j)}{2h}, 1\Big\}.
\]
Next, note that
\[
\widebar{\mathbb L}_{n,I}(\bm x_I) = \sum_{i=1}^n Y_{i,I}(\bm x_I),
\]
where
\begin{align}
\label{eq:YiI}
Y_{i,I}(\bm x_I) \nonumber
&=
\frac1{\sqrt k} \Big[ \bm 1(\exists j \in I: V_{ij} < kx_j/n) - \Prob(\exists j \in I: V_{ij} < kx_j/n) 
     \\ & \hspace{4cm}- \sum_{j\in I} \pdL{j}_I(\bm x_I) \big\{\bm 1(V_{ij} < kx_j/n) -  kx_j/n \big\} \Big].
\end{align}
Define observable counterparts of $Y_{i,I}(\bm x_I)$ by
\begin{align} \label{eq:hatYI}
    \widehat Y_{i,I} (\bm x_I)
&= \nonumber
\frac1{\sqrt k} \Big[ \bm 1(\exists j \in I: \hat V_{ij} < kx_j/n) - (k/n) \Lhat_{n,I}(\bm x_I)
     \\ & \hspace{4cm}- \sum_{j\in I} \pdLhat{j}_{I}(\bm x_I) \big\{\bm 1(\hat V_{ij} < kx_j/n) - kx_j/n \big\} \Big],
\end{align}
where $\hat {\bm V}_i=(\hat V_{i1}, \dots, \hat V_{id})^\top$ has coordinates $\hat V_{ij} = 1 +n^{-1}-n^{-1}R_{ij}$. For $e_1, e_2, \dots$ iid standard normal and independent of the observations $\bm X_i$, we propose to use
\begin{align} \label{eq:Sn*-boot}
    \bm S_n^* = (\widebar{\mathbb L}^*_{n,I}(\bm x_{I,\ell}))_{I \in \mathcal I, \ell \in [p_I]},
    \qquad
    \widebar{\mathbb L}^*_{n,I}(\bm x_I) 
    =
    \sum_{i=1}^n e_i \widehat Y_{i,I} (\bm x_I)
\end{align}
as a bootstrap approximation for $\bm S_n$ from \eqref{eq:Sn-clt}. The following result provides high-probability bounds for
\[
d_K( \mathcal L(\bm S_n^* \mid \mathrm{data}), \bm G_n)
\]
under a suitable H\"older smoothness assumption on each $L_I$. Unlike for the CLT from Theorem~\ref{theo:clt}, we restrict attention to the case where the H\"older exponent is 1; extensions to other exponents or smoothness assumptions as in Corollary~\ref{cor:linearization-good} are possible but are omitted for the sake of a clear exposition.

\begin{theorem}\label{theo:niceboot-new}
Let $\mathcal I$ and $(A_I)_{I \in \mathcal I}$ be as described in the beginning of Section~\ref{sec:clts} and suppose that the STDF $L_I$ of $\bm X_I$ exists for every $I \in \mathcal I$. Assume that there exist $\kappa_L, \Kl \in(0,\infty)$ such that 
\begin{align*}
\forall I \in \mathcal I,&\forall j \in I, \forall \bm x_I \in  A_I^{\oplus \min(1,\kappa_L/2)}, \forall \bm y_I \in [0,\infty)^I  \text{ with } \|\bm x_I - \bm y_I\|_\infty \le \kappa_L: \\
&\pdL{j}_I(\bm x_I), \pdL{j}_I(\bm y_I) \text{ exist and satisfy }
|\pdL{j}_I(\bm x_I)-\pdL{j}_I(\bm y_I)| \le \Kl \|\bm x_I - \bm y_I\|_\infty.
\end{align*}
Assume the conditions (i)--(iii) of Theorem~\ref{theo:clt} are met 
with the condition $\log(m|\mathcal I| k^{1/4}) \le \kappa_L^2k/ \neunzig^2$ replaced by $\log(m|\mathcal I| k^{1/4}) \le \kappa_L^2k/ (8\neunzig^2)$, and with $n/k\ge 2$.
Let $0<c_h< c_h'< \infty$ be constants, and assume that the bandwidth $h<(\min_{I \in \mathcal I} \min_{\bm x_I \in A_I}  \min_{j \in I}x_{I,j}) \wedge (\kappa_L/2)$ satisfies
\[
c_h\Big( \frac{\log(p+k)}k\Big)^{1/2} 
\le h \le 
c_h' \Big( \frac{\log(p+k)}k\Big)^{1/4}.
\]
Then, there exist constants $c_i = c_i(m,\Kl,\sigma_{\mathrm{min}},c_h,c_h'), i = 1,2$ such that, with probability at least $1-c_1\delta_n$
\[
d_K( \mathcal L(\bm S_n^* \mid \mathrm{data}), \bm G_n) \le 
c_2 \Big\{\delta_n + \sqrt{\log(p+k)}  B_{n,k}(L_I ;  A_I^{\oplus\kappa_L}) \Big\},
\]
where $\delta_n = [k^{-1}\log^5(pn)]^{1/4}$.
\end{theorem}

We briefly comment on the conditions and the result. The smoothness condition is a slightly stronger version of the one imposed for Theorem~\ref{theo:clt}: first, we restrict attention  to $\alpha_L=1$ for simplicity, and second, the third $\forall$-quantor requires $\bm x_I$ to be from a small extension of $A_I$ rather than from $A_I$ only. This extension is needed in the proofs when passing from estimated partial derivatives to true unknown partial derivatives. The strengthening of condition (iii) from Theorem~\ref{theo:clt} is mild. Finally, the condition on the bandwidth is mild in the sense that the same approximation bound is obtained for a large range of bandwidth choices. 
The obtained rate is almost the same as in Theorem~\ref{theo:clt}, with a factor $\sqrt{\log(p+k)}$ instead of $\sqrt{\log(p)}$ in front of the bias term; in particular, the same `rate' is obtained in the (high-dimensional) case where $k \lesssim p$.

\section{Applications}
\label{sec:applications}

\subsection{Parametric M-estimators for tail dependence}
\label{sec:Mest}

As an application of the uniform linearizations established in Section~\ref{sec:linearizations}, we derive corresponding linearizations for moment estimators based on integrals of $\widehat L_n$. We first consider estimators constructed from the full $d$-variate function $\widehat L_n$ and subsequently turn to a collection of M-estimators based on lower-dimensional margins of $\widehat L_n$. In the latter setting, we additionally establish high-dimensional central limit theorems.

In defining the estimators, we follow the setup in \cite{EinKraSeg12}. Let $\{L(\cdot; \theta) \colon \theta \in \Theta\}$ be a parametric family of STDFs, with a parameter space \(\Theta \subseteq \mathbb{R}^s\). Next, let 
\[
Q_n(\theta) := \Big\|\int_{[0,1]^d} \bg(\bx) \big(L(\bx;\theta)-\Lhat_n(\bx)\big) \diff \mu(\bx)\Big\|_2
\]
for a (known) measure $\mu$ on $[0,1]^d$ and a (known) function $\bg: [0,1]^d \to \R^q$ with \(q \in \N_{\ge s}\) such that
\begin{equation}\label{eq:defCg}
C_g \coloneqq \ \int_{[0,1]^d} \|\bg(\bx)\|_2 \,\diff \mu(\bx)<\infty. 
\end{equation}
For the subsequent analysis, we also define the population version of $Q_n$ which is given by 
\[
Q_L(\theta) := \Big\|\int_{[0,1]^d} \bg(\bx)\big( L(\bx;\theta)-L(\bx) \big) \diff \mu(\bx) \Big\|_2.
\]
\cite{EinKraSeg12} assume that $\theta \mapsto \int \bg L(\cdot;\theta)\diff \mu$ is a homeomorphism between $\Theta$ and its codomain and show that, under certain conditions, $Q_{n}$ has a unique minimizer in $\Theta$ with probability going to one when the sample sizes grows to infinity. We will take a different route and instead prove results for any sufficiently good \emph{approximate} minimizer of $Q_n$, i.e. any $\hat \theta_n$ that satisfies
\begin{equation}
Q_n(\mestimator)-\inf_{\theta \in \Theta} Q_n(\theta) < \nearmin
\end{equation}
for $\nearmin$ 'small' in a sense made precise below. This allows us to give statistical guarantees for estimators that are computed by numerical optimization, which is a common scenario in practice. We will work under the following assumptions.

\begin{assumption}\label{assumption: smoothness}
There exist constants $\thetainterior>0, \gammaholder \in (0,1]$ and $\Cholder>0$ such that the tuple $(L, \{L(\cdot; \theta): \theta \in \Theta\}, \bm g, \mu)$ satisfies the following:
\begin{compactenum}[(i)]
    \item The function $\theta \mapsto Q_L(\theta)$ has a unique minimum in $\theta_0$. 
    \item The closed ball $\overline B_\thetainterior(\theta_0) := \{\theta \colon \| \theta - \theta_0 \|_2\leq \thetainterior\}$ is contained in $\Theta$.
    \item  The function $\bm \varphi \colon \Theta \subset \R^s \to \R^q$ defined by $\bm \varphi(\theta) = \int_{[0,1]^d} \bg(\bx)L(\bx;\theta)\diff\mu(\bx)$ is twice differentiable on \(B_\thetainterior(\theta_0) := \{\theta \colon \| \theta - \theta_0 \|_2< \thetainterior\} \).
    \item All mixed second order partial derivatives of $\bm \varphi$ are uniformly H\"older continuous at $\theta_0$ in the following sense: 
    \[
    \forall \theta \in B_{\thetainterior}(\theta_0): \quad 
    \max_{j,\ell \in [s], p \in [q]}\abs{\partial_{j\ell} \varphi_p(\theta)-\partial_{j\ell}\varphi_p(\theta_0)} \leq \Cholder \norm{\theta -\theta_0}_2^{\gammaholder },
    \]
    where, for $\theta' \in B_\thetainterior(\theta_0)$, $j,\ell \in [s]$ and $p \in [q]$,
    \begin{equation*}
        \partial_j \varphi_p(\theta') = \frac{\partial}{\partial \theta_j} \varphi_p(\theta)\bigg|_{\theta = \theta'} \quad \text{ and }\quad   \partial_{j\ell} \varphi_p(\theta') = \frac{\partial^2}{\partial \theta_j \partial \theta_\ell} \varphi_p(\theta)\bigg|_{\theta = \theta'}.
    \end{equation*}
    \item The function $\theta \mapsto d_Q(\theta) = Q_L^2(\theta) - Q_L^2(\theta_0)$ 
    (which has a unique minimum in $\theta_0$ by (i) and which is twice differentiable on $B_\thetainterior(\theta_0)$ by (iii)) 
    is bounded away from zero on $\Theta \setminus B _{\thetainterior}(\theta_0)$, has invertible Hessian $V_{\theta_0} \in \R^{s \times s}$ at $\theta_0$ and satisfies
    \[
    \forall \theta \in B_{\thetainterior}(\theta_0): \quad d_Q(\theta) \ge \frac{\lambda_{\min}(V_{\theta_0})}4 \| \theta- \theta_0\|^2_2.
    \]
\end{compactenum}
\end{assumption}

Parts (ii)–(iv) are standard smoothness assumptions. Under sufficient smoothness, Parts (i) and (v) are essentially equivalent to requiring that $\theta_0$ be the unique well-separated minimizer of $Q_L$ (and thus of $d_Q$); this follows from a standard Taylor expansion of $d_Q$ around $\theta_0$. 
In the present assumption, this property is formulated in a more quantitative manner to facilitate later arguments. Finally, note that we do not assume that $Q_L(\theta_0)=0$. Consequently, the subsequent result also applies to misspecified models, that is, to situations where $L \notin \{L(\cdot;\theta): \theta \in \Theta\}$.

Before providing a linear representation for $\mestimator - \theta_0$, we need to introduce some additional notation. 
Denote by $J_\theta \in \R^{q \times s}$ the Jacobian matrix of $\bm \varphi$ evaluated at $\theta$. 
Let \(V_{n, \theta}\) denote the Hessian matrix of the map \(\theta \mapsto Q_n^2(\theta)\) evaluated at \(\theta\). Let $\partial_j \widetilde L(\bx)$ denote the partial derivative of $L$ where it exists and the right-side directional partial derivative with respect to $x_j$ otherwise; note that the right-hand partial derivative always exists by convexity of $L$. For $i \in [n]$, define
\begin{equation}\label{eq:defZjn}
Z_{i,n} := 2V_{\theta_0}^{-1} J_{\theta_0}^\top \int_{[0,1]^d} \Big\{\bm 1\Big( \exists j \in [d]: V_{ij} < \frac{k}nx_j \Big) - \sum_{j=1}^d\partial_j \widetilde L(\bx) \bm 1\Big( V_{ij} < \frac{k}nx_j \Big) \Big\} \bg(\bx) \, \diff \mu(\bx) 
\end{equation}
and note that $Z_{1,n}, \dots, Z_{n,n}$ are iid. 
Finally, note that Assumption~\ref{assumption: smoothness} implies that the constants
\begin{align} \label{eq:C_partial}
C_\partial &:= \max_{j \in [s], p \in [q]}\sup_{\theta \in B_\thetainterior(\theta_0)} |\partial_j \varphi_p(\theta)| , 
\qquad
C_{\partial^2} := \max_{j,\ell \in [s], p \in [q]} \sup_{\theta \in B_\thetainterior(\theta_0)} |\partial_{j\ell} \varphi_p(\theta)|, 
\end{align}
are finite, while the following two constants are positive:
\begin{align} \label{eq:lambda-min}
C_{V} :=  \lambda_{\min}(V_{\theta_0}), 
\qquad
C_Q := \inf_{\theta \in \Theta \setminus B_{\thetainterior}(\theta_0)} d_Q(\theta).
\end{align}

\begin{theorem}\label{theo:Mestlin3}
Let $L$ be a $d$-variate STDF satisfying~\ref{cond:smoothness-good}, and assume that the tuple $(L, \{L(\cdot; \theta): \theta \in \Theta\}, \bg, \mu)$ satisfies Assumption~\ref{assumption: smoothness}.
Then, there exist constants  $D_1, D_2>0$ only depending on $d$ and $\Kl$ (from Theorem~\ref{theo:linearization-good}) and $\tilde{C}_\beta, \tilde{C}_\nearmin \in (0,1], \tilde C_{r1}, \tilde C_{r2}>0$ only depending on $d,s,q$, the constant $C_g$ from \eqref{eq:defCg}, the three parameters $\thetainterior,\Cholder,\gammaholder$ from Assumption~\ref{assumption: smoothness} and the four constants defined in \eqref{eq:C_partial} and \eqref{eq:lambda-min}
such that, for any $n \in \N, k \in [n], \delta \in (0,e^{-1})$ satisfying $\log(d/\delta)\le 2k/7$, 
$\neunzig r \le 1$ and $n/k \ge 2$ with $\neunzig$ the universal constant from Lemma~\ref{lem:bound-on-order-statistics}, the following holds with probability at least $1- 7(d+1) \delta$: 

If $\nearmin \in (0, \tilde C_\nearmin)$ and if 
\begin{align*}
\zeta_{n,1} := k^{-1/2} \sup_{\bm  x \in [0,2]^d} |B_n(\bm x)|
+(\neunzig+188\sqrt 2/3)\cdot d r  \le \tilde{C}_\beta
\end{align*}
with $B_n$ from \eqref{eq:bias} and with $r = \sqrt{ k^{-1}\log(1/\delta)}$ as in \eqref{eq:definition-r}, then
\begin{equation*}
\sqrt{k} \big( \mestimator - \theta_0 \big)  = \frac{1}{\sqrt{k}} \sum_{i=1}^n \big( Z_{i,n} - \Eb[Z_{i,n}] \big)
+ \bm r_{n,1} + \bm r_{n,2}
\end{equation*}
where $\|\bm r_{n,1} \|_2^2 \le \tilde C_{r1} k (\zeta_{n,1}^{2+\gamma_h} + \nearmin)$ and
\begin{align*}
\|\bm r_{n,2}\|_2 &\le \tilde C_{r2}\Big(\sup_{\bm  x \in [0,2]^d} |B_n(\bm x)| +
\frac{d}{\sqrt{k}}
+
D_{1} \sqrt{r\log\Big(\frac{D_{2}}{\delta r}\Big)} 
\\& \hspace{5cm} +  \sqrt k \zeta_{n,1}\int_{\mathfrak B^{\oplus \neunzig r}} \|\bg(\bx)\|_2\, \diff \mu(\bx)\Big).
\end{align*}
\end{theorem}

Theorem~\ref{theo:Mestlin3} can be combined with the central limit theorem to yield an alternative proof of Theorem~4.2 in \cite{EinKraSeg12} for approximate, rather than exact, M-estimators, albeit under stronger smoothness assumptions on $L$. In contrast to \cite{EinKraSeg12}, which establishes only weak convergence, our result additionally provides non-asymptotic remainder bounds with explicit rates. We further note that the part of the remainder term involving the integral over $\mathfrak B^{\oplus \neunzig r}$ can be small even in irregular models. For instance, for the factor models from Example~\ref{ex:fac1}, a straightforward computation shows that the Lebesgue measure of $\mathfrak B^{\oplus \neunzig r} \cap [0,1]^d$ is bounded by a constant multiple of $r$. For functions $g$ with uniformly bounded norm and for $\mu$ corresponding to Lebesgue measure, we thus have
\[
\int_{\mathfrak B^{\oplus \neunzig r}} \|\bg(\bx)\|_2\, \diff \mu(\bx) \lesssim r.
\]

Similarly to Corollaries~\ref{cor:linearization-hoelder} and~\ref{cor:linearization-good}, Theorem~\ref{theo:Mestlin3} can be combined with the union bound to obtain uniform linearizations for collections of M-estimators based on lower-dimensional margins, where the number of estimators may grow at a rate of the form $\exp(k^a)$ for sufficiently small $a$. Such settings naturally arise when a multivariate tail dependence model is characterized through parametric bivariate dependencies only, as is the case for the Hüsler--Reiss model from Example~\ref{ex:HR1}.
Moreover, in the same framework, the result also provides the basis for a high-dimensional Gaussian approximation analogous to the results derived in Section~\ref{sec:clts}. We conclude this section by establishing such a result.

Specifically, let $\mathcal I$ be a collection of index sets $I \subset [d]$ with $|I| \ge 2$, and write $m=\max_{I \in \mathcal I} |I|$. For each \(I \in \mathcal{I}\), let \(\{L_I(\cdot; \theta^I) \colon \theta^I \in \Theta^I\}\) be a parametric family of STDFs, with a parameter space $\Theta^I \subseteq \R^{s^I}$; note that $\Theta^I$ has a different meaning than the notation $A^I$ for $A \subset [-\infty, \infty]$ introduced in Section~\ref{sec:notation}.
Let
\begin{align*}
    Q_n^I(\theta^I) \coloneq \Big\|\int_{[0,1]^{I}} \bm g^I(\bx_I)\big(L_I(\bx_I; \theta^I)-\hat{L}_{n,I}(\bx_I)\big) \,\diff\mu^I(\bx_I)\Big\|_2
\end{align*}
for known measures $\mu^I$ on $[0,1]^I$ and known functions \(\bm g^I \colon [0,1]^I \to \R^{q^I}\) with \(q^I \in \N_{\geq s^I}\). Suppose that \(C_g^I := \int_{[0,1]^{I}} \|\bg^I\|_2 \,\diff\mu^I < \infty\) for any \(I \in \mathcal{I}\). Likewise, define
\begin{align*}
    Q_{L_I}^I(\theta^I) \coloneqq \Big\|\int_{[0,1]^{I}} \bm g^I(\bx_I)\big( L_I(\bx_I;\theta^I)-L_I(\bx_I) \big) \, \diff \mu^I(\bx_I)\Big\|_2.
\end{align*}
Let \(\hat{\theta}_n^I\) be an approximate minimizer of \(Q_n^I\) in the sense that for some \(\eta>0\),
\begin{align*}
    Q_n^I(\hat{\theta}_n^I)-\inf_{\theta^I \in \Theta^I} Q_n^I(\theta^I)< \eta.
\end{align*}

\begin{assumption}\label{ass:MEstUnif}
There exist constants $\thetainterior>0$ and $\Cholder>0$ such that, for each $I \in \mathcal I$ the tuple $(L_I, \{L_I(\cdot; \theta^I): \theta^I \in \Theta^I\}, \bm g^I, \mu^I)$ satisfies the conditions (i) - (v) from Assumption~\ref{assumption: smoothness} with $\gammaholder = 1$ and with $\theta_0 = \theta_0^I$,  $\bm \varphi = \bm \varphi^I$, $d_Q=d_Q^I$ and $V_{\theta_0} = V_{\theta_0^I}$.
\end{assumption}

Write $C_\partial^{I}, C_\partial^{I}, C_{V}^I$ and $C_Q^I$ for the constants in \eqref{eq:C_partial} and  \eqref{eq:lambda-min} when applied for $(L^I, \{L^I(\cdot; \theta^I): \theta^I \in \Theta^I\}, \bm g^I, \mu^I)$, and let
\begin{align}
\label{eq:constants-Ic}
C_\partial^{\Ic} = \max_{I \in \Ic} C_\partial^{I},
\quad
C_{\partial^2}^{\Ic} = \max_{I \in \Ic} C_{\partial^2}^{I},
\quad
C_g^{\Ic} =  \max_{I \in \Ic} C_g^{I},
\quad
C_{V}^{\Ic} =  \min_{I \in \Ic} C_V^I,
\quad
C_{Q}^{\Ic} =  \min_{I \in \Ic} C_Q^I.
\end{align}
Let \(J_{\theta^I} \in \R^{q^I \times s^I}\) denote the Jacobian matrix of \(\bm{\varphi}^I\) evaluated at \(\theta^I\).
Let $\partial_j \widetilde L_I(\bx_I)$ denote the partial derivative of $L_I$ where it exists and the right-side directional partial derivative with respect to $x_j$ otherwise; note that the right-hand partial derivative always exists by convexity of $L_I$. For each $i \in [n]$ and \(I \in \mathcal{I}\), let 
\begin{align}
    \bm A_{i,n}^I 
    &= \label{eq:def-Ani}
    \int_{[0,1]^{I}}
    \Big\{\mathbf{1}\Big(\exists j\in I:\; V_{ij}<\frac{k}{n}x_j\Big)
    - \sum_{j\in I} \partial_j \widetilde L_I(\bx_I)\,
    \mathbf{1}\Big(V_{ij}<\frac{k}{n}x_j\Big)\Big\}
    \bg^I(\bx_I)\,\diff \mu^I(\bx_I)
    \\
    \bm Z_{i,n}^I 
    &= 
    2 \,\big(V_{\theta_0^I}\big)^{-1} J_{\theta_0^I}^{\!\top} \bm A_{i,n}^I
\end{align}
For \(t \in [s^I]\), let \(Z_{i,n}^{I,t}\) be the \(t\)-th component of the random vector \(\bm Z_{i,n}^I\). Define
\[
B_n^I(\bm x) := \sqrt{k}\{\Lpreasy_{n,I}(\bm x) - L_I(\bm x)\}
\]
and, with \(s= \sum_{I \in \mathcal{I}} s^I\), let
\begin{align*}
    \bm S_n =  \big(\bm S_n^I\big)_{I \in \mathcal I}  =  \big(\sqrt{k}(\hat{\theta}^I_n-\theta_0^I) \big)_{I \in \mathcal{I}} \in \R^{s}
\end{align*}
and
\begin{align*}
    \bm{T}_n = \big(\bm T_n^I\big)_{I \in \mathcal I}  =  \bigg(\frac{1}{\sqrt{k}} \sum_{i=1}^n \big(Z_{i,n}^I - \Exp[Z_{i,n}^I]\big) \bigg)_{I\in \mathcal{I}} \in \R^{s}. 
\end{align*}
Let $\Sigma_n \coloneq \text{Var}(\bm{T}_n)$ and $\bm{G}_n \sim \mathcal{N}_{s}(\bm 0, \Sigma_n)$.

\begin{theorem}\label{thm:Mest-hd-CLT}
Suppose that, for each $I\in \mathcal I$, $\bm X_I$ has STDF $L_I$ satisfying \ref{cond:smoothness-good}; denote the respective set of bad points by $\mathfrak B_{I}$ and the constants by $K_{L_I}$. Suppose Assumption~\ref{ass:MEstUnif} holds.
Moreover, assume that  $m|\mathcal I| \geq 3, k \geq 2, s \geq 3, n/k \ge 2$ and that
\begin{compactenum}[(i)]
    \item $\sigma_{\min}^2 \coloneq \min_{i\in [s]} (\Sigma_n)_{ii}>0$,
    \item $\log(m^2 |\mathcal I|k^{1/4} ) \le 2k/{7}$,
    \item $\log(m |\mathcal I|k^{1/4} ) \le k/{\neunzig^2}$,
\end{compactenum}
with $\neunzig$ the universal constant from Lemma~\ref{lem:bound-on-order-statistics}.
Then there exist constants $\tilde C_{k}^{\mathcal I},\tilde C_\eta^{\Ic}>0$ and $\tilde C_N>0$ only depending on $m = \max_{I \in \mathcal I} m^I, s^{\Ic} = \max_{I \in \mathcal I} s^I, q^{\Ic} = \max_{I \in \mathcal I} q^I, K_L^{\mathcal I} = \max_{I \in \mathcal I} K_{L_I}$ and $ \thetainterior, \Cholder$ and the constants in \eqref{eq:constants-Ic}  such that, for any $\eta \le \tilde C_\eta^{\Ic}$ and any $k$  that additionally satisfies $k^{-1}\log(m |\mathcal I|k^{1/4} )  \le \tilde C_{k}^{\mathcal I}$, we have
\begin{align*}
    d_K(\bm S_n, \bm G_n) & \le \tilde C_N \Big\{ \sqrt{k\nearmin\log s} + \sqrt{\log s} B_{n,k}^{\mathcal I} + \Big(\frac{\log^5(sn)}{k}\Big)^{1/4} 
    \\&\hspace{4cm}+ 
    \log(sk) \max_{I \in \Ic} \int_{\mathfrak B_I^{\oplus \zeta_n}} \|{\bg}^I(\bx_I)\|_2 \, \diff \mu^I(\bx_I)   \Big\}, 
\end{align*}
where 
\[
 B_{n,k}^{\mathcal I} = \max_{I \in \Ic} 
 \sup_{\bm{x}_I \in [0,2]^{I}} \abs{B_n^I(\bm{x}_I)}, 
 \qquad
 \zeta_n := \neunzig \sqrt{1+\log m}\sqrt{\frac{\log (sk)}k}
\]
If $\mathfrak B_I = \emptyset$ and if $\nearmin$ is additionally chosen smaller than $k^{-3/2}$, the bound simplifies:
 \begin{align*}
d_K(\bm S_n, \bm G_n) & \le \tilde C_N \Big\{   \sqrt{\log s} B_{n,k}^{\mathcal I}
 + \Big(\frac{\log^5(sn)}{k}\Big)^{1/4} \Big \}, 
\end{align*}
which is analogous to the bound in Theorem~\ref{theo:clt}.
\end{theorem}

As noted earlier, Theorem~\ref{thm:Mest-hd-CLT} is useful for studying high-dimensional parametric models whose parameter is identifiable from lower-dimensional margins. A prominent example is the Hüsler–Reiss model from Example~\ref{ex:HR1}, for which the full parameter is identifiable from the collection of bivariate margins.  
As discussed after Lemma~\ref{lem:bivariate-bound-pickands}, each bivariate Hüsler–Reiss margin with $\gamma_{j\ell}>0$ satisfies Condition~\ref{cond:smoothness-good} with $\mathfrak B_{j\ell}=\emptyset$.

\subsection{Testing isotropy in spatial extremes}
\label{subsec:isotropy}

Suppose $\mathbb X=\{ X(\bm s): \bm x \in \mathcal S\}$ is a random field indexed by a spatial domain $\mathcal S  \subset \R^2$; for instance, $X(\bm s)$ could correspond to daily maximal wind speed at location $\bm s$ during a winter day. We assume that, for each pair $(\bm s, \bm s')$, the stable tail dependence function $L_{(\bm s_1, \bm s_2)}$ of $(X(\bm s_1), X(\bm s_2))$ exists. 
(Bivariate) extremal isotropy refers to the assumption that $L_{(\bm s_1, \bm s_2)}$ depends on $\bm s_1, \bm s_2$ only through the spatial domain distance $\| \bm s_1 - \bm s_2\|_2$; an assumption that is met for many max-stable models like the Smith model \citep{Smith1990} or Schlather's model \citep{Schlather2002}. In this section, we illustrate how the assumption can be tested (non-parametrically) based on repeated observations of $\mathbb X$ at a finite set of locations $\mathcal S_d=\{\bm s_1, \dots, \bm s_d\}$. In the non-extreme world, tests for isotropy are used routinely for model building and diagnostics \citep{WellerHoeting2016}. 

More formally, let $\mathcal P_d = \{ (\bm s_1, \bm s_2) \in \mathcal S_d \times \mathcal S_d: \bm s_1 \ne \bm s_2\}$ denote the set of (ordered) pairs of unequal locations, with $|\mathcal P_d | =d^2(d^2-1)$. For a given spatial distance $\rho >0$, let
\[
\mathcal P_d(\rho) = \{ (\bm s_1, \bm s_2)  \in \mathcal P_d: \| \bm s_1 - \bm s_2 \|_2 = \rho\}
\]
denote the set of (ordered) pairs of locations whose Euclidean distance is $\rho$; note that $\mathcal P_d(\rho)$ is non-empty for a finite set of distances only. For such a distance, consider the null hypothesis of \textit{extremal isotropy at spatial distance $\rho$} defined as
\[
H(\rho): \quad L_{(\bm s_1, \bm s_2)} = L_{(\bm s_1', \bm s_2')} \quad \text{ for all } \quad (\bm s_1,\bm s_2), (\bm s_1',\bm s_2') \in \mathcal P_d(\rho);
\] 
note each equality in the hypothesis essentially corresponds to the hypothesis considered in Section~4.2 in \cite{BucDet13}.
The intersection hypothesis $H = \bigcap_{\rho >0} H(\rho)$ then corresponds to \textit{(bivariate) extremal isotropy}.

In the following, and for simplicity, we restrict ourselves to the case of gridded observations on a rectangular domain; without loss of generality, $\mathcal S_d = \{1, \dots, d\}^2$. In that case, $|\mathcal P_d(1)| = 4d(d-1)$, $|\mathcal P_d(\sqrt 2)| = 4(d-1)^2$, and so on. We will concentrate on testing for $H(\rho)$ for $\rho \in \{1, \sqrt 2\}$ only, and illustrate how these tests can be combined to test for the intersection hypothesis $H(1, \sqrt 2) := H(1) \cap H(\sqrt 2)$. The resulting combination test can be interpreted as a test for extremal isotropy that is able to detect non-isotropic behavior for `small' distances ($\rho \le \sqrt 2)$.

A natural test statistic for $H(\rho)(L) $ is given by
\[
\widetilde T_n^{(\rho)}  = \max_{(\bm s_1, \bm s_2), (\bm s_1',\bm s_2') \in \mathcal P_d(\rho)} \sup_{t \in [0,1] }  
\sqrt k\big\{ \hat L_{(\bm s_1, \bm s_2)}(1-t,t) - \hat L_{(\bm s_1', \bm s_2')}(1-t,t) \big\},
\]
where $\hat L_{(\bm s_1, \bm s_2)}$ denotes the empirical STDF corresponding to the bivariate sample $(X_i(\bm s_1), \allowbreak X_i( \bm s_2))_{i \in [n]}$ and where we restrict attention to evaluation points $(1-t,t)$ since the population counterparts $L_{(\bm s_1, \bm s_2)}$ are uniquely determined by their restriction to the unit simplex.
To reduce the computational complexity, we further approximate the supremum by a finite maximum, and consider
\begin{align*}
T_n^{(\rho)} &=  \sqrt k \max_{(\bm s_1, \bm s_2), (\bm s_1',\bm s_2') \in \mathcal P_d(\rho)} \max_{t \in A}  
\big| \hat L_{(\bm s_1, \bm s_2)}(1-t,t) - \hat L_{(\bm s_1', \bm s_2')}(1-t,t)\big|
\\
&= \sqrt k \max_{t \in A} \Big\{ \max_{(\bm s_1, \bm s_2)\in \mathcal P_d(\rho)} \hat L_{(\bm s_1, \bm s_2)}(1-t,t) - \min_{(\bm s_1', \bm s_2')\in \mathcal P_d(\rho)} \hat L_{(\bm s_1', \bm s_2')}(1-t,t) \Big\} 
\end{align*}
instead, where $A=\{1/12, 2/12, \dots, 11/12\}$. Bootstrap versions of this statistic can be obtained as in Section~\ref{sec:clts}. Specifically, as in \eqref{eq:hatYI}, for some small positive bandwidth parameter $h$, let
\begin{align} \label{eq:hatY-isotropy}
\hat Y_{i,(\bm s_1, \bm s_2)}(x_1, x_2) 
&=  \nonumber
\frac1{\sqrt k} \Big[ \bm 1(\exists j \in [2]: \hat V_{i}(\bm s_j) < kx_j/n) - (k/n) \Lhat_{(\bm s_1, \bm s_2)}(\bm x)
     \\ & \hspace{2cm}-\sum_{j \in [2]}\pdLhat{j}_{(\bm s_1, \bm s_2)}(x_1, x_2) \big\{\bm 1( \hat V_{i}(\bm s_j) < kx_j/n) - kx_j/n \big\} \Big],
\end{align}
and with iid standard normal multipliers $e_1, \dots, e_n$ define
\[
T_{n,b}^{(\rho),*}  = 
\max_{(\bm s_1, \bm s_2), (\bm s_1',\bm s_2') \in \mathcal P_d(\rho)} \max_{t \in A}
\sum_{i=1}^n e_{i} \big\{ \hat Y_{i,(\bm s_1, \bm s_2)}(1-t,t)  - \hat Y_{i,(\bm s_1', \bm s_2')}(1-t,t)  \big\}.
\]
Theorem~\ref{theo:niceboot-new} suggests that, under the null hypothesis $H(\rho)$, the distribution of $T_n^{(\rho)}$ can be approximated by the conditional distribution of $T_{n}^{(\rho),*}$ given the data. 
Under fixed alternatives, however, $T_n^{(\rho)}$ explodes while the bootstrap can be expected to stay stochastically bounded. Overall, these considerations suggest to reject $H(\rho)$ if the $p$-value 
\[
    \hat p_{n}^{(\rho)} = 1-F_{n}^{(\rho),*}(T_n^{(\rho)} )
\]
is smaller than the nominal level $\alpha$; here, $F_{n}^{(\rho),*}$ denotes the conditional cdf of $T_n^{(\rho),*}$ given the data (in practice, the latter can be approximated using repeated simulation of $T_n^{(\rho),*} $).

The intersection hypothesis $H(1, \sqrt 2) := H(1) \cap H(\sqrt 2)$ can be tested using the approach described in \cite[Section 2]{bucher2019combining}. More specifically, let 
\[
C_n = \hat p_{n}^{(1)} \wedge \hat p_{n}^{(\sqrt 2)},
\]
and note that small values of $C_n$ provide evidence against the intersection  hypothesis. Critical values will be obtained using the bootstrap analogue
\[
    C_n^* = \hat p_{n}^{(1),*} \wedge \hat p_{n}^{(\sqrt 2),*},
    \qquad 
    \hat p_{n}^{(\rho), *} =1-F_{n}^{(\rho),*}(T_n^{(\rho),*} ),
\]
where it is crucial that the bootstrap expressions with $\rho \in \{1, \sqrt 2\}$ are based on the same multipliers.
Specifically, if $\hat q_{n,\alpha}^{*}$ denotes the conditional $\alpha$-quantile of $C_n^* $, given the data (again, the latter can be approximated by repeated simulation) we propose to reject $H(1,\sqrt 2)$ if $C_n \le \hat q_{n,\alpha}^{*}$.

The following result yields (approximate) finite-sample level control. Let $\mathcal P_d(1, \sqrt 2) = \mathcal P_d(1) \cup \mathcal P_d(\sqrt 2)$ and $D(1,\sqrt2 ) = D(1) \cup D(\sqrt 2)$ with $p = |D(1,\sqrt2 )|$, where
\begin{align}
\label{eq:Drho}
D(\rho) = \Big\{ (t, \bm s_1, \bm s_2, \bm s_1', \bm s_2') : t \in A,  (\bm s_1, \bm s_2), (\bm s_1', \bm s_2') \in  \mathcal P_d(\rho) , (\bm s_1, \bm s_2) \ne (\bm s_1', \bm s_2')\Big\}.
\end{align}
Further, writing $V(\bm s) = 1-F_{\bm s}(X(\bm s))$ with $F_{\bm s}$ the cdf of $X(\bm s)$, introduce the bias term
\begin{align} \label{eq:bias-isotropy}
B_{n, (\bm s_1, \bm s_2)}(x_1, x_2)
=
\sqrt k \Big\{    \frac{n} k \Prob\Big( \exists j \in \{1,2\}: V(\bm s_j) \le \frac{k}n x_j\Big) - L_{(\bm s_1, \bm s_2)}(x_1, x_2) \Big\}.
\end{align}

\begin{theorem}
\label{theo:level-control-isotropy}
Suppose $H(1,\sqrt 2)$ holds and that there exist $\kappa_L, \Kl \in(0,\infty)$ such that
\begin{align*}
&\forall (\bm s_1, \bm s_2) \in \mathcal P_d(1, \sqrt 2),  \forall \bm x \in A(1 \wedge (\kappa_L/2)), \bm y \in [0,\infty)^2  \text{ with } \| \bm x- \bm y\|_\infty \le \kappa_L, \forall j \in [2]: \\
& \hspace{3cm}\pdL{j}_{(\bm s_1, \bm s_2)}(\bm x), \pdL{j}_{(\bm s_1, \bm s_2)}(\bm y) \text{ exist and satisfy } \\
&\hspace{4cm} |\pdL{j}_{(\bm s_1, \bm s_2)}(\bm x)-\pdL{j}_{(\bm s_1, \bm s_2)}(\bm y)| \le \Kl \|\bm x - \bm y\|_\infty,
\end{align*} 
where $A(\kappa) 
= \{(1-t,t): t \in A\}^{\oplus \kappa}$.
Moreover, assume that $n/k \ge 2, |\mathcal P_d(1, \sqrt 2)|\ge 3$ and 
\begin{compactenum}[(i)]
\item $\sigma_{\min}^2 := \min_{(t,\bm s_1, \bm s_2,\bm s_1',\bm s_2') \in D(1,\sqrt2 )} \Var(\widebar \Db_{n, (\bm s_1, \bm s_2), (\bm s_1', \bm s_2')}(t) )>0$  with $\widebar \Db_{n}$ from \eqref{eq:bardbn}.
\item $\log(2 |\mathcal P_d(1, \sqrt 2)|k^{1/4} ) \le 2k/7$.
\item 
$\log(|\mathcal P_d(1, \sqrt 2)| k^{1/4}) \le \kappa_L^2 k / (8\neunzig^2)$ with $\neunzig$ from Lemma~\ref{lem:bound-on-order-statistics},
\end{compactenum}
Let $0<c_h< c_h'< \infty$ be constants, and assume that the bandwidth $h<(\min_{t\in A}  (1-t) \wedge t) \wedge (\kappa_L/2)$ satisfies
\[
c_h\Big( \frac{\log(p+k)}k\Big)^{1/2} 
\le h \le 
c_h' \Big( \frac{\log(p+k)}k\Big)^{1/4}.
\]
There exists a constant $c_0$ depending on $K_L, \sigma_{\min}^2c_h, c_h'$ only such that
\[
    \Big|\Prob( C_n \le \hat q_{n,\alpha}^* ) - \alpha \Big|
    \le 
    c_0 
    \Big[ 
        \sqrt{\log(p+k)}\,  B_{n,k} + \Big( \frac{\log^5(pn)}{k} \Big)^{1/4}
    \Big],
\]
where 
$B_{n,k}= \max_{\rho \in \{1, \sqrt 2\},(\bm s_1, \bm s_2) \in \mathcal P_d(\rho),(x_1, x_2) \in A(\kappa_L)} |B_{n, (\bm s_1, \bm s_2)}(x_1, x_2)|$ with $B_{n, (\bm s_1, \bm s_2)}$ from \eqref{eq:bias-isotropy}.
\end{theorem}

We end this section by illustrating the performance of the above tests in a small simulation study. For that purpose, we consider data generated from the max-stable Brown-Resnick random field \citep{KabSchDeh09}, whose bivariate STDF at location pair $(\bm s_1, \bm s_2)$ is given by that of the bivariate H\"usler-Reiss distribution from Example~\ref{ex:HR1}, i.e.,
\[
L_{(\bm s_1, \bm s_2)}(x_1,x_2) = x_1 \Phi\left( \frac{a}{2} + \frac{1}{a}\log\left(\frac{x_1}{x_2}\right)\right) 
  + x_2\Phi\left( \frac{a}{2} + \frac{1}{a}\log\left(\frac{x_2}{x_1}\right)\right),
\] 
where $\Phi$ denotes the c.d.f. of standard normal distribution and where 
\[
a^2 = \gamma_{\xi, \beta}(\bm s_1, \bm s_2) = \beta \Big[ (\bm s_1 - \bm s_2)^\top \Sigma^{-1} (\bm s_1 - \bm s_2) \Big]^{\xi/2}
\]
for some $\Sigma \in \R^{2\times 2}$ positive definite and parameters $\beta>0$ and $\xi \in (0,2]$. Note that the respective extremal coefficients are given by
\[
\chi(\bm s_1, \bm s_2) = 2-L_{(\bm s_1, \bm s_2)}(1,1) = 2 - 2 \Phi\Big( \frac{\gamma_{\xi, \beta}(\bm s_1, \bm s_2)^{1/2}}2\Big).
\]
For the simulation study, we consider the choices $\xi \in \{0.9,1.8\}, \beta = 0.5$ and covariance matrices 
\[
\Sigma_1 = \begin{pmatrix}
    1 & 0 \\
    0 & 1
\end{pmatrix} \text{ (isotropic)}, \qquad \Sigma_2 = \begin{pmatrix}
    0.5 & 0.25 \\
    0.25 & 1
\end{pmatrix} \text{ (anisotropic)},
\]
The resulting extremal coefficients only depend on the (linear span of the) spatial lag $\bm\rho=\bm s_1 - \bm s_2$; they are explicitly provided in Table~\ref{tab:chi-brown-resnick} for the case where $\| \bm\rho \|_2 \in \{1, \sqrt 2\}$.

\begin{table}[t]
\caption{Values of $\chi(\bm s_1, \bm s_2)$ for spatial lag $\bm{\rho}=\bm s_1 - \bm s_2 = (1,0)^\top$ [hor], $\bm{\rho}=(0,1)^\top$ [vert], $\bm{\rho}=(1,-1)^\top$ [dia1] and $\bm \rho = (1,1)^\top$ [dia2].}
\label{tab:chi-brown-resnick}
\begin{tabular}{@{}cc|cc|cc@{}}
\toprule
$\Sigma$                                        & $\xi$ & \multicolumn{1}{l}{hor} & \multicolumn{1}{l|}{vert} & \multicolumn{1}{l}{dia1} & \multicolumn{1}{l}{dia2} \\ \midrule
\multicolumn{1}{c}{\multirow{2}{*}{$\Sigma_1$}} & 0.9      & \multicolumn{2}{c|}{0.72}                              & \multicolumn{2}{c}{0.68}                                \\
\multicolumn{1}{c}{}                            & 1.8      & \multicolumn{2}{c|}{0.72}                              & \multicolumn{2}{c}{0.63}                                \\ \midrule
\multicolumn{1}{c}{\multirow{2}{*}{$\Sigma_2$}}                    & 0.9      & 0.67 & 0.72 & 0.67 & 0.62                       \\
                                                & 1.8      & 0.61 & 0.71 & 0.61 & 0.48                        \\ \bottomrule
\end{tabular}
\end{table}

For the simulation study, we consider a sample size of $n = 10^4$ and a spatial grid $\mathcal S_{10} = [10]^2$. The number of equations to be tested for the hypothesis $H(\rho)$ is $\binom{|\mathcal P_d(1)|}{2}=360\cdot 359/2= 64\, 620$ for $\rho = 1$ and $52\,326$ for $\rho = \sqrt{2}$, yielding a total of $116\,946$ equations for the combined intersection hypothesis. For each parameter configuration, we generate 200 datasets and evaluate the three tests corresponding to $H(1)$, $H(\sqrt 2)$, and $H(1, \sqrt 2)$. In each case, we employ $B = 500$ bootstrap replications and consider threshold parameters $k \in \{200, 350, 500\}$. The results are summarized in Table~\ref{tab:sim-results}, which reports rejection frequencies at significance level $0.05$. The findings are consistent with theoretical expectations: all tests maintain the nominal level. Moreover, the power increases from $H(1)$ to $H(\sqrt 2)$ to $H(1, \sqrt 2)$ and is also increasing in $\xi$.

\begin{table}[t]
\caption{Rejection rates (in percent) for the null hypothesis $H(\Xi)$ with $\Xi=1, \Xi=\sqrt{2}$ and $\Xi=(1, \sqrt{2})$. Entries for $\Sigma_1$ correspond to incorrect rejections, and for $\Sigma_2$ to correct rejections.} \label{tab:sim-results}
\begin{tabular}{ll | rrr | rrr}
        \toprule
         &  & \multicolumn{3}{c}{$\Sigma_1$} & \multicolumn{3}{c}{$\Sigma_2$} \\
        \cmidrule(lr){3-5} \cmidrule(lr){6-8}
        $\xi$ & $\Xi$ & $k=200$ & $k=350$ & $k=500$ & $k=200$ & $k=350$ & $k=500$ \\
        \midrule
        
        \multirow{3}{*}{0.9}
        & $1$ & 2.5 & 4.0 & 4.5 & 22.5 & 50.0 & 75.5 \\
        & $\sqrt{2}$ & 1.5 & 4.5 & 5.5 & 29.5 & 62.0 & 85.0 \\
        & $1, \sqrt{2}$ & 1.0 & 2.0 & 4.5 & 31.0 & 73.0 & 88.0 \\
        \midrule
        
        \multirow{3}{*}{1.8}
        & $1$ & 3.0 & 3.0 & 4.0 & 90.5 & 100.0 & 100.0 \\
        & $\sqrt{2}$ & 3.5 & 5.0 & 4.0 & 99.0 & 100.0 & 100.0 \\
        & $1, \sqrt{2}$ & 4.0 & 3.0 & 4.0 & 99.0 & 100.0 & 100.0 \\
        \bottomrule
        \end{tabular}
\end{table}

\begin{funding}
AB and KE were supported by the Deutsche Forschungsgemeinschaft (DFG, German Research Foundation; Project-ID 520388526; TRR 391: Spatio-temporal Statistics for the
Transition of Energy and Transport) which is gratefully acknowledged. YC and SV gratefully acknowledge support by a Discovery Grant
from the Natural Sciences and Engineering Research Council of Canada (NSERC, grant RGPIN-2024-05528). Calculations for this publication were performed on the HPC cluster Elysium of the Ruhr University Bochum,
subsidised by the DFG (INST 213/1055-1).
\end{funding}



\bibliographystyle{imsart-nameyear} 
\bibliography{biblio}       

\newpage

\appendix
\thispagestyle{empty}
\numberwithin{equation}{section}

\begin{center}
	
	{\bfseries SUPPLEMENT TO THE PAPER:  \\  ``EMPIRICAL TAIL DEPENDENCE FUNCTIONS IN HIGH DIMENSIONS:
UNIFORM LINEARIZATIONS AND INFERENCE'' }
	\vspace{.5cm}
	
	{\textsc{By Axel Bücher, Yeonjoon Choi, Katharina Effertz and Stanislav Volgushev}}
	
	\vspace{.28cm}
	
	{\textit{Ruhr-Universität Bochum and University of Toronto}}
	
	\vspace{.28cm}

	\begin{center}
		\begin{minipage}{.6\textwidth}
			{\small \hspace{.5cm}
					Appendix~\ref{sec:proofs} contains the proofs of all results stated in the main text. Additional auxiliary results are stated and proved in Appendix~\ref{sec:auxiliary-results}.
				 }
		\end{minipage}
	\end{center}

\end{center}

\vspace{.5cm}

\addtocontents{toc}{\protect\setcounter{tocdepth}2} 

\tableofcontents

\section{Proofs}
\label{sec:proofs}

\subsection{Proofs for Section~\ref{sec:linearizations}}

\begin{proof}[Proof of Lemma~\ref{lem:bivariate-bound-pickands}] 
Since $L(x_1, x_2) = (x_1+x_2) A(x_2/(x_1+x_2))$ for all $\bm x =(x_1, x_2) \in [0,\infty)^d$ such that $x_1 + x_2>0$, we have, for $\bm x \in (0,\infty)^2$,
\begin{align*}
    \pdL{1}(x_1, x_2) &=  A\Big(\frac{x_2}{x_1+x_2}\Big) -\frac{x_2}{x_1+x_2}A'\Big(\frac{x_2}{x_1+x_2}\Big), \\
    \pdL{2}(x_1, x_2) &=  A\Big(\frac{x_2}{x_1+x_2}\Big) + \frac{x_1}{x_1+x_2}A'\Big(\frac{x_2}{x_1+x_2}\Big).
\end{align*}
Moreover, $\pdL{1}(x_1, 0) = \pdL{2}(0,x_2) = 1$ for $x_1,x_2>0$. Continuity of $\pdL{1}$ on $(0,\infty)^2$ is immediate. Further, for a sequence $\bm x_n$ in $(0,\infty)^2$ converging to $\bm x=(x_1, 0)$ with $x_1>0$, we have $\lim_{n\to\infty}x_{n2}/(x_{n1}+x_{n2})=0$, which implies $\lim_{n\to\infty} \pdL{1}(\bm x_n) = A(1)=1=\pdL{1}(\bm x)$ by continuity of $A$ on $[0,1]$ and boundedness of $A'$ on $(0,1)$.
Hence, $\pdL{1}$ is continuous on $E_1$, and the same arguments show continuity of $\pdL{2}$ on $E_2$.
Regarding the second-order partial derivatives, note that, for $\bm x \in (0,\infty)^2$,
\begin{align*}
\pdL{11} (x_1,x_2) &= \frac{x_2^2}{(x_1+x_2)^3}A''\Big(\frac{x_2}{x_1+x_2}\Big) =  \frac{t^2 A''(t)}{x_1+x_2} \\
\pdL{22} (x_1,x_2) &= \frac{x_1^2}{(x_1+x_2)^3}A''\Big(\frac{x_2}{x_1+x_2}\Big) = \frac{(1-t)^2 A''(t)}{x_1+x_2} \\
\pdL{12}(x_1,x_2) &= -\frac{x_1x_2}{(x_1+x_2)^3}A''\Big(\frac{x_2}{x_1+x_2}\Big) = - \frac{t(1-t) A''(t)}{x_1+x_2},
\end{align*}
where we write $t=x_2/(x_1+x_2)$. Continuity on $(0,\infty)^2$ is immediate. Moreover,
\begin{align*}
\frac{t^2 A''(t)}{x_1+x_2} 
&= 
t(1-t)A''(t) \frac{x_2}{x_1+x_2} \frac{1}x_1 
\le A_\infty \frac{1}{x_1} \\
\frac{(1-t)^2 A''(t)}{x_1+x_2} 
&= 
t(1-t)A''(t) \frac{x_1}{x_1+x_2} \frac{1}x_2 
\le A_\infty \frac{1}{x_2}
\end{align*}
and
\begin{align*}
\Big|-\frac{t(1-t) A''(t)}{x_1+x_2} \Big|
\le 
\frac{A_\infty}{x_1+x_2}
\le
\frac{A_\infty}{x_1 \vee x_2},
\end{align*}
which finalizes the proof.
\end{proof}

\begin{proof}[Proof of Theorem~\ref{theo:linearization-hoelder} and Theorem~\ref{theo:linearization-good}]
We start by noting that our assumption $n/k \ge T$ implies that, for any $\bm x \in [0,T]^d$, we have $kx_j/n \le 1$ for all $j \in [d]$. In the subsequent proof, we will only consider such $\bm x$.

Recall the definition $\bm V_i=(V_{i1}, \dots, V_{id})^\top$ with $V_{ij}=1-F_j(X_{ij})$ for $j\in[d]$ and $i\in[n]$. Let $V_{1:n,j} \le V_{2:n, j} \le \dots \le V_{n:n,j}$ denote the order statistics of $V_{1j}, \dots V_{nj}$, and define $Q_{nj}(v_j) = V_{\lceil nv_j \rceil:n, j}$ for $v_j \in (0,1]$, where $\lceil a \rceil$ denote the smallest integer not smaller than $a$. For completeness, we define $Q_{nj}(0)=0$. Note that $Q_{nj}(v_j) =  G_{nj}^\leftarrow(v_j)$ with $G_{nj}(u_j) = \frac1n \sum_{i=1}^n \bm 1(V_{ij} \le u_j)$ the empirical cdf of $V_{1j}, \dots, V_{nj}$ and 
\begin{align} \label{eq:generalized-inverse}
H^\leftarrow(v) = \inf\{u \in[0,\infty): H(u) \ge v\}
\end{align}
the left-continuous generalized inverse of a non-decreasing function $H:[0,\infty) \to [0,\infty)$. 

Observing that the rank of $V_{ij}$ among $V_{1j}, \dots, V_{nj}$ is equal to $n+1-R_{ij}$, we have $V_{ij} < V_{\lceil kx_j \rceil:n, j}$ if and only if $n+1-R_{ij} < \lceil kx_j \rceil$, which in turn is equivalent to $R_{ij}>n+1-kx_j$ \footnote{{ $n+1-kx_j \in [ n+1-\lceil kx_j \rceil, n+2-\lceil kx_j \rceil)$ and $R_{ij} \in \N$, thus $R_{ij}>n+1-\lceil kx_j \rceil$ implies $R_{ij} \ge n+2-\lceil kx_j \rceil > n+1-kx_j$, also conversely $R_{ij} > n+1-kx_j \ge n+1-\lceil kx_j \rceil$}}. We may therefore write $\Lhat_n(\bm x)=\Loracle_n(S_n(\bm x))$ for $\bm x \in [0, T]^d$, where $\Loracle_n$ is from \eqref{eq:Ltilde} and where $S_n(\bm x) = (S_{n1}(x_1), \dots, S_{nd}(x_d))^\top$ with 
\begin{align}
\label{eq:snj}
S_{nj}(x_j)
&= 
\frac{n}k Q_{nj}\Big(\frac{k}n x_j \Big) 
=
\frac{n}k V_{\lceil kx_j \rceil:n, j} \bm1(x_j>0), \qquad j \in [d].
\end{align}
Further, let
\begin{align} \label{eq:Lnj-tilde}
\Loracle_{nj}(x_j) &:= \Loracle_n(0, \dots, 0, x_j, 0, \dots 0) = \frac1k \sum_{i=1}^n \bm 1\Big(V_{ij} < \frac{k}nx_j \Big)
\end{align}
and note that $\Loracle_{nj}^\leftarrow(x_j) = S_{nj}(x_j)$. Finally, recalling the definition of $\Lpreasy_n$ from \eqref{eq:mutilde}, note that $\Exp[\Loracle_n(\bm x)]=\Lpreasy_n(\bm x)$ and that $\Lpreasy_{nj}(x_j):= \Lpreasy_n(0, \dots, 0, x_j, 0, \dots 0)$ satisfies $\Lpreasy_{nj}(x_j)=\Lpreasy^\leftarrow_{nj}(x_j)=x_j$. 

The above definitions and identities imply the decomposition
\begin{align}
\label{eq:decomposition1-Lbn}
\Lb_n 
= 
\sqrt k (\Lhat_n - L ) 
&=  \nonumber
\sqrt k (\Loracle_n \circ S_n - \Lpreasy_n \circ S_n) 
+
\sqrt k (L \circ S_n - L)
+ \sqrt k (\Lpreasy_n \circ S_n - L \circ S_n)
\\&=
\widetilde \Lb_n \circ S_n + \sqrt k (L \circ S_n - L) + \sqrt k (\Lpreasy_n - L) \circ S_n.
\end{align}
By Lemma~\ref{eq:bound-on-order-statistics-2}, we have, on an event $\Omega_0$ with probability at least $1-(d+1)\delta$,
\begin{align} \label{eq:Snj-uniform-bound}
\max_{j\in[d]} \sup_{x_j \in [0,T]} |S_{nj}(x_j) - x_j| \le \neunzig  r(\delta,T,k),
\end{align}
where $\neunzig \approx 89.18$ is from Lemma~\ref{lem:bound-on-order-statistics} and where $r$ is defined in~\eqref{eq:definition-r}. Subsequently, we work on this event. 

We now distinguish between the two theorems: under the conditions of Theorem~\ref{theo:linearization-hoelder}, we have $\neunzig r \le \kappa_L$ by our assumption $r \le \kappa_L/\neunzig$. Hence, for any $\bm x \in A$, we have $S_n(\bm x) \in A^{\oplus \kappa_L}$, whence we can apply \ref{cond:smoothness-hoelder} and the mean value theorem to conclude that there exists a (random) $t^* := t^*_n(\bm x) \in [0,1]$ such that 
\begin{align*}
    \sqrt k \{ L(S_n(\bm x)) - L(\bm x) \} 
&= \sum_{j \in [d]} \pdL{j}(\bm x + t^*(S_n(\bm x) - \bm x))\sqrt k\{ S_{nj}(x_j) - x_j \}.
\end{align*}
Likewise, under the conditions of Theorem~\ref{theo:linearization-good}, for any $\bm x \in [0,T]^d \setminus (\mathfrak B^{\oplus \neunzig r})$, we have $S_n(\bm x) \in [0,T+\neunzig r]^d \setminus \mathfrak B$ by~\eqref{eq:Snj-uniform-bound}, and \ref{cond:smoothness-good} and the mean value theorem allows to conclude that the previous display holds for any $\bm x \in [0,T]^d \setminus (\mathfrak B^{\oplus \neunzig r})$.

In the following, we either consider $\bm x \in A$ (Theorem~\ref{theo:linearization-hoelder}), or $\bm x \in  [0,T]^d \setminus (\mathfrak B^{\oplus \neunzig r})$ (Theorem~\ref{theo:linearization-good}). In both cases,
the previous display and \eqref{eq:decomposition1-Lbn}, together with the definitions $\widebar {\mathbb L}_n(\bm x)  = \widetilde {\mathbb L}_n(\bm x)  - \sum_{j=1}^d \pdL{j}(\bm x) \widetilde {\mathbb L}_{nj}(x_j)$ and $B_n(\bm x) = \sqrt k \{ \Lpreasy_n(\bm x) - L(\bm x)\}$, imply the fundamental decomposition
\begin{align}
\label{eq:decomposition2-Lbn}
\Lb_n(\bm x) - \widebar \Lb_n (\bm x) - B_n(S_n(\bm x)) =
D_{n1}(\bm x) + D_{n2}(\bm x) + D_{n3}(\bm x),
\end{align}
where
\begin{align}
D_{n1}(\bm x) \label{eq:definition-dn1}
& = \widetilde \Lb_n \circ S_n(\bm x) - \widetilde \Lb_n(\bm x), \\
D_{n2}(\bm x) \label{eq:definition-dn2}
&= \sum_{j \in [d]} \pdL{j}\big(\bm x + t^*(S_n(\bm x) - \bm x)\big) \big[  \sqrt k\{ S_{nj}(x_j) - x_j \} +  \widetilde \Lb_{nj}(x_j) \big] \\
D_{n3}(\bm x) \label{eq:definition-dn3}
&=\sum_{j \in [d]}  \big[ \pdL{j}(\bm x)- \pdL{j}\big(\bm x + t^*(S_n(\bm x) - \bm x)\big) \big]  \widetilde \Lb_{nj}(x_j).
\end{align}
Moreover, since the partial derivatives of $L$ are bounded by $1$ (whenever they exist), we have 
\begin{align}
\label{eq:definition-dn2-prime}
|D_{n2}(\bm x)| 
\le 
\sum_{j\in [d]} \Big|\sqrt k\{ S_{nj}(x_j) - x_j \} +  \widetilde {\mathbb L}_{nj}(x_j)  \Big|
=: 
D_{n2}'(\bm x);
\end{align}
note that $D_{n2}'$ is well-defined on $[0,\infty)^d$. 

Regarding Theorem~\ref{theo:linearization-hoelder}, 
its first result is now an immediate consequence of Lemma~\ref{lem:dn1-l}, \ref{lem:dn2-l} and \ref{lem:dn3-l-hoelder}. Moreover, 
\[
\sup_{\bm x \in A} |B_n(S_n(\bm x))|  \le \sup_{\bm x \in A^{\oplus \neunzig r}} |B_n(\bm x)|
\]
is an immediate consequence of \eqref{eq:Snj-uniform-bound}.

Regarding Theorem~\ref{theo:linearization-good}, its first result is an immediate consequence of Lemma~\ref{lem:dn1-l}, \ref{lem:dn2-l} and \ref{lem:dn3-l-good}.
\end{proof}

\begin{lemma}
\label{lem:dn1-l}
Fix $d \in \N_{\ge 2}$.
There exist constants $D_{1,1} = D_{1,1}(d) \ge 1$ and $D_{1,2} = D_{1,2}(d) \ge 1$ only depending on $d$ such that, for any  $n \in \N,k \in [n], T \in \N$ and $\delta \in (0,e^{-1})$ satisfying $\log(d/ \delta) \le 2kT/7$, we have
\begin{align} \label{eq:dn1-l}
\sup_{\bm x \in [0,T]^d} |D_{n1}(\bm x) | 
\le 
D_{1,1} \sqrt{r\log\Big(\frac{T D_{1,2}}{\delta r}\Big)} 
=: 
\lambda_{n,k,d,T}^{(1)}(\delta)
\end{align}
with probability at least $1-(d+2) \delta$, where $D_{n1}(\bm x)$ is from \eqref{eq:definition-dn1} and where $r=r(\delta, T, k)$ is from \eqref{eq:definition-r}.
\end{lemma}

\begin{lemma} \label{lem:dn2-l}
There exist universal constants $D_{2,1} \ge 1$ and $D_{2,2} \ge 1$  such that, 
for any  $n \in \N,k \in [n], d \in \N, T \in \N$ and $\delta \in (0,e^{-1})$ satisfying $\log(d/ \delta) \le 2kT/7$ and $n/k \ge T$,
\begin{align} \label{eq:dn2-l}
\sup_{\bm x \in [0,T]^d} D_{n2}'(\bm x)  
\le 
\frac{d}{\sqrt{k}}
+
D_{2,1}d  \sqrt{r\log\Big(\frac{T D_{2,2}}{\delta r}\Big)} 
=: 
\lambda_{n,k,d,T}^{(2)}(\delta)
\end{align}
with probability at least $1-(2d + 1)\delta$, where $D_{n2}'(\bm x)$ is from \eqref{eq:definition-dn2-prime} and where $r=r(\delta, T, k)$ is from \eqref{eq:definition-r}.
\end{lemma}

\begin{lemma}\label{lem:dn3-l-hoelder}
Fix $d,T\in \N$ and let $(A,L)$ with $A \subset[0,T]^d$ satisfy
\ref{cond:smoothness-hoelder} from Theorem~\ref{theo:linearization-hoelder}.
Then there exist some constant $D_{3} = D_{3}(d,\Kl,\alpha_L) \ge 1$ only depending on $d,\Kl$ and $\alpha$ such that, for any $n \in \N,k \in [n]$ and $\delta \in (0,e^{-1})$ satisfying $\log(d/\delta)\le 2kT/7$ and $r \le \kappa_L/\neunzig$ with $\neunzig\approx 89.18$ from Lemma~\ref{lem:bound-on-order-statistics},
    \[
    \sup_{\bm x \in A} |D_{n3}(\bm x)| 
    \le 
    D_{3} r^{\alpha_L}\sqrt{T\log(1/\delta)}
    =:
    \lambda_{n,k,d,T,\Kl,\alpha_L}^{(3)}
    \]
with probability at least $1-(2d + 1)\delta$, where $D_{n3}(\bm x)$ is from \eqref{eq:definition-dn3} and where $r=r(\delta, T, k)$ is from \eqref{eq:definition-r}.
\end{lemma}

\begin{lemma}\label{lem:dn3-l-good}
Fix $d,T\in \N$ and assume that \ref{cond:smoothness-good} from Theorem \ref{theo:linearization-good} if met. Then, there exists a constant $D_{4} = D_{4}(d,\Kl) \ge 1$ such that, for any $n \in \N,k \in [n]$ and $\delta \in (0,e^{-1})$ satisfying $\log(d/\delta)\le 2kT/7$ and $n/k\ge 2T$, we have
\[
    \sup_{\bm x \in [0,T]^d \setminus (\mathfrak B^{\oplus \neunzig r})} |D_{n3}(\bm x)| 
    \le 
    D_{4} \sqrt{r \log\Big(\frac{T}{\delta r}\Big)}
    =:
    \lambda_{n,k,d,T,\Kl}^{(4)}
\]
with probability at least $1-(3d + 1)\delta$, where $D_{n3}(\bm x)$ is from \eqref{eq:definition-dn3} and where $r=r(\delta, T, k)$ is from \eqref{eq:definition-r}.
\end{lemma}

\begin{proof}[Proof of Lemma~\ref{lem:dn1-l}]

Subsequently, let $\Omega_0$ denote the event of probability at least $1-(d+1)\delta$ on which \eqref{eq:bound-on-order-statistics-1} and \eqref{eq:bound-on-order-statistics-2} are met, and let $\neunzig\approx 89.18$ denote the universal constant in \eqref{eq:bound-on-order-statistics-2}.

Let $\bm x \in [0,T]^d$. Then, on $\Omega_0$, we have
\begin{align*}
\sup_{\bm x\in [0,T]^d} | D_{n1}(\bm x) | 
=
\sup_{\bm x\in [0,T]^d} | \widetilde{\Lb}_n(S_n(\bm x)) -  \widetilde{\Lb}_n(\bm x) |  
\le \omega_{\widetilde \Lb_{n}}\Big(\max_{j \in [d]} \sup_{x_j \in [0,T]} |S_{nj}(x_j)  -  x_j |; [0,2T]^d\Big)
\end{align*}
where $\omega_{f}(\eps;B)$ denotes the modulus of continuity of $f$ with respect to the maximum norm as defined in~\eqref{eq:definition-modulus}, and where we used \eqref{eq:bound-on-order-statistics-1}.

We next distinguish two cases. First, suppose that $\neunzig r \le 2T$, where $r=\sqrt{(T/k) \log(1/\delta)}$ is from \eqref{eq:definition-r}. Then, on the event $\Omega_0$, by \eqref{eq:bound-on-order-statistics-1} and \eqref{eq:bound-on-order-statistics-2},
\begin{align*}
\sup_{\bm x\in [0,T]^d} | D_{n1}(\bm x) | 
    \le \omega_{\widetilde \Lb_{n}}(\neunzig r; [0,2T]^d)
    = 
    \sqrt{\frac{n}{k} }
    \omega_{ \beta_{n}}\Big(\frac{k}{n}\neunzig r; [0,2Tk/n]^d\Big),
\end{align*}
with $\beta_n$ from \eqref{eq:betan}.
Next, by \eqref{eq:modulus_beta1} from Lemma \ref{lem:modulus_newnewnew} (which is applicable since $\neunzig r \le 2T$), there exists a set $\Omega_1$ with probability at least $1-\delta$ such that, on $\Omega_1$, 
\begin{align} \label{eq:lemma1-bound1}
\sqrt{\frac{n}{k} }\omega_{ \beta_{n}}\Big(\frac{k}{n}\neunzig r; [0,2Tk/n]^d\Big) 
\le 
\kappa \sqrt{\neunzig r \log\Big(\frac{4dT}{\neunzig r\delta}\Big)},
\end{align}
where 
\[
\kappa =   2d \Big[ \sqrt{\frac4{9\neunzig kr} \log\Big( \frac{4dT}{\neunzig r\delta}\Big)} + 2 +60 \sqrt{2d}  \Big].
\]
Since $\log(x) \le x$ and $1 \le \log(1/\delta) \le 2kT/7$, we have
\begin{align} \label{eq:bound-kappa}
\frac4{9\neunzig kr} \log\Big( \frac{4dT}{\neunzig r\delta}\Big)
&= \nonumber
\frac4{9\neunzig kr} \Big\{ \log \Big(\frac{4dT}{\neunzig r}\Big) + \log(1/\delta) \Big\}
\\&\le \nonumber
\frac4{9\neunzig kr} \Big\{ \frac{4dT}{\neunzig r} + \sqrt{\log(1/\delta)\cdot2kT/7}\Big\}
\\&= 
\frac4{9\neunzig kr} \Big\{ \frac{4d rk}{\neunzig \log(1/\delta)} + \sqrt{2/7}rk\Big\}
\le
\frac4{9\neunzig }  \Big\{ \frac{4d}{\neunzig } + \sqrt{2/7}\Big\};
\end{align}
note that the upper bound only depends on $d$. As a consequence, by \eqref{eq:lemma1-bound1}, there exist constants $D_{1,1}=D_{1,1}(d)$ and $D_{1,2}=D_{1,2}(d)$ only depending on $d$ such that, on $\Omega_1$,
\begin{align*}
\sqrt{\frac{n}{k} }\omega_{ \beta_{n}}\Big(\frac{k}{n}\neunzig r; [0,2Tk/n]^d\Big) 
\le 
D_{1,1} \sqrt{r \log\Big(\frac{TD_{1,2}}{r\delta}\Big)} = \lambda_{n,k,d,T}^{(1)}(\delta),
\end{align*}
which in turn implies \eqref{eq:dn1-l} on the event $\Omega_0 \cap \Omega_1$ and in the case $\neunzig r \le 2T$. The assertion follows from the fact that this event has probability at least $1-(d+2)\delta$.

It remains to treat the case $\neunzig r > 2T$. In that case,
on $\Omega_0$, by the triangle inequality,
\begin{align*}
\sup_{\bm x\in [0,T]^d} | D_{n1}(\bm x) | 
\le 2 \sup_{\bm x \in [0,2T]^d} |{\widetilde \Lb_{n}} (\bm x) |.
\end{align*}
By Lemma \ref{lem:bnd-Ltilde}, there exists an event $\Omega_1'$ that has probability at least $1-\delta$ such that, on $\Omega_1'$ and with $\neunzig $ from \eqref{eq:bound-on-order-statistics-2},  $\sup_{\bm x \in [0,2T]^d} |{\widetilde \Lb_{n}} (\bm x) | \le (188/3) \cdot \sqrt 2  \cdot d \sqrt{T \log(1/\delta)} \le \neunzig  d  \sqrt{T\log(1/ \delta)}$. Hence, on $\Omega_0 \cap  \Omega_1'$, we have
\begin{align*}
\sup_{\bm x\in [0,T]^d} | D_{n1}(\bm x) | 
\le 2 \neunzig  d \sqrt{T \log(1/\delta)}
&\le \sqrt{2} \neunzig ^{3/2} d  \sqrt{r \log(1/\delta)}
\\&\le \sqrt{2} \neunzig ^{3/2} d  \sqrt{r\log\Big(\frac{\sqrt{2/7} \cdot T}{r\delta} \Big)},
\end{align*}
where we used that $T \le \neunzig r/2$ and $r \le \sqrt{2/7} \cdot T$ at the last two inequalities. By possibly increasing $D_{1,1}$ and $D_{1,2}$, the upper bound is bounded by $\lambda_{n,k,d,T}^{(1)}(\delta)$. Overall, we have shown that \eqref{eq:dn1-l} holds on the event $\Omega_0 \cap \Omega_1'$ and in the case $\neunzig r > 2T$. The assertion follows from the fact that this event has probability at least $1-(d+2)\delta$.
\end{proof}

\begin{proof}[Proof of Lemma~\ref{lem:dn2-l}]

We start by writing
\begin{align*}
\sqrt k\{ S_{nj}(x_j) - x_j \} 
&=-
\sqrt k \{ \Loracle_{nj}(S_{nj}(x_j)) - S_{nj}(x_j) \} +
\sqrt k \{ \Loracle_{nj}(S_{nj}(x_j)) - x_j \} \\
&=
- \widetilde {\mathbb L}_{nj} \circ S_{nj}(x_j) + \sqrt k \{ \Loracle_{nj}(\Loracle_{nj}^\leftarrow(x_j)) - x_j \}
\end{align*}
A picture reveals that $|\Loracle_{nj}(\Loracle_{nj}^\leftarrow(x_j)) - x_j| \le k^{-1}$ for all $x_j \le n/k$. Hence, since $n/k\ge T$ by assumption, we obtain the bound
\[
D_{n2}'(\bm x)
\le 
\sum_{j\in [d]} \Big|\sqrt k\{ S_{nj}(x_j) - x_j \} +  \widetilde {\mathbb L}_{nj}(x_j)  \Big|
\le 
\frac{d}{\sqrt k} + 
\sum_{j\in [d]} \Big| \widetilde {\mathbb L}_{nj}(x_j)  - \widetilde {\mathbb L}_{nj} \circ S_{nj}(x_j)   \Big|.
\]

We now argue as in the proof of Lemma~\ref{lem:dn1-l}: let $\Omega_0$ denote the event of probability at least $1-(d+1)\delta$ on which \eqref{eq:bound-on-order-statistics-1} and \eqref{eq:bound-on-order-statistics-2} are met, and let $\neunzig  \ge 1$ denote the universal constant in \eqref{eq:bound-on-order-statistics-2}. In the case where $\neunzig r\le 2T$, we then have, on $\Omega_0$,
\begin{align*}
\sup_{\bm x \in [0,T]^d} D_{n2}'(\bm x) 
&\le 
\frac{d}{\sqrt{k}} + d  \max_{j \in [d]} \omega_{\widetilde \Lb_{n,j}}(\neunzig r; [0,2T])
=
\frac{d}{\sqrt{k}} + 
    d \max_{j \in [d] }
    \sqrt{\frac{n}{k} }
    \omega_{ \beta_{n,j}}\Big(\frac{k}{n}\neunzig r; [0,2Tk/n]\Big),
\end{align*}
where $r=\sqrt{(T/k)\log(1/\delta)}$ is as in \eqref{eq:definition-r} and where $\beta_{n,j}$ is the $j$th margin of $\beta_n$ from \eqref{eq:betan}. As a consequence, by Lemma \ref{lem:modulus_newnewnew} and the union bound, 
\begin{align*}
\sup_{\bm x \in [0,T]^d} D_{n2}'(\bm x)
&\le 
\frac{d}{\sqrt{k}} + d \kappa  \sqrt{\neunzig r\log\Big(\frac{4T}{\neunzig \delta r}\Big)}
\end{align*}
with probability at least $1-(2d+1)\delta$, where 
\[
\kappa=2\Big[ \sqrt{\frac4{9\neunzig kr} \log\Big( \frac{4T}{\neunzig r\delta}\Big)} + 2 +60 \sqrt{2}  \Big] 
\le 
2\Big[ \frac{2\sqrt{4+\neunzig (2/7)^{1/2}}}{3\neunzig } + 2 + 60\sqrt 2\Big]
\]
and where we used \eqref{eq:bound-kappa} with $d=1$ for the last inequality. We hence find universal constants $D_{2,1}$ and $D_{2,2}$ such that that \eqref{eq:dn2-l} holds with with probability at least $1-(2d+1)\delta$, for the case $\neunzig r\le 2T$.

For the case $\neunzig r > 2T$, note that $\widetilde {\mathbb L}_{nj}(x_j) = \widetilde {\mathbb L}_{nj}(0,\dots,0,x_j,0,\dots,0)$ and thus 
\[
\max_{j \in [d]} \omega_{\widetilde \Lb_{n,j}}(\neunzig r; [0,2T]) \le \omega_{\widetilde \Lb_{n}}(\neunzig r; [0,2T]^d).
\]
Using the bound 
\[
\max_{j \in [d]} \omega_{\widetilde \Lb_{n,j}}(\neunzig r; [0,2T]) \le 2 \max_{j \in [d]} \sup_{x_j \in [0,2T]}\big|\widetilde \Lb_{n,j}(x_j)\big|
\]
and then arguing similarly to the case $\neunzig r > 2T$ in the proof of Lemma~\ref{lem:dn1-l} completes the proof after possibly enlarging $D_{2,1}$ and $D_{2,2}$.
\end{proof}

\begin{proof}[Proof of Lemma~\ref{lem:dn3-l-hoelder}] 
Recall that, for $\bm x \in A$,
\begin{align*}
D_{n3}(\bm x) 
=\sum_{j \in [d]}  \big[ \pdL{j}(\bm x)- \pdL{j}\big(\bm x + t^*(S_n(\bm x) - \bm x)\big) \big]  \widetilde \Lb_{nj}(x_j),
\end{align*}
with $t^*=t^*(n,x)\in[0,1]$.
By Lemma \ref{lem:bound-on-order-statistics}, it holds that $\max_{j\in[d]} \sup_{x_j \in [0,T]} |S_{nj}(x_j) - x_j| \le \neunzig r$ on a set $\Omega_0$ of probability at least $1-(d+1)\delta.$ 
Hence, on this set, the assumption $\neunzig r \le \kappa_L$ 
and \ref{cond:smoothness-hoelder} imply that 
\[
\big| \pdL{j}(\bm x)- \pdL{j}\big(\bm x + t^*(S_n(\bm x) - \bm x) \big|
\le 
\Kl \|\bm x - S_{n}(\bm x)\|^{\alpha_L}_{\infty}  
\le 
\Kl (\neunzig r)^{\alpha_L}
\]
for all $\bm x\in A$. As a consequence, 
\[
 |D_{n3}(\bm x)| 
 \le \Kl (\neunzig r)^{\alpha_L} \sum_{j \in[d]} |\widetilde \Lb_{nj}(x_j)| 
 \le d \Kl(\neunzig r)^{\alpha_L} \max_{j\in [d]} \sup_{x_j \in [0,T]} |\widetilde \Lb_{nj}(x_j) |.
\]
By Lemma \ref{lem:bnd-Ltilde}, with probability at least $1-d\delta$,
\[
\max_{j \in [d]} \sup_{x_j \in [0,T]} |\tilde \Lb_{nj}(x_j) | \le (188/3)  \sqrt{T\log(1/\delta)}
\le \neunzig  \sqrt{T\log(1/\delta)}.
\] 
Combining the previous displays, we find that
\[
\sup_{\bm x \in A} |D_{n3}(\bm x)|  \le \neunzig ^{1+\alpha_L}  \Kl  d r^{\alpha_L} \sqrt{T \log(1/\delta)}
\] 
with probability at least $1-(2d+1)\delta$. Choosing $D_{3} = \neunzig ^{1+\alpha_L} \Kl d$ yields the desired bound. \end{proof}

\begin{proof}[Proof of Lemma~\ref{lem:dn3-l-good}] 
Subsequently, let $\Omega_0$ denote the event of probability at least $1-(d+1)\delta$ on which \eqref{eq:bound-on-order-statistics-1} and \eqref{eq:bound-on-order-statistics-2} are met, and let $\neunzig \approx 89.18$ denote the universal constant in \eqref{eq:bound-on-order-statistics-2}.

Recall that, for $\bm x \in [0,T]^d \setminus(\mathfrak B^{\oplus \neunzig r})$,
\begin{align*}
D_{n3}(\bm x) 
=\sum_{j \in [d]}  \big[ \pdL{j}(\bm x)- \pdL{j}\big(\bm x + t^*(S_n(\bm x) - \bm x)\big) \big]  \widetilde \Lb_{nj}(x_j),
\end{align*}
with $t^*=t^*(n,x)\in[0,1]$. We now distinguish two cases, according to whether $4\neunzig r \le T$ or $4\neunzig r>T$. In the latter case, using that $0 \le \pdL{j}(\cdot )  \le 1$ and Lemma \ref{lem:bnd-Ltilde} (which is applicable since  $\log(1/\delta) \le \log(d/\delta) \le 2kT/7 \le Tk$), we have
\begin{equation*}
\sup_{\bm x \in [0,T]^d \setminus(\mathfrak B^{\oplus \neunzig r})} |D_{n3}(\bm x)| 
\le
d \max_{j\in [d]} \sup_{x_j < T} \big| \widetilde {\mathbb L}_{nj}(x_j) \big|
\le d(188/3) \sqrt{T \log(1/\delta)} = d \neunzig \sqrt{T\log(1/\delta)/2}
\end{equation*}
with probability at least $1-d\delta$. Since $T<4\neunzig r$ and $r \le \sqrt{2/7} \cdot T \le  T$, the upper bound satisfies
\[
d \neunzig  \sqrt{T \log(1/\delta) /2 } 
\le
d \neunzig ^{3/2}  \sqrt{2r \log(1/\delta)} 
\le 
d \neunzig ^{3/2} \sqrt{2r \log\Big(\frac{T}{r\delta} \Big)} 
\le 
\lambda_{n,k,m,T,\Kl}^{(4)},
\]
provided we choose $D_{4}\ge d \neunzig ^{3/2}\sqrt{2}$. Note that we do not need any smoothness assumptions on $L$ here.

It remains to treat the case $4\neunzig r \le T$. 
For each $\bm x \in [0,T]^d \setminus(\mathfrak B^{\oplus \neunzig r})$, we may decompose 
\[
D_{n3}(\bm x) = D_{n3}^0(\bm x) + D_{n3}^+(\bm x) := \sum_{j\in [d]} A_{nj}(\bm x) \bm 1(x_j < 2\neunzig r) + \sum_{j\in [d]} A_{nj}(\bm x) \bm 1(x_j \in [2\neunzig r,T])
\]
where
\[ 
A_{nj}(\bm x) 
:= 
\big[ \pdL{j}(\bm x)- \pdL{j}\big(\bm x + t^*(S_n(\bm x) - \bm x)\big) \big]  \widetilde \Lb_{nj}(x_j).
\] 
We start by bounding $D_{n3}^0(\bm x)$. Again using that $ 0 \le \pdL{j}(\cdot)  \le 1$, we have, for any $j \in [d]$,
\[
 |A_{nj}(\bm x)| \bm 1 (x_j < 2\neunzig r) 
 \le 
 \sup_{0 < x_j < 2\neunzig r} \big|\widetilde {\mathbb L}_{nj}(x_j) \big|.
\]
As a consequence, again by Lemma~\ref{lem:bnd-Ltilde} applied with $T = 2\neunzig r$ and $d=1$, the union bound and the fact that $r \le T$, we have
\begin{equation} \label{eq:bound-dn3-0}
\sup_{\bm x \in [0,T]^d \setminus(\mathfrak B^{\oplus \neunzig r})} |D_{n3}^0(\bm x)|
\le
d \neunzig ^{3/2} \sqrt{r \log(1/\delta)} 
\le
d \neunzig ^{3/2} \sqrt{r\log\Big(\frac{T}{r\delta} \Big)} 
\end{equation} 
with probability at least $1-d\delta$; note that Lemma~\ref{lem:bnd-Ltilde} can be applied with $T = 2\neunzig r$ here because $\log(1/\delta) = r\sqrt{k\log(1/\delta) /T} 
\le \sqrt{2/7}\cdot rk = 
[\sqrt{2/7} / (2\neunzig )] \cdot 2\neunzig r k \le 2\neunzig r k$ by assumption. 

We continue by bounding $\sup_{\bm x \in [0,T]^d \setminus(\mathfrak B^{\oplus \neunzig r})} |D_{n3}^+(\bm x)|$. Again working on the set $\Omega_0$, note that $\bm x \in [0,T]^d \setminus(\mathfrak B^{\oplus \neunzig r})$ implies that  $[\bm x, S_n(\bm x)] \subset G :=[0,T]^d \setminus \mathfrak B$. Further, the condition $x_j \ge 2\neunzig r$ implies that $S_{nj}(x_j) \ge \neunzig r >0$.
As a consequence, we may apply Lemma~\ref{lem:continuity-lprime-good} to obtain the bound
\begin{align*}
    \big| A_{nj}(\bm x) \big|  \bm 1(x_j \in [2\neunzig r,T]) 
    &\le 
    \Kl \max\Big\{\frac1{x_j},\frac1{S_{nj}(x_j)}\Big\} \| S_n(\bm x) - \bm x\|_1 \big| \widetilde \Lb_{nj}(x_j)\big| \bm 1 (x_j \in [2\neunzig r , T ]) \\
    &\le 
    \Kl d \times C_{n1} \times C_{n2},
\end{align*}
where 
\begin{align*}
    C_{n1} 
    &= 
    \max_{\ell \in [d]} \sup_{x_\ell \in [0,T]} \big| S_{n\ell}(x_\ell) - x_\ell \big|, \\
    C_{n2} 
    &= 
    \max_{j \in [d]} \sup_{x_j \in [2\neunzig r,T]} \max \Big\{\frac1{x_j},\frac1 {S_{nj}(x_j)}  \Big\} \big| \widetilde \Lb_{nj}(x_j)\big|,
\end{align*}
which in turn yields
\[
\sup_{\bm x \in [0,T]^d \setminus(\mathfrak B^{\oplus \neunzig r})} |D_{n3}^+(\bm x)|
\le
\Kl d^2 \times C_{n1} \times C_{n2}.
\]
Since we are working on $\Omega_0$, we have  $C_{n1} \le \neunzig r$. Concerning $C_{n2}$, note that for $x_j\ge 2\neunzig r$,
\begin{align*}
S_{nj}(x_j) = x_j \Big(1+\frac{S_{nj}(x_j) - x_j}{x_j} \Big) 
&\ge 
x_j \Big(1- \frac{\max_{\ell \in [d]} \sup_{x_\ell \in [0,T]}|S_{n\ell}(x_j) - x_j| }{2\neunzig r} \Big)
\\&= 
x_j 
\Big(1-\frac{C_{n1}}{2\neunzig r} \Big) 
\ge 
\frac{x_j}{2},
\end{align*}
where we have used that $C_{n1} \le \neunzig r$ on the event $\Omega_0$. As a consequence, with $\beta_{nj}(u_j)$ the $j$th coordinate of $\beta_n$ from \eqref{eq:betan},
\begin{align*} 
C_{n2} 
\le
2\max_{j \in [d]} \sup_{x_j \in [2\neunzig r,T]} \frac1{x_j} \big| \widetilde \Lb_{nj}(x_j)\big|
&\le 
2 (2\neunzig r)^{-1/2} \max_{j \in [d]} \sup_{x_j \in [2\neunzig r,T]}\frac1{x_j^{1/2}}  \big| \widetilde \Lb_{nj}(x_j)\big|
\\&=
2^{1/2} (\neunzig r)^{-1/2} \max_{j \in [d]} \sup_{x_j \in [2\neunzig r,T]}\frac1{x_j^{1/2}} \Big|\sqrt{\frac nk}\beta_{nj}\Big(\frac k n x_j\Big)\Big| 
\\&= 
2^{1/2} (\neunzig r)^{-1/2} \max_{j \in [d]} \sup_{x_j \in [ 2\neunzig r\frac kn,  T\frac kn]} \frac{|\beta_{nj}(x_j)|}{{x_j}^{1/2}}. 
\end{align*} 
Thus, on $\Omega_0$, we obtain the upper bound 
\[
\sup_{\bm x \in [0,T]^d \setminus(\mathfrak B^{\oplus \neunzig r})} |D_{n3}^+(\bm x)| 
\le 
 \Kl d^2 (2\neunzig r)^{1/2} \max_{j \in [d]} \sup_{x_j \in [ 2\neunzig r\frac kn,  T\frac kn]} \frac{|\beta_{nj}(x_j)|}{{x_j}^{1/2}}.
\]
By Corollary 11.2.1 on page 446 in \cite{ShoWel09} (with $\delta=1/2$ in the notation of that reference; it should also be noted that some considerations show that the result also applies with our definition of $\beta_{nj}$ that is based on `$<$' instead of `$\le$' inside the indicators), which is applicable since  $n/k\ge 2T$ by assumption and since $2\neunzig r\tfrac{k}n / (T \tfrac{k}n) =2\neunzig r/T \le 1/2$ in our current case $4\neunzig r \le T$, we have, for any $\eps >0$,
\begin{equation}
\label{eq:showel_alpha/x-new}
\Prob\Big( \sup_{x_1 \in [ 2\neunzig r\frac kn,  T\frac kn]} \frac{\beta_{n1}(x_1)^{\pm}}{{x_1}^{1/2}} \ge  \eps \Big) 
\le 
6 \log\Big( \frac{T}{2\neunzig r}\Big)\exp\Big(- \gamma_{\pm} \frac{\eps^2}{8} \Big),
\end{equation}
where $a^+ = \max(a,0)$ and $a^- = \max(-a,0)$ for $a \in \R$ and $\gamma_- =1$ and \[
\gamma_+ = \begin{cases}
    \frac12 & \text{if } \eps \le \frac32(2\neunzig kr)^{1/2},\\
    \frac34 \frac{(2\neunzig kr)^{1/2}}{\eps} & \text{if } \eps > \frac32(2\neunzig kr)^{1/2}.
\end{cases} 
\]

We will later show that for $\eps=\lambda/(\Kl d^2(2\neunzig r)^{1/2})$ and our choice of $\lambda$ below it holds that $\eps \le \frac 32 \sqrt{2\neunzig rk}$.
Then, since $\gamma_- = 1 \ge 1/2 = \gamma_+$ and $|a| = a^+ \vee a^-$ for any $a\in \R$, Equation \eqref{eq:showel_alpha/x-new} implies that
 \[
\Prob\Big( \sup_{x_1 \in [2r\frac kn ,T \frac kn]} \frac{|\beta_{n1}(x_1)|}{{x_1}^{1/2}} > \eps \Big) 
\le 12 \log\Big( \frac{T}{2\neunzig r}\Big) \exp\Big(-  \frac{\eps^2}{16} \Big) .
\]
As a result,
\begin{align*}
     \Prob\Big( \Big\{\sup_{\bm x \in [0,T]^d \setminus(\mathfrak B^{\oplus \neunzig r})} |D_{n3}^+(\bm x)| > \lambda \Big\} \cap \Omega_0\Big) 
     \le 
     12 d \log\Big( \frac{T}{2\neunzig r}\Big) \exp\Big(-\frac{\lambda^2}{32\neunzig \Kl ^2d^4 r} \Big)
\end{align*}
which is equal to $d\delta$ if we set 
\begin{align}
\label{eq:lambda-dn5+}
\lambda = 4\sqrt{2\neunzig }\Kl d^2 \sqrt{r \log\Big(\frac{12 \log(T/(2\neunzig r))}{\delta}\Big)}.
\end{align}
Overall,
\begin{align*}
\Prob\Big( \sup_{\bm x \in [0,T]^d \setminus(\mathfrak B^{\oplus \neunzig r})} |D_{n3}^+(\bm x)| > \lambda \Big)
&\le
\Prob\Big( \Big\{\sup_{\bm x \in [0,T]^d \setminus(\mathfrak B^{\oplus \neunzig r})} |D_{n3}^+(\bm x)| > \lambda \Big\} \cap \Omega_0 \Big) + \Prob(\Omega_0^c) 
\\&\le (2d+1)\delta,
\end{align*}
and together with \eqref{eq:bound-dn3-0}, we get
\[
\sup_{\bm x \in [0,T]^d \setminus(\mathfrak B^{\oplus \neunzig r})}|D_{n3}(\bm x)|
\le
d \neunzig ^{3/2}  \sqrt{r\log\Big(\frac{T}{\delta r}\Big)}
+
4\Kl d^2 \sqrt{2\neunzig r \log\Big(\frac{12 \log(T/(2\neunzig r))}{\delta}\Big)}
\]
with probability at least $1-(3d+1)\delta$. Since $\log(x) \le x/e$ for $x\ge 1$ 
\begin{equation} \label{eq:lambda5bnd-new}
\frac{12\log(T/(2\neunzig r))}{\delta} 
\le \frac{6e^{-1} T}{ \delta \neunzig r} \le  \frac{T}{ \delta r},
\end{equation}
we obtain that, with probability at least $1-(3d+1)\delta$,
\[
\sup_{\bm x \in \Wmtl} |D_{n3}(\bm x)|
\le
\Big(d \neunzig ^{3/2} + 4 \Kl d^2 \sqrt{2\neunzig } \Big)  \sqrt{r \log\Big(\frac {T}{\delta r} \Big)},
\] 
which is bounded by $\lambda_{n,k,m,T,\Kl}^{(4)}$
if we choose $D_{4}$ at least as large as the term in round brackets. This yields the claim for the case $4\neunzig r \le T$. The two cases $4\neunzig r \le T$ and $4\neunzig r> T$ can then easily be merged by choosing $D_4$ appropriately.

Finally, we need to show that $\eps = \lambda/(\Kl d^2(2\neunzig r)^{1/2}) \le \frac 32 \sqrt{2\neunzig kr}$ holds for $\lambda$ in \eqref{eq:lambda-dn5+}, provided that $4\neunzig r \le T$. Using \eqref{eq:lambda5bnd-new}, we have
\[
\eps = \frac{\lambda}{\Kl d^2\sqrt{2\neunzig r}} 
= 4\sqrt{ \log\Big(\frac{12 \log(T/(2\neunzig r))}{\delta}\Big)} 
\le 4\sqrt{\log\Big(\frac{6e^{-1}T}{\neunzig \delta r}\Big)}.
\]
Next, using $r=\sqrt{T\log(1/\delta)/k} \ge\sqrt{T/k}$, $\neunzig \ge 1$ and $6   /e \ge 1$, and again using that $\log(x)\le x/e$ for $x \ge 1$, it follows that 
\[
\log\Big(\frac{6e^{-1}T}{\delta \neunzig r}\Big)\le \log\Big(\frac{6e^{-1}\sqrt{Tk}}{\delta}\Big)\le 6e^{-2}\sqrt{Tk}+ \log(1/\delta).
\]
By assumption, we also have $1\le \log(1/\delta) \le   \sqrt{Tk\log(1/\delta)}\sqrt{2/7}$, which yields the upper bound 
\[
6e^{-2}\sqrt{Tk}+ \log(1/\delta) \le \sqrt{Tk \log(1/\delta)}\big(6e^{-2} + \sqrt{2/7}\big).
\]
With $16(6e^{-2} + \sqrt{2/7}) = 21.54\ldots < 22$ we obtain that $\eps^2 \le 22\sqrt{Tk\log(1/\delta)} = 22 rk$, which is bounded by $(9/2) \neunzig kr$ by definition of $\neunzig  \approx 89.18$ in Lemma \ref{lem:bound-on-order-statistics}.
\end{proof}

\subsection{Proofs for Section~\ref{sec:clts}}
\label{sec:proofs-clts}

\begin{proof}[Proof of Theorem~\ref{theo:clt}]
Without loss of generality, we can assume that $\log^5(pn) / k \le 1 $; otherwise, the result is trivial.

The triangle inequality yields
\begin{equation*} 
    d_K(\bm S_n, \bm G_n) \le  d_K(\bm S_n, \bm T_n) + d_K(\bm T_n, \bm G_n).
\end{equation*}
We start by bounding $d_K(\bm S_n , \bm T_n)$. An application of Lemma \ref{lem:ks1} yields, for any $\lambda>0$, 
\begin{align} \label{eq:d-s-t}
d_K(\bm S_n , \bm T_n) 
\le 
\Prob\big(\|\bm S_n-\bm T_n\|_\infty \ge \lambda\big)
    + \sup_{\bm x \in \R^p} \Prob( \bm T_n \le \bm x+\lambda \bm1 ) - \Prob( \bm T_n \le \bm x-\lambda \bm 1).
\end{align}
The first term can be dealt with using Corollary \ref{cor:linearization-hoelder}. Denote by $\lambda=\lambda_{n,k}(\delta)$ the upper bound in Corollary~\ref{cor:linearization-hoelder} for suitable $\delta$ chosen below and for $T=1$; we justify below that the corollary can be applied. With this, we obtain that
\begin{align*}
\Prob\Big(\| \bm S_n - \bm T_n \|_\infty > \lambda \Big) 
=
\Prob\Big( \max_{\bm y \in A}  \big|{\mathbb{L}}_{n}(\bm y) - \widebar{\mathbb{L}}_{n}(\bm y) \big| > \lambda \Big) \le 
|\mathcal I|(6m+5) \delta \le 11|\mathcal I|m\delta.
\end{align*}
Regarding the supremum on the right of \eqref{eq:d-s-t}, we have, by Theorem~\ref{theo:nazarov},
\begin{align}
&\phantom{{}={}} \nonumber \Prob( \bm T_n \le \bm x+\lambda\bm 1 ) - \Prob( \bm T_n \le \bm x-\lambda\bm 1 )
\\&= \nonumber
\Prob( \bm G_n \le \bm x+\lambda\bm 1 ) - \Prob( \bm G_n \le \bm x-\lambda\bm 1 ) + 
\big\{ \Prob( \bm T_n \le \bm x+\lambda\bm 1 ) - \Prob( \bm G_n \le \bm x+\lambda\bm 1 ) \big\} 
\\&\hspace{7cm}\nonumber
+ \big\{ \Prob( \bm G_n \le \bm x-\lambda\bm 1 ) - \Prob( \bm T_n \le \bm x-\lambda\bm 1 ) \big\}
\\&\le\nonumber
\frac{2\lambda}{\sigma_{\min}^2} \big\{ 2+\sqrt{2\log p} \big\} + 2 d_K(\bm T_n, \bm G_n)
\\&\le \label{eq:bound-dK}
\frac{8\lambda}{\sigma_{\min}^2} \sqrt{\log p}  + 2 d_K(\bm T_n, \bm G_n)
\end{align}
where we have used that $p \ge 2$ and that $2/\sqrt{\log(2)}+\sqrt 2 \approx 3.81 \le 4$ at the last inequality. Overall,
\begin{equation} \label{eq:cltXZ}
d_K(\bm S_n, \bm G_n)
\le 
11|\mathcal I|m\delta + \frac{8\lambda_{n,k}(\delta)}{\sigma_{\min}^2} \sqrt{\log(p)}  + 3 d_K(\bm T_n, \bm G_n).
\end{equation}
We proceed by bounding $d_K(\bm T_n, \bm G_n)$. Note that the coordinates of $\bm T_n$ are of the form $\sum_{i=1}^n Y_{i,n,I}(\bm x_I) $, where
\begin{align*}
Y_{i,n,I}(\bm x_I) 
&=
\frac1{\sqrt k} \Big[ \bm 1(\exists j \in I: V_{ij} < kx_j/n) - \Prob(\exists j \in I: V_{ij} < kx_j/n) 
     \\ & \hspace{3cm}- \sum_{j\in I} \pdL{j}_I(\bm x_I) \big\{\bm 1(V_{ij} < kx_j/n) -  kx_j/n \big\} \Big],
\end{align*}
with $\Exp[Y_{i,n,I}(\bm x_I)  ]=0$ and $\sum_{i=1}^n \Exp[|Y_{i,n,I}(\bm x_I)|^2 ]$ equal to one of the diagonal entries of $\Sigma_n$. We are going to apply the CCK-result from Theorem~\ref{theo:clt-new}, and need to check its conditions. The first conditions holds with $b_1=\sigma_{\min}^2$. The second and third condition hold with $B_n = (m+1) (\log 2)^{-1} \sqrt{n/k}$ and $b_2 = 4(1+m)m(\log2)^2$; indeed,
\begin{align}
\sum_{i=1}^n \Exp[|Y_{i,n,I}(\bm x_I) |^{4}] 
&\le  \nonumber
(1+m)^3\frac{n}{k^{3/2}}  \Exp[|Y_{i,n,I}(\bm x_I)|] 
\\&\le \nonumber
2(1+m)^3\frac{1}{k} \Big[ \Lpreasy_{n,I}(\bm x_I) + \sum_{j \in I} x_j \Big]
\\&\le  \label{eq:bound-yni4}
4(1+m)^3m \frac{1}{k}= b_2 B_n^2 \frac{1}{n},
\end{align}
where we used the triangle inequality, the fact that that for a Bernoulli$(p)$ random variable $X$ we have $\Exp|X-p| = 2p(1-p) \le 2p$, and $|\Lpreasy_{n,I}(\bm x_I)| \le \sum_{j \in I} x_j \le m$ by the union bound. Moreover
\[
\sqrt n|Y_{i,n,I}(\bm x_I) | / B_n \le  \sqrt{n/k} (m+1)/B_n = \log(2).
\]
An application of Theorem~\ref{theo:clt-new} then yields
\[
3d_K(\bm T_n , \bm G_n) \le c_1 \Big( \frac{\log^5(pn)}{k} \Big)^{1/4},
\]
for some constant $c_1$ depending on $\sigma_{\min}^2$ and $m$ only.

It remains to bound the first and second term in \eqref{eq:cltXZ}, for which we use
\[
\delta=\frac{1}{m|\mathcal I|}\Big( \frac{\log^5(pn)}{k} \Big)^{1/4}
\]
to balance the first and the last term. Indeed, 
the first term in \eqref{eq:cltXZ} then satisfies
\[
11|\mathcal I|m\delta 
\le 
11 \Big( \frac{\log^5(pn)}{k} \Big)^{1/4}.
\]

Finally, regarding the second summand in \eqref{eq:cltXZ}, we start by justifying the application of Corollary~\ref{cor:linearization-hoelder} with the above choice of $\delta$ and with $T=1$.
First, our assumption $\log^5(pn) / k  \le 1$ from the beginning of the proof implies that $\delta \le 1/(m|\mathcal I|) < 1/e$, while the assumption $\log(m^2 |\mathcal I|k^{1/4} ) \le 2k/7$ yields,
\begin{align*}
\log(m/\delta) 
&= 
\log\Big( \frac{m^2|\mathcal I|k^{1/4}}{\log^{5/4}(pn)}\Big)
\le \log(m^2 |\mathcal I| k^{1/4}) \le 2k/7.
\end{align*}
Finally, the assumption $\log(m |\mathcal I|k^{1/4} ) \le \kappa_L^2 k/ \neunzig^2$ yields 
\[
r = \sqrt{\frac1k\log\Big( \frac1\delta\Big)} 
= 
\sqrt{\frac1k\log\Big( \frac{m|\mathcal I|k^{1/4}}{\log^{5/4}(pn)} \Big)} 
\le 
\sqrt{\frac1k\log( m|\mathcal I|k^{1/4})}
\le \kappa_L/\neunzig.
\]
Overall, all Conditions of Corollary \ref{cor:linearization-hoelder} are met.  

It remains to bound the second summand in \eqref{eq:cltXZ}, which is
\begin{align} \label{eq:bound-on-logp-term}
\frac{8\lambda_{n,k}(\delta)}{\sigma_{\min}^2} \sqrt{\log(p)} 
&= \nonumber
\frac{8}{\sigma_{\min}^2} \sqrt{\log(p)}  \Big\{ \max_{I \in \mathcal I} B_{n,k,T}(L_I;  A_I^{\oplus \kappa_L})+ \frac{m}{\sqrt{k}} 
\\& \hspace{3cm}
+ D_{1} \sqrt{r\log\Big(\frac{D_{2}}{\delta r}\Big)}  + D_3 r^{\alpha_L}\sqrt{ \log\Big(\frac{1}{\delta}\Big)} \Big\}.
\end{align}
First, since $\log(p)/k \le 1$ by our assumption at the beginning of the proof, we have
\[
\sqrt{\frac{\log p}{k}} \le \Big( \frac{\log p}{k} \Big)^{1/4} \le \Big( \frac{\log^5(pn)}{k} \Big)^{1/4}.
\]
Next, with our above choice of $\delta$, we have, using $|\mathcal I| \le p$ and the fact that $pk \ge 2$ implies $\log(mpk) \le C_{1,m}^2 \log(pk)$ with $C_{1,m} = \{1 + \log(m)/\log(2)\}^{1/2}$, 
\[
r = \sqrt{\frac1k \log\Big( \frac{m|\mathcal I|k^{1/4}}{\log^{5/4}(pn)}\Big)} 
\le 
\sqrt{\frac1k \log\big(m p k^{1/4}\big)} \le C_{1,m} \sqrt{\frac{\log( p k)}{k}}
\]
Also, 
\[
\delta = \frac{1}{m|\mathcal I|} \Big( \frac{\log^5(pn)}{k} \Big)^{1/4} \ge \frac{1}{m|\mathcal I|k^{1/4}}
\ge 
\frac{1}{mpk^{1/4}}
\]
and $r \ge k^{-1/2}$ (since $\delta < 1/e$).
Hence, the last two terms in \eqref{eq:bound-on-logp-term} can be bounded as follows: first, 
\begin{align*}
\sqrt{r\log\Big(\frac{D_{2}}{\delta r}\Big)} \sqrt{\log p} 
&\le 
\Big(\frac{C_{1,m}^2\log(pk)}{k} \Big)^{1/4} \sqrt{\log(D_2mpk^{3/4}) \log p} 
\\&\le 
\Big(\frac{C_{1,m}^2\log(pk)}{k} \Big)^{1/4} \sqrt{D_2'\log(pk) \log p} 
\\&\le
(C_{1,m}D_2')^{1/2}
\Big(\frac{\log^5(pk)}{k} \Big)^{1/4}\le
(C_{1,m}D_2')^{1/2}
\Big(\frac{\log^5(pn)}{k} \Big)^{1/4},
\end{align*}
where $D_2'=1+ \log(D_2m)/\log (2)$ only depends on $m$.
Second, 
\begin{align*}
r^{\alpha_L}\sqrt{\log\Big(\frac{1}{\delta }\Big)} \sqrt{\log p} 
&\le 
C_{1,m}^{\alpha_L}\Big(\frac{\log(pk)}{k} \Big)^{\alpha_L/2} \sqrt{\log(mpk^{1/4}) \log p} 
\\&\le 
C_{1,m}^{\alpha_L}\Big(\frac{\log(pk)}{k} \Big)^{\alpha_L/2} \sqrt{C_{1,m}^2\log(pk) \log p} 
\\&\le 
C_{1,m}^{1+\alpha_L}\Big(\frac{\log(pk)}{k} \Big)^{1/4} \sqrt{\log(pk) \log p} 
\\&\le
C_{1,m}^2
\Big(\frac{\log^5(pn)}{k} \Big)^{1/4},
\end{align*}
where we used that $\alpha_L \in [1/2,1]$ and that $\log(pk)/k \le 1$ (which is a consequence of our assumption at the beginning of the proof). Assembling terms starting from \eqref{eq:cltXZ}, we have shown that 
\begin{align*} 
d_K(\bm S_n, \bm G_n)
&\le 
\frac{8}{\sigma_{\min}^2} \sqrt{\log p} \Big(\max_{I \in \mathcal I} B_{n,k,T}(L_I; A_I^{\oplus \kappa_L}) \Big)  \\
& \hspace{2cm} + \Big(c_1 + 11 + 8\frac{m + D_1(C_{1,m}D_2')^{1/2} + D_3 C_{1,m}^2}{\sigma_{\min}^2} \Big) \Big(\frac{\log^5(pn)}{k} \Big)^{1/4},
\end{align*}
which implies the assertion.
\end{proof}

\begin{proof}[Proof of Remark~\ref{rem:variance-bound}]
A generic element of $\Sigma_n$, say the entry at position $(q,q') \in [p]^2$, can be written as 
\begin{align*}
\sigma_{n,I,J}(\bm x_I, \bm x_J) = \Exp[\widebar \Lb_{n,I}(\bm x_I) \widebar \Lb_{n,J}(\bm x_J)]
\end{align*}
for certain $I, J \in \mathcal I$ and $\bm x_I \in A_I, \bm x_J \in A_J$. 
Write
\begin{align*}
Y_{I}(\bm x_I) = \frac{1}{\sqrt k}\Big[ \bm1(J_{I}(\bm x_I)) - \Prob(J_{I}(\bm x_I)) - \sum_{j \in I} \pdL{j}_I(\bm x_I) \big\{\bm 1(J_{j}(x_{I,j})) -  kx_{I,j}/n \big\} \Big],
\end{align*}
where $\bm x_I=(x_{I,j})_{j \in I} \in (0,1]^I$, $J_{I}(\bm x_I) = \{ \exists j \in I: V_{j} < kx_{I,j}/n\}$ and  $J_{j}(x_{I,j})=J_{\{j\}}(x_{I,j})=\{V_{j} < kx_{I,j}/n\}$. We then have
\begin{align}
    \sigma_{n,I,J}(\bm x_I, \bm x_J)
    &=\nonumber
    n \Exp[Y_{I}(\bm x_I) Y_{J}(\bm x_J)]
    \\&=  \nonumber
    \frac{n}k \bigg[ \Prob[J_{I}(\bm x_I) \cap J_{J}(\bm x_J)] - \Prob[J_{I}(\bm x_I)] \Prob[J_{J}(\bm x_J)] 
    \\ & \hspace{1cm} - \nonumber
    \sum_{\ell \in I}\pdL{\ell}_I(\bm x_I) \Big\{ \Prob[J_{\ell}( x_{I,\ell}) \cap J_{J}(\bm x_J)] - \frac{kx_{I,\ell}}n \Prob[J_{J}(\bm x_J)] \Big\}
    \\ & \hspace{1cm} - \nonumber
    \sum_{j \in J}\pdL{j}_J(\bm x_J) \Big\{ \Prob[J_{j}(x_{J,j}) \cap J_{I}(\bm x_I)] - \frac{kx_{J,j}}n \Prob[J_{I}(\bm x_I)] \Big\}
    \\ & \hspace{1cm} + \label{eq:covariance-formula}
    \sum_{\ell \in I, j \in J}\pdL{\ell}_I(\bm x_I)\pdL{j}_J(\bm x_J) \Big\{ \Prob[J_{\ell}(x_{I,\ell}) \cap J_{j}(x_{J,j})] - \frac{k^2x_{I,\ell}x_{J,j}}{n^2} \Big\} \bigg].
\end{align}
The variance is obtained for $I=J$ and $\bm x_I = \bm x_J$, which yields
\begin{align*}
    \sigma_{n,I}^2(\bm x_I)
    &= \frac{n}k \bigg[ \Prob[J_{I}(\bm x_I)] - \Prob[J_{I}(\bm x_I)]^2
    \\ & \hspace{1cm} - 
    2 \sum_{\ell \in I}\pdL{\ell}_I(\bm x_I) \Big\{ \frac{kx_{I,\ell}}n - \frac{kx_{I,\ell}}n \Prob[J_{I}(\bm x_I)] \Big\}
    \\ & \hspace{1cm} + 
    \sum_{\ell \in I}\{ \pdL{\ell}_I(\bm x_I) \}^2 \Big\{ \frac{kx_{I,\ell}}n - \frac{k^2x_{I,\ell}^2}{n^2}  \Big\}
    \\ & \hspace{1cm} + 
    \sum_{j,\ell \in I, j \ne \ell}\pdL{\ell}_I(\bm x_I)\pdL{j}_I(\bm x_I) \Big\{ \Prob[J_{\ell}(x_{I,\ell}) \cap J_{j}(x_{I,j})] - \frac{k^2x_{I,\ell}x_{I,j}}{n^2} \Big\}
    \bigg],
\end{align*}
where we have used that $\Prob[J_{\ell}(x_{I,\ell}) \cap J_{I}(\bm x_I)] =\Prob[J_{\ell}(x_{I,\ell})] = kx_{I,\ell}/n $. As a consequence, 
\begin{align*}
    \sigma_I^2 (\bm x_I) = \lim_{n \to \infty}\sigma_{n,I}^2(\bm x_I)
    &=
    L_I(\bm x_I) - \sum_{\ell \in I}x_{I,\ell}\pdL{\ell}_I(\bm x_I) \{ 2-\pdL{\ell}_I(\bm x_I)\}
    \\ & \hspace{3cm} + 
    2\sum_{j,\ell \in I, j < \ell}\pdL{\ell}_I(\bm x_I)\pdL{j}_I(\bm x_I) R_{\{j, \ell\}}(x_{I,j}, x_{I,\ell}).
\end{align*}
Homogeneity of $L_I$ implies that the directional derivative of $L_I$ in $\bm x_I$ in direction $\bm v=\bm x_I/\| \bm x_I\|_2$ is given by
\[
\partial_{\bm v} L_I(\bm x_I) = \lim_{h \to 0}h^{-1}\{  L_I(\bm x_I+h\bm x_I/\| \bm x_I\|_2) - L_I(\bm x_I)\} = L_I(\bm x_I) / \| \bm x_I\|_2. 
\]
If $L_I$ is differentiable at $\bm x_I$ (a consequence of convexity and existing continuous partial derivatives in neighbourhood of $\bm x_I$; see Lemma~\ref{lem:convex-differentiable}), we obtain that 
\[
L_I(\bm x_I) = \| \bm x_I\|_2 \cdot \partial_{\bm v} L_I(\bm x_I) =  \| \bm x_I\|_2 \cdot \langle \bm v, \nabla L_I(\bm x_I)) = \sum_{\ell \in I} x_{I,\ell} \pdL{\ell}_I(\bm x_I).
\]
As a consequence, we may write
\begin{align*}
\sigma_I^2 (\bm x_I) 
&= 
- \bm x_I^\top \nabla L_I(\bm x_I) + (\nabla L_I(\bm x_I))^\top \mathcal R_I (\nabla L_I(\bm x_I)) 
\\&=
- L_I(\bm x_I) +  (\nabla L_I(\bm x_I))^\top \mathcal R_I (\nabla L_I(\bm x_I)) ,
\end{align*}
where $\mathcal R_I = (R_{j,\ell}(x_{I,j}, x_{I,\ell}))_{j,\ell \in I}$ is a $|I| \times |I|$ matrix, with diagonal entries $R_{j,j}(x_{I,j}, x_{I,j}) = x_{I,j}$. Suppose that $\mathcal R_I$ is positive definite. Then, by the Cauchy-Schwarz-inequality, 
\[
(\nabla L_I(\bm x_I))^\top \mathcal R_I (\nabla L_I(\bm x_I))
\ge 
\frac{(\bm x_I^T \nabla L_I(\bm x_I))^2}{\bm x_I^\top \mathcal R_I^{-1} \bm x_I}
= \frac{L_I^2(\bm x_I)}{\bm x_I^\top \mathcal R_I^{-1} \bm x_I},
\]
which yields
\[
\sigma_I^2 (\bm x_I)  \ge - L_I(\bm x_I) + \frac{L_I^2(\bm x_I)}{\bm x_I^\top \mathcal R_I^{-1} \bm x_I}.
\]
In the bivariate case $I=\{j,\ell\}$ and $\bm x_I=(x_j,x_\ell)$, some tedious but straightforward calculation shows that the right-hand side is equal to 
\[
\frac{r(x_j+x_\ell - r)(x_j-r)(x_\ell-r)}{(x_j+x_\ell-2r)x_j x_\ell}
\]
where $r=R_I(x_j,x_\ell)$ denotes the off-diagonal element of $\mathcal R_I$. Since $0 \le r \le x_j \wedge x_\ell$, the expression is strictly positive if an only if $R_I \notin\{ R_{{\text{ind}}}, R_{\text{pd}}\}$, where $R_{{\text{ind}}} \equiv 0$ and $R_{\text{pd}}(x,y) = x \wedge y$ correspond to tail independence and perfect tail dependence, respectively. 
\end{proof}

The bootstrap consistency result in Theorem~\ref{theo:niceboot-new} will be an immediate consequence of the following proposition, which in turn will follow from a couple of intermediate results stated below.

\begin{proposition}\label{prop:boot1}
Let $L$ be a $d$-variate stable tail dependence function and let $\mathcal I$ and $(A_I)_{I \in \mathcal I}$ be as described in the beginning of Section~\ref{sec:clts}. Assume that there exist $\kappa_L, \Kl \in(0,\infty)$ such that 
\begin{align*}
\forall I \in \mathcal I, \forall j \in I,& \forall \bm x_I \in  A_I^{\oplus \min(1,\kappa_L/2)}, \forall \bm y_I \in [0,\infty)^I  \text{ with } \|\bm x_I - \bm y_I\|_\infty \le \kappa_L: \\
&\pdL{j}_I(\bm x_I), \pdL{j}_I(\bm y_I) \text{ exist and satisfy }
|\pdL{j}_I(\bm x_I)-\pdL{j}_I(\bm y_I)| \le \Kl \|\bm x_I - \bm y_I\|_\infty.
\end{align*}
Assume the conditions (i)--(iii) of Theorem~\ref{theo:clt} are met 
with the condition $\log(m|\mathcal I| k^{1/4}) \le \kappa_L^2k/ \neunzig^2$ replaced by $\log(m|\mathcal I| k^{1/4}) \le \kappa_L^2k/ (8\neunzig^2)$, and with $n/k\ge 2$.
Let 
\[
h<(\min_{I \in \mathcal I} \min_{\bm x_I \in A_I}  \min_{j \in I}\bm x_{I,j}) \wedge (\kappa_L/2).
\]
Then, there exist constants $c_i = c_i(m,\Kl,\sigma_{\mathrm{min}}) \ge 1, i = 1,2$, such that, with probability at least $1-c_1\delta_n$ 
\begin{multline} \label{eq:prop-boot}
d_K( \mathcal L(\bm S_n^* \mid \mathrm{data}), \bm G_n) 
\leq  c_2\delta_n + c_2\log(p+k) \\
\times \Big( h + \sqrt{r_{2,n}} + \frac{r_{2,n}}{\sqrt{h}} + \frac{r_{2,n}^2}{h} 
+ \frac1{h\sqrt{k}} \Big\{ B_{n,k}(L_I ;  A_I^{\oplus\kappa_L}) + \Big[\frac{\log^3(pk)}{k} \Big]^{1/4} \Big\}\Big)
\end{multline}
where $\delta_n := [k^{-1}\log^5(pn)]^{1/4}$ and $r_{2,n} := \sqrt{k^{-1}\log(pk)}$.
\end{proposition}

\begin{proof}[Proof of Theorem~\ref{theo:niceboot-new}] 
The conditions of Proposition~\ref{prop:boot1} are a subset of the conditions of Theorem~\ref{theo:niceboot-new}, whence it suffices to show that the upper bound in \eqref{eq:prop-boot} can be bounded as claimed in the theorem.
Since $n,p \ge 2$ we may assume without loss of generality that $k \ge 2$, which yields $\log(p+k) \le \log(pk) \le \log(pn)$. 
Hence, 
\begin{align*}
h\log(p+k) & \le c_h' [k^{-1}\log (p+k)]^{1/4} \log(p+k) \le c_h' \delta_n,
\\
\frac{r_{n,2}\log(p+k)}{\sqrt{h}} &\le c_h^{-1/2}\frac{k^{-1/2}\log^{1/2}(pk) \log(p+k)}{k^{-1/4}\log^{1/4}(p+k)} \le  c_h^{-1/2}\frac{\log^{5/4}(pk)}{k^{1/4}} \le c_h^{-1/2}\delta_n,
\\
\frac{r_{n,2}^2\log(p+k)}{h} &\le c_h^{-1}\frac{k^{-1}\log(pk)\log(p+k)}{k^{-1/2}\log^{1/2}(p+k)} \le c_h^{-1}\frac{\log^{3/2}(pk)}{k^{1/2}} \le c_h^{-1} \delta_n^2,
\\
\frac{\log(p+k)}{h\sqrt{k}} \Big[\frac{\log^3(pk)}{k} \Big]^{1/4} &\le c_h^{-1}\frac{\log(p+k)\log^{3/4}(pk)}{k^{3/4}k^{-1/2} \log^{1/2}(p+k)} \le
c_h^{-1} \frac{\log^{5/4}(pk)}{k^{1/4}} \le c_h^{-1} \delta_n.
\end{align*}
Finally
\[
\frac{\log(p+k)}{h\sqrt{k}} \le c_h^{-1}\frac{\log(p+k)}{k^{1/2}k^{-1/2} \log^{1/2}(p+k)} = c_h^{-1} \sqrt{\log(p+k)},
\]
so 
\[
\frac1{h\sqrt{k}} B_{n,k}(L_I ;  A_I^{\oplus\kappa_L})\log(p+k) \le c_h^{-1} \sqrt{\log(p+k)} B_{n,k}(L_I ;  A_I^{\oplus\kappa_L}).
\]
Combining the above and noting that we can assume $\delta_n \le 1$ since otherwise the bound is trivial by setting $c_2=1$ completes the proof. \end{proof}

The proof of Proposition~\ref{prop:boot1} and the subsequent lemmas require additional notation. Recall $\bm S_n$ and $\bm S_n^*$ from \eqref{eq:Sn-clt} and \eqref{eq:Sn*-boot}, respectively, and let
\begin{align}\label{eq:Scirc-boot}
\bm S_n^\circ = (\widebar{\mathbb L}^\circ_{n,I}(\bm x_{I,\ell}))_{I \in \mathcal I, \ell \in [p_I]}, \qquad
\widebar{\mathbb L}^\circ_{n,I}(\bm x_I) 
    =\sum_{i=1}^n e_i \Big\{Y_{i,I}(\bm x_I) - \frac1n \sum_{i'=1}^n Y_{i',I}(\bm x_I) \Big\}
\end{align}
which is unobservable.

\begin{proof}[Proof of Proposition~\ref{prop:boot1}] 
Throughout the proof we assume $k^{-1}\log(pk) \le 1$ as the statement is trivial otherwise. 
By Lemma~\ref{lem:genboundboot} we have with probability one
\begin{equation}\label{eq:boot-bound-KS-first}
d_K( \mathcal L(\bm S_n^* \mid \mathrm{data}), \bm G_n)
\lesssim \frac1k + \frac{\Delta \cdot \log(p+k)}{\sigma_{\mathrm{min}}^2} + d_K( \mathcal L(\bm S_n^\circ \mid \mathrm{data}), \bm G_n).
\end{equation}
Set 
\[
\delta := \frac1{m|\mathcal I|} \Big( \frac{\log^5(pn)}{k} \Big)^{1/4}.
\]
In the proof of Theorem~\ref{theo:clt} we verify that the conditions of Corollary~\ref{cor:linearization-hoelder} hold with this choice of $\delta$. Moreover, $n/k\ge 2$ by assumption, and using that $|\mathcal I| \le p$ and $\log(pn) \ge 1$, the assumption $\log(mpk^{1/4}) \le \kappa_L^2/(8 \neunzig^2)$ implies $r=\sqrt{k^{-1}\log(1/\delta)} \le \kappa_L/(2^{3/2} \neunzig)$. Hence all conditions of Lemma~\ref{lem:stochastic-bound-on-si} hold with this choice of $\delta$. The latter lemma shows that, with probability at least $1 - |\Ic|(6m+7)\delta$ 
\[
\Delta \lesssim h + \sqrt{r} + 
\frac{r^2}{h} + \frac{r}{\sqrt{h}}
+ \frac1{h\sqrt{k}} \Big\{ B_{n,k}(L_I ;  A_I^{\oplus\kappa_L}) + \sqrt{r \log\Big(\frac{1}{\delta r}\Big)} \Big\}
\]
where the implicit constant depends on $m$ and $\Kl$ only.

The assumption $p \ge 2$ implies $\log(mpk) \le C_{1,m}^2 \log(pk)$, where $C_{1,m} = \{1 + \log(m)/\log(2)\}^{1/2}$ only depends on $m$. Recalling that $p,n \ge 2$ and $k^{-1}\log(pn) \le 1$, and noting $p \ge |\Ic|$ by definition of $\Ic$, we find
\[
r = \sqrt{\frac1k \log\Big( \frac{m|\mathcal I|k^{1/4}}{\log^{5/4}(pn)}\Big)} 
\le 
\sqrt{\frac1k \log\big(m p k^{1/4}\big)} \le C_{1,m} \sqrt{ \frac{\log(pk)}{k}}
\]
and  
\[
\delta = \frac1{m|\mathcal I|} \Big( \frac{\log^5(pn)}{k} \Big)^{1/4} \ge \frac{1}{m |\Ic| k^{1/4} } \ge \frac{1}{mpk^{1/4}}.
\]
Thus, noting that $r \ge k^{-1/2}$ (this follows from $\delta < e^{-1}$)
\[
r\log\Big(\frac{1}{r\delta}\Big) \le r \log(mpk^{3/4}) 
\le r \log(mpk)
\le C_{1,m}^3 \Big[\frac{\log^3(pk)}{k} \Big]^{1/2}. 
\]
In summary, there exists a universal constant $c_1$ and  constant $c_{2,m}$
depending only on $m$ and $\Kl$ such that, with probability at least $1-c_1\delta_n$, 
\begin{equation}\label{eq:boot-bound-Delta-final}
\Delta \leq c_{2,m} \Big[ h + \sqrt{r_{2,n}} + \frac{r_{2,n}^2}{h} + \frac{r_{2,n}}{\sqrt{h}}
+ \frac1{h\sqrt{k}} \Big\{ B_{n,k}(L_I ;  A_I^{\oplus\kappa_L}) + \Big[\frac{\log^3(pk)}{k} \Big]^{1/4} \Big\} \Big]
\end{equation}
where $r_{2,n} = \sqrt{k^{-1}\log(pk)}$ as defined in the theorem..

To bound $d_K( \mathcal L(\bm S_n^\circ \mid \mathrm{data}), \bm G_n)$ we apply Theorem~3 from \cite{CheCheKat23}. In the proof of Theorem~\ref{theo:clt}, we verified that the conditions of that theorem  are satisfied by $X_i$ in their notation replaced with $\sqrt{n}\bm Y_{i,n}$ in our notation with $\underline{\sigma}^2 = \sigma^2_{\mathrm{min}}$, $B_n = (m+1)(\log 2)^{-1}\sqrt{n/k}$ and $\overline{\sigma}^2 = 4 (\log 2)^2 m (m+1)$. From this we obtain, for constants $c_{3,m}, c_{4,m}$ that depend on $m, \sigma_{\mathrm{min}}$ only, 
\begin{equation}\label{eq:boot-bound-KS}
d_K( \mathcal L(\bm S_n^\circ \mid \mathrm{data}), \bm G_n) \leq c_{3,m}\delta_n
\end{equation}
with probability at least $1-c_{4,m}\delta_n$. Combining the bounds in~\eqref{eq:boot-bound-KS-first}--\eqref{eq:boot-bound-KS} completes the proof. \end{proof}

\begin{lemma}\label{lem:genboundboot}
Recall the definitions of $\bm S_n^*$ and  $\bm S_n^\circ$ from \eqref{eq:Sn*-boot} and \eqref{eq:Scirc-boot}, respectively.

If $p\ge 2$, we have with probability one
\begin{equation}
d_K( \mathcal L(\bm S_n^* \mid \mathrm{data}), \bm G_n)
\lesssim \frac1k + \frac{\Delta \cdot \log(p+k)}{\sigma_{\mathrm{min}}^2} + d_K( \mathcal L(\bm S_n^\circ \mid \mathrm{data}), \bm G_n),   
\end{equation}
where the constant in $\lesssim$ is universal and where 
\begin{align}
\label{eq:def-delta-and-si}
\Delta^2 := \max_{I \in \mathcal I} \max_{\bm x_I \in A_I} \sum_{i=1}^n S_{i,I}^2(\bm x_I), \quad S_{i,I}(\bm x_I) :=  
    \widehat Y_{i,I}(\bm x_I) - Y_{i,I}(\bm x_I)  + \frac1n \sum_{i'=1}^n Y_{i',I}(\bm x_I).
\end{align}
\end{lemma}

\begin{proof}[Proof of Lemma~\ref{lem:genboundboot}] 
By the triangle inequality, we have
\[
d_K( \mathcal L(\bm S_n^* \mid \mathrm{data}), \bm G_n)
\le 
d_K( \mathcal L(\bm S_n^* \mid \mathrm{data}),  \mathcal L(\bm S_n^\circ \mid \mathrm{data})) +
d_K( \mathcal L(\bm S_n^\circ \mid \mathrm{data}), \bm G_n). 
\]
To bound $d_K( \mathcal L(\bm S_n^* \mid \mathrm{data}),  \mathcal L(\bm S_n^\circ \mid \mathrm{data}))$ we will apply Lemma~\ref{lem:ks1} conditionally on the data. Write $\Prob_e$ and $\Exp_e$ for the conditional probability/expectation given the data $(\bm X_1, \dots, \bm X_n)$.
Then, for any $\lambda>0$,
\begin{align*}
d_K( \mathcal L(\bm S_n^* \mid \mathrm{data}),  \mathcal L(\bm S_n^\circ \mid \mathrm{data}))
&\le 
\Prob_e( \| \bm S_n^* - \bm S_n^\circ\|_\infty \ge \lambda) 
\\ & \hspace{2cm} + 
\sup_{\bm x \in \R^p} \Prob_e( \bm S_n^\circ \le \bm x+\lambda \bm1 ) - \Prob_e( \bm S_n^\circ \le \bm x-\lambda \bm 1),
\end{align*}
By the same calculation as in \eqref{eq:bound-dK}  in the proof of Theorem~\ref{theo:clt}, we have
\begin{align*}
&\phantom{{}={}} \Prob_e( \bm S_n^\circ \le \bm x+\lambda\bm 1 ) - \Prob_e( \bm S_n^\circ \le \bm x-\lambda\bm 1 )
\\&=
\Prob( \bm G_n \le \bm x+\lambda\bm 1 ) - \Prob( \bm G_n \le \bm x-\lambda\bm 1 ) + 
\big\{ \Prob_e( \bm S_n^\circ \le \bm x+\lambda\bm 1 ) - \Prob( \bm G_n \le \bm x+\lambda\bm 1 ) \big\} 
\\&\hspace{7cm}
+ \big\{ \Prob( \bm G_n \le \bm x-\lambda\bm 1 ) - \Prob_e( \bm S_n^\circ \le \bm x-\lambda\bm 1 ) \big\}
\\&\le
\frac{8\lambda}{\sigma_{\min}^2} \sqrt{\log p} + 2 d_K(\mathcal L(\bm S_n^\circ \mid \mathrm{data}), \bm G_n)
\end{align*}
where we have used Theorem~\ref{theo:nazarov}. Overall, 
\begin{align}
\label{eq:bound-boot}
d_K( \mathcal L(\bm S_n^* \mid \mathrm{data}), \bm G_n)
\le 
\Prob_e( \| \bm S_n^* - \bm S_n^\circ\|_\infty \ge \lambda)
+
\frac{8\lambda}{\sigma_{\min}^2} \sqrt{\log p}  + 3 d_K( \mathcal L(\bm S_n^\circ \mid \mathrm{data}), \bm G_n),
\end{align}
and it remains to choose $\lambda$ appropriately and to bound the first summand on the right. For that purpose, write
\[
\|\bm S_n^* - \bm S_n^\circ\|_\infty =  \max_{I \in \mathcal I} \max_{\bm x_I \in A_I} |D_{I}(\bm x_I)| ,
\]
where
\begin{align*}
    D_{I}(\bm x_I) := \widebar{\mathbb L}^*_{n,I}(\bm x_I)  - \widebar{\mathbb L}^\circ_{n,I}(\bm x_I) 
    =
    \sum_{i=1}^n e_i S_{i,I}(\bm x_I)
    \end{align*}
with $S_{i,I}(\bm x_I)$ defined in the statement of the lemma. We also let
\[
\Delta_I^2(\bm x_I) 
:= \sum_{i=1}^n S_{i,I}^2(\bm x_I)
\]
and note that $\Delta^2 = \max_{I \in \mathcal I} \max_{\bm x_I \in A_I} \Delta_I^2(\bm x_I)$.

Since the multipliers $e_1,\dots,e_n$ are standard Gaussian, we have
\[
\Prob_e( D_{I}(\bm x_I) \in \cdot) =   \Nc(0, \Delta_I^2(\bm x_I))(\cdot).
\]
For $\eta>0$, let
\[
\lambda = \Exp_e[\max_{I \in \mathcal I} \max_{\bm x_I \in A_I} |D_{I}(\bm x_I)| ] + \eta.
\]
The Borell-TIS inequality \citep[Theorem 2.1.1]{AdlerTaylor2007} then yields
\begin{align*}
\Prob_e\Big(  \max_{I \in \mathcal I} \max_{\bm x_I \in A_I} |D_{I}(\bm x_I)|  > \lambda\Big) 
&=
\Prob_e\Big(  \max_{I \in \mathcal I} \max_{\bm x_I \in A_I} |D_{I}(\bm x_I)|  >  \Exp_e[\max_{I \in \mathcal I} \max_{\bm x_I \in A_I} |D_{I}(\bm x_I)| ] + \eta\Big)
\\&\le 
\exp\Big(- \frac{\eta^2}{2 \max_{I \in \mathcal I} \max_{\bm x_I \in A_I} \Exp_e[|D_{I}(\bm x_I)|^2 ]}\Big)
\\&=
\exp\Big(- \frac{\eta^2}{2 \Delta^2}\Big).
\end{align*}
Moreover, by the inequality at the beginning of Section 2.5 in \cite{Boucheron2013}, we have
\[
\Exp_e\Big[\max_{I \in \mathcal I} \max_{\bm x_I \in A_I} |D_{I}(\bm x_I)| \Big] \le \Delta \sqrt{2\log(2p)} \le 2 \Delta \sqrt{\log p},
\]
where the last inequality follows from $p \ge 2$. Using these bounds and definitions, \eqref{eq:bound-boot} yields
\begin{align*}
d_K( \mathcal L(\bm S_n^* \mid \mathrm{data}), \bm G_n)
& \le \nonumber
\exp\Big(- \frac{\eta^2}{2 \Delta^2}\Big)
+ 
\frac{8}{\sigma_{\min}^2}  \eta\sqrt{\log p}  
+
\frac{16}{\sigma_{\min}^2} \Delta \log p 
\\& \hspace{5cm} + 3 d_K(\mathcal L(\bm S_n^\circ \mid \mathrm{data}), \bm G_n).
\end{align*}
Setting $\eta = \Delta \sqrt{2\log k}$ and noting that $\log k, \log p \leq \log(p+k)$ completes the proof. \end{proof}

The following two lemmas provide bounds on $\sum_{i=1}^n S_{i,I}^2(\bm x_I)$ with $S_{i,I}$ from \eqref{eq:def-delta-and-si}. Note that the first one is non-stochastic.

\begin{lemma}\label{lem:non-stochastic-bound-on-si}
    Let $I\subset[d]$, $\bm x_I \in (0,1]^I$, and $n/k \ge 2$. Assume there exists an $\eps \in(0,1)$ such that on the set $\bar B_\eps(\bm x_I)=\{ \bm y_I \in (0,\infty)^I: \| \bm x_I - \bm y_I \|_\infty \le \eps\}$, all partial derivatives $\pdL{j}_I$ with $j \in I$ exist and are Lipschitz-continuous with constant $\Kl$. Then, for any $0< h < (\min_{j \in I} x_j) \wedge \eps$, we have
    \begin{align} \label{eq:bound-on-si}
    \Delta_I^2(\bm x_I) = \sum_{i=1}^n S_{i,I}^2(\bm x_i)
    \lesssim 
    |I|^2h^2 + \frac{|I|^2}{k} 
    &{} \nonumber
    + \frac{|I|^2}{\sqrt k} \max_{j \in I}\sup_{y_j\in [x_j-h,x_j+h]} \big|\widetilde \Lb_{nj}(y_j)\big|
    \\&{} \nonumber
    + \frac{|I|^4}{k} \max_{j \in I}\sup_{y_j\in [x_j-h,x_j+h]} \big|\widetilde \Lb_{nj}(y_j)\big|^2
    \\&{} \nonumber
    +
    \frac{1}{\sqrt{k}} | \widetilde \Lb_{n,I}(\bm x_I) |
    \\&{} \nonumber
    + \frac{|I|^2}{h^2k}\sup_{\bm y_I \in \bar B_{h}(\bm x_I)} \big|\Lb_{n,I}(\bm y_I) - \widebar \Lb_{n,I}(\bm y_I)\big|^2
    \\&{}
    + \frac{|I|^2}{h^2k} \omega_{\widetilde \Lb_{n,I}}(2h;\bar B_h(\bm x_I))^2.
\end{align}
where the implicit constant in $\lesssim$ depends on $\Kl$ only. 
\end{lemma}

\begin{proof}[Proof of Lemma~\ref{lem:non-stochastic-bound-on-si}] 
We start by introducing the notation
\begin{align} \label{eq:J_iI}
J_{i,I} = \{ \exists j \in I: V_{ij} < kx_j/n \},
\qquad 
\hat J_{i,I} = \{ \exists j \in I: \hat V_{ij} < kx_j/n \},
\end{align}
and note that $\Prob(J_{i,I}) = (k/n) \Lpreasy_{n,I}(\bm x_I)$.
Hence, 
\begin{align*}
S_{i,I}(\bm x_I) 
&\equiv
     \widehat Y_{i,I}(\bm x_I) - Y_{i,I}(\bm x_I) + \frac1n \sum_{i'=1}^n Y_{i',I}(\bm x_I)
    =\frac1{\sqrt k } \big( A_{i,I} - B_{i,I} - C_{i,I} + D_{i,I} \big)
\end{align*}
where
\begin{align*}
    A_{i,I} &= \bm 1(\hat J_{i,I}) - \bm 1(J_{i,I}) \\
    B_{i,I} &= \frac{k}n \Big\{ \Lhat_{n,I}(\bm x_I)  - \Lpreasy_{n,I}(\bm x_I) \Big\} \\
    C_{i,I} &= \sum_{j \in I} 
    \pdLhat{j}_{I}(\bm x_I) \Big\{ \bm 1(\hat V_{ij} < kx_j/n) - kx_j/n \Big\}  - \pdL{j}_{I}(\bm x_I)  \Big\{ \bm 1(V_{ij} < kx_j/n) - kx_j/n \Big\} \\
    D_{i,I} &= \frac1n \widebar \Lb_{n,I}(\bm x_I);
\end{align*}
note that $B_{i,I}$ and $D_{i,I}$ do not depend on $i$. As a consequence, since $(a+b+c+d)^2 \le 4(a^2+b^2+c^2+d^2)$, we obtain that
$\sum_{i=1}^n S_{i,I}^2(\bm x_i) \le 4(A^2+B^2+C^2+D^2)$, where
\[
A^2 = \frac1k\sum_{i=1}^n  A_{i,I}^2, \qquad   
B^2 =\frac{n}k B_{1,I}^2, \qquad
C^2 = \frac1k\sum_{i=1}^n  C_{i,I}^2, \qquad
D^2 = \frac{n}k D_{1,I}^2.
\]
A direct computation yields 
\begin{equation*}
D^2
\le
\frac1{kn} |\widebar \Lb_{n,I}(\bm x_I)|^2
    \le \frac2{kn} | \widetilde \Lb_{n,I}(\bm x_I) |^2 + \frac{2|I|^2}{kn}  \max_{j\in I} | \widetilde \Lb_{nj}(x_j)|^2.    
\end{equation*}
We will further show below that
\begin{align}
\label{eq:sum-Ai}
    A^2
    &\le 
    \frac{|I|}{\sqrt k} \max_{j \in I} | \widetilde\Lb_{nj}(x_j)| + \frac{|I|}k, \\
\label{eq:sum-Bi}
    B^2 
    &\le  
    \frac{3|I|^2}n \max_{j\in I} | \widetilde \Lb_{nj}(x_j)|^2 + \frac{3}n | \widetilde \Lb_{n,I}(\bm x_I) |^2 + \frac{3|I|^2}{kn}, \\
\label{eq:sum-Ci}
    C^2
    &\le 
    2|I|^2\max_{j\in I} \big| \pdLhat{j}_{I}(\bm x_I)  - \pdL{j}_{I}(\bm x_I) \big|^2 + \frac{2|I|^2}{\sqrt k}\max_{j\in I} |\widetilde \Lb_{nj}(x_j)| + \frac{2|I|^2}{k}, 
\end{align}
which in turn implies
\begin{align*}
    \sum_{i=1}^n S_{i,I}^2(\bm x_i)
    \le 
    \frac{4|I|+(8+12/n)|I|^2}{k} 
    &{} 
    + \frac{4|I|+8|I|^2}{\sqrt k} \max_{j\in I} | \widetilde \Lb_{nj}(x_j)|
    \\&{}
    + \frac{(12+8/k)|I|^2}{n} \max_{j\in I} | \widetilde \Lb_{nj}(x_j)|^2
    \\&{} 
    +
    \frac{12+8/k}{n} | \widetilde \Lb_{n,I}(\bm x_I) |^2
    \\&{}
    +
    8|I|^2\max_{j\in I} \big| \pdLhat{j}_{I}(\bm x_I)  - \pdL{j}_{I}(\bm x_I) \big|^2.
\end{align*}
The squared terms involving $| \widetilde \Lb_{nj}(x_j)|^2$ and $ | \widetilde \Lb_{n,I}(\bm x_I) |^2$ can be absorbed into the non-squared ones by using the trivial bounds $| \widetilde \Lb_{nj}(x_j)| \le n/\sqrt k$ and $ | \widetilde \Lb_{n,I}(\bm x_I) | \le n/\sqrt k$.
Further, it follows from Lemma~\ref{lem:chi_derivative} that
\begin{align*}
\big| \pdLhat{j}_{I}(\bm x_I)-  \pdL{j}_{I}(\bm x_I) \big|^2    
& \le
4\Kl^2 h^2 
+ \frac4{h^2k}\sup_{\bm y_I \in \bar B_{h}(\bm x_I)} \big|\Lb_{n,I}(\bm y_I) - \widebar \Lb_{n,I}(\bm y_I)\big|^2 
\\
&+ 4\Kl^2 \frac{|I|^2}{k} \max_{j \in I}\sup_{y_j\in [x_j-h,x_j+h]} \big|\widetilde \Lb_{nj}(y_j)\big|^2
\\
&+ \frac4{h^2k} \omega_{\widetilde \Lb_{n,I}}(2h;\bar B_h(\bm x_I))^2.
\end{align*}
Assembling terms we find the claimed bound in the formulation of the lemma.

It remains to show \eqref{eq:sum-Ai}-\eqref{eq:sum-Ci}.
We start by showing \eqref{eq:sum-Ai}.
For that purpose, note that 
\begin{align*}
\phantom{{}={}} 
\big| \bm 1(\hat J_{i,I}) - \bm 1(J_{i,I}) \big| 
&\le 
\sum_{j\in I} \big|\bm 1(\hat V_{ij} < kx_j/n) - \bm 1(V_{ij} < kx_j/n)\big|. 
\end{align*}
Subsequently, we fix $j \in I$.
By definition of $\hat V_{ij}$, we have $\hat V_{ij} < kx_j/n$ if and only if $R_{ij}>n+1-kx_j$, which in turn is equivalent to $V_{ij}<V_{\lceil kx_j \rceil:n,j}$, as shown at the beginning of the proof of Theorem~\ref{theo:linearization-good}. Hence, depending on whether $V_{\lceil kx_j \rceil:n,j} < kx_j/n$ or not, we either have `$\{\hat V_{ij} < kx_j/n\} \subset \{V_{ij} < kx_j/n\}$ for all $i\in[n]$' or `$\{V_{ij} < kx_j/n\} \subset \{\hat V_{ij} < kx_j/n\}$ for all $i\in [n]$'.
It follows that all differences $\bm 1(\hat V_{ij} < kx_j/n) -\bm 1(V_{ij}< k/n)$ with $i \in [n]$ have the same sign, and we can rewrite 
\begin{align} \label{eq:sum-Ai-2}
\sum_{i=1}^n \big|\bm 1(\hat V_{ij} < kx_j/n) -\bm 1(V_{ij} < kx_j/n)\big| 
&= \nonumber
\Big|\sum_{i=1}^n \bm 1(\hat V_{ij} < kx_j/n) -\bm 1(V_{ij} < kx_j/n)\Big| 
\\&=\nonumber
\Big|\sum_{i=1}^n \bm 1(R_{ij} > n+1-\lceil kx_j \rceil ) -\bm 1(V_{ij} < kx_j/n)\Big| 
\\&=\nonumber
\Big|(\lceil kx_j \rceil -1) - k \Loracle_{nj}(x_j)\Big| 
\\&\le \nonumber
k |\Loracle_{nj}(x_j) - x_j| +  \big|(\lceil kx_j \rceil -1) - kx_j\big|
\\&\le
\sqrt k |\widetilde \Lb_{nj}(x_j) | + 1.
\end{align}
The previous two displays yield \eqref{eq:sum-Ai}.

We next show \eqref{eq:sum-Bi}. Note that
\begin{align*}
    B_{1,I} \nonumber
    = 
    \frac{k}n \Big\{ \Lhat_{n,I}(\bm x_I)  - \Lpreasy_{n,I}(\bm x_I) \Big\}
    &= 
        \frac{k}n \Big\{ \Lhat_{n,I}(\bm x_I)  - \Loracle_{n,I}(\bm x_I)  + \Loracle_{n,I}(\bm x_I)  - \Lpreasy_{n,I}(\bm x_I) \Big\}
 \\
    &= \frac{k}n \Big\{ \Lhat_{n,I}(\bm x_I)  - \Loracle_{n,I}(\bm x_I) \Big\} + \frac{\sqrt{k}}{n}\widetilde \Lb_{n,I}(\bm x_I).
\end{align*} 
By the triangle inequality, we have
\begin{align}
    \big|\Lhat_{n,I}(\bm x_I)  - \Loracle_{n,I}(\bm x_I) \big| 
    \le \frac 1k \sum_{i=1}^n \big| \bm 1(\hat J_{i,I}) - \bm 1(J_{i,I}) \big| 
    &\le \label{eq:bound-for-b}
    \frac{|I|}{\sqrt k} \max_{j \in I} | \widetilde\Lb_{nj}(x_j)| + \frac{|I|}{k}
\end{align}
where we used \eqref{eq:sum-Ai} at the last inequality. The claimed identity in \eqref{eq:sum-Bi} then follows from combining the previous two displays and the inequality  $(a+b+c)^2 \le 3(a^2+b^2+c^2)$.

We next show \eqref{eq:sum-Ci}, and for that purpose, note that $C_{i,I}=\sum_{j \in I} C_{i,I,j}$, where 
\begin{align*}
    C_{i,I,j} &\equiv  
    \pdLhat{j}_{I}(\bm x_I) \Big\{ \bm 1(\hat V_{ij} < kx_j/n) - kx_j/n \Big\}  - \pdL{j}_{I}(\bm x_I)  \Big\{ \bm 1(V_{ij} < kx_j/n) - kx_j/n \Big\} 
    \\&=
    \Big\{ \pdLhat{j}_{I}(\bm x_I)-  \pdL{j}_{I}(\bm x_I) \Big\}  \Big\{ \bm 1(\hat V_{ij} < kx_j/n) - kx_j/n \Big\}  \\
    &\hspace{5cm} + \pdL{j}_{I}(\bm x_I)  \Big\{\bm 1(\hat V_{ij} < kx_j/n) -  \bm 1(V_{ij} < kx_j/n) \Big\} .
\end{align*}
Next,
\begin{align*}
\frac1k \sum_{i=1}^n  \Big| \bm 1(\hat V_{ij} < kx_j/n) - kx_j/n \Big|^2
&=
\frac1k \Big\{ (1-2kx_j/n) \Big( \sum_{i=1}^n  \bm 1(\hat V_{ij} < kx_j/n) \Big) + k^2x_j^2/n \Big\}
\\&=
\frac1k \Big\{(1-2kx_j/n)(\lceil kx_j \rceil -1) + k^2x_j^2/n \Big\}
\\& \le x_j (1-kx_j/n) \le x_j \le 1. 
\end{align*}
where we used the assumption that $x_j \le 1 \le n/(2k)$ and the fact that $(\lceil kx_j \rceil -1) \le kx_j$. 
As a consequence, since $0 \le \pdL{j}(\bm x_I) \le 1$ and $(a+b)^2 \le 2(a^2+b^2)$, we obtain the bound
\begin{align*}
\frac1k \sum_{i=1}^n C_{i,I,j}^2 
&\le 
2 \big| \pdLhat{j}_{I}(\bm x_I)-  \pdL{j}_{I}(\bm x_I) \big|^2 +  \frac2k \sum_{i=1}^n \big| \bm 1(\hat V_{ij} \le kx_j/n) -  \bm 1(V_{ij} \le kx_j/n) \big|^2 
\\&\le
2 \big| \pdLhat{j}_{I}(\bm x_I)-  \pdL{j}_{I}(\bm x_I) \big|^2 + \frac2{\sqrt k} | \widetilde \Lb_{nj}(x_j)| + \frac2k,
\end{align*}
where the last bound follows from \eqref{eq:sum-Ai-2}. This inequality, combined with 
\[
\frac1k \sum_{i=1}^n C_{i,I}^2 
\le
\frac1k \sum_{i=1}^n |I| \sum_{j \in I} C_{i,I,j}^2
\le 
|I|^2 \max_{j \in I} \frac1k \sum_{i=1}^n C_{i,I,j}^2 
\]
yields \eqref{eq:sum-Ci}. \end{proof}

\begin{lemma}\label{lem:stochastic-bound-on-si}
Let $L$ be a $d$-variate stable tail dependence function. Let $\mathcal I$ be a collection of index sets $I \subset [d]$ with $|I| \ge 2$, and write $m=\max_{I \in \mathcal I} |I|$. Let $(A_I)_{I \in \mathcal I}$ be a collection of sets with $A_I \subset (0,1]^I$, and suppose that there exist $\kappa_L, \Kl \in(0,\infty)$ such that
\begin{align*}
\forall I \in \mathcal I, \forall j \in I,& \forall \bm x_I \in {A_I^{\oplus \min(1,\kappa_L/2)}}, \forall \bm y_I \in [0,\infty)^I  \text{ with } \|\bm x_I - \bm y_I\|_\infty \le \kappa_L: \\
&\pdL{j}_I(\bm x_I), \pdL{j}_I(\bm y_I) \text{ exist and satisfy }
|\pdL{j}_I(\bm x_I)-\pdL{j}_I(\bm y_I)| \le \Kl \|\bm x_I - \bm y_I\|_\infty.
\end{align*}
Suppose further that $n\in\N_{\ge 2}, k \in \N, \delta \in (0, e^{-1})$ satisfy $\log(m/\delta)\le 2k/7$, $n/k \ge 2$ and $r = \sqrt{k^{-1} \log(1/\delta)} \le \kappa_L/(2^{3/2} \neunzig)$ with $\neunzig$ from Lemma~\ref{lem:bound-on-order-statistics}.
Then, for any $h$ satisfying 
\[
h < (\min_{I \in \mathcal I} \min_{\bm x_I \in A_I}  \min_{j \in I} x_{I,j}) \wedge (\kappa_L/2),
\]
we have
\[
\Delta = \max_{I \in \mathcal I} \max_{\bm x_I \in A_I} \Delta_I(\bm x_I)
\lesssim 
h + \sqrt{r} + 
\frac{r}{\sqrt{h}}
+
\frac{r^2}{h} 
+ \frac1{h\sqrt{k}} \Big\{ B_{n,k}(L_I ;  A_I^{\oplus\kappa_L}) + \sqrt{r \log\Big(\frac{1}{\delta r}\Big)} \Big\}
\]
with probability at least $1-|\mathcal I|(6m+7)\delta$, where the implicit constant in $\lesssim$ only depends on $m$ and $\Kl$.
\end{lemma}

\begin{proof}[Proof of Lemma~\ref{lem:stochastic-bound-on-si}]
Throughout the proof, $\lesssim$ denotes inequality up to a constant only depending on $m$ and $\Kl$. Fix some $I \in \mathcal I$, and recall that $|I| \le m$. We apply Lemma~\ref{lem:non-stochastic-bound-on-si} with $\eps = (\kappa_L/2) \wedge 1$ and $\bm x_I \in A_I$ to obtain that 
\begin{align} \label{eq:bound-on-si-uniform}
\sup_{\bm x_I \in A_I} \Delta_I^2(\bm x_I)
\lesssim 
    h^2 + \frac{1}{k} 
    + \frac{1}{\sqrt k}\sup_{\bm y_I \in [0,2]^I}\big|\widetilde \Lb_{n,I}(\bm y_I)\big|
    &{} \nonumber
   +  \frac{1}{k}\sup_{\bm y_I \in [0,2]^I}\big|\widetilde \Lb_{n,I}(\bm y_I)\big|^2
    \\&{} \nonumber
    + \frac1{h^2k}\sup_{\bm y_I \in A_I^{\oplus h}} \big|\Lb_{n,I}(\bm y_I) - \widebar \Lb_{n,I}(\bm y_I)\big|^2
    \\&{}
    + \frac1{h^2k} \sup_{\bm x_I \in A_I} \omega_{\widetilde \Lb_{n,I}}(2h;\bar B_h(\bm x_I))^2.
\end{align}
where we have used that, for each $\bm x_I \in A_I \subset (0,1]^I$,
\[
\max \Big(| \widetilde \Lb_{n,I}(\bm x_I) |,  \max_{j \in I}\sup_{y_j\in [x_j-h,x_j+h]} \big|\widetilde \Lb_{nj}(y_j)\big| \Big)
\le 
\sup_{\bm y_I \in [0,2]^I}\big|\widetilde \Lb_{n,I}(\bm y_I)\big|,
\]
(recall that $h < \eps\le 1$). We need to bound each term on the right-hand side of \eqref{eq:bound-on-si-uniform}. First,
by Lemma~\ref{lem:bnd-Ltilde}, we have 
\begin{align} \label{eq:sni-proof-1}
\frac{1}{\sqrt k}\sup_{\bm y_I \in [0,2]^I}\big|\widetilde \Lb_{n,I}(\bm y_I)\big|
\lesssim  \sqrt{\frac2{k} \log\Big(\frac1\delta\Big)} \lesssim  r
\end{align}
on an event $\Omega_{I,1}$ with probability at least $1-\delta$.  Moreover, since $r = \sqrt{k^{-1} \log(1/\delta)}  \le \sqrt{2/7} < 1$ by our assumption $\log(m/\delta) \le 2k/7$, the same upper bound holds true for the squared term $k^{-1}\sup_{\bm y_I \in [0,2]^I}|\widetilde \Lb_{n,I}(\bm y_I)|^2$.

Next, we apply Theorem~\ref{theo:linearization-hoelder} with $T=2$ (note that $n/k \ge 2$ by assumption), $L=L_I$ and $A= A_I^{\oplus h}$; note that $A_I^{\oplus h} \subset {A_I^{\oplus \min(1,\kappa_L/2)}}$ such that $(A_I^{\oplus h}, L_I)$ satisfies \ref{cond:smoothness-hoelder} with $\alpha_L =1$ by our assumption on $L$. Further note that $r(\delta,2,k)$ in Theorem~\ref{theo:linearization-hoelder} is equal to $\sqrt 2 r = \sqrt{2} r(\delta,1,k)$ in our current notation.
We obtain that
\[
\sup_{\bm y_I \in A_I^{\oplus h}} \big|\Lb_{n,I}(\bm x) - \widebar \Lb_{n,I}(\bm x)\big|
\lesssim
B_{n,k}(L_I ;  A_I^{\oplus h+\neunzig \sqrt 2r})
+
\frac{1}{\sqrt{k}}
+
\sqrt{r \log\Big(\frac{1}{\delta r}\Big)} +
r\sqrt{\log\Big(\frac1\delta\Big)}
\]
on an event $\Omega_{I,2}$ with probability at least $1-(6m+5)\delta$. Since $r \le \sqrt{2/7}<1$ as noted earlier, and $\delta < 1/e$, we have
\[
\frac{1}{\sqrt{k}} + r\sqrt{\log\Big(\frac1\delta\Big)} \lesssim \sqrt{r \log\Big(\frac{1}{\delta r}\Big)}. 
\]
Next, since $h + \neunzig  \sqrt 2 r \le \kappa_L/2+\kappa_L/2=\kappa_L$ by assumption, we have
\[
B_{n,k}(L_I ;  A_I^{\oplus h+\neunzig   \sqrt{2} r}) \le B_{n,k}(L_I ;  A_I^{\oplus \kappa_L}).
\]
Overall, 
\begin{align} \label{eq:sni-proof-2}
\frac1{h^2k}\sup_{\bm y_I \in A_I^{\oplus h}} \big|\Lb_{n,I}(\bm x) - \widebar \Lb_{n,I}(\bm x)\big|^2
\lesssim 
\frac1{h^2k} \Big\{ B_{n,k}^2(L_I ;  A_I^{\oplus\kappa_L}) + r \log\Big(\frac{1}{\delta r}\Big) \Big\}.
\end{align}
Next, from Lemma \ref{lem:modulus_newnewnew} we get 
\[
    \omega_{\widetilde \Lb_{n,I}}(2h;\bar B_h(\bm x_I)) 
    =  
    \sqrt{\frac nk} \omega_{\beta_{n,I}}\Big(\frac kn 2h ; \frac kn[\bm x_I - h \bm 1_I,\bm x_I + h \bm 1_I]\Big)
\le \kappa \sqrt{2h \log(2 |I|/\delta)}
\] 
on an event $\Omega_{I,3}$ with probability at least $1-\delta$, where 
\[
\kappa = 2|I|\bigg[\sqrt{ \frac{2}{9kh}\log({2|I|}/{\delta})} + 2 + 60 \sqrt{2|I|}\bigg]
\lesssim \Big(\frac{\log(1/\delta)}{kh}\Big)^{1/2} + 1.
\] 
As a consequence, on $\Omega_{I,3}$,
\begin{align}\label{eq:sni-proof-3}
\frac1{h^2k} \omega_{\widetilde \Lb_{n,I}}(2h;\bar B_h(\bm x_I))^2
\lesssim 
\frac1{kh} \kappa^2 \log(1/\delta) \lesssim
\Big(\frac{\log(1/\delta)}{kh}\Big)^2 + \Big(\frac{\log(1/\delta)}{kh}\Big)
=
\frac{r^4}{h^2} + \frac{r^2}{h}.
\end{align}
Overall, combining \eqref{eq:bound-on-si-uniform} with \eqref{eq:sni-proof-1}, \eqref{eq:sni-proof-2} and \eqref{eq:sni-proof-3} and the fact that $k^{-1/2} \le r$,
we find that, on the event $\Omega_{I,1} \cap \Omega_{I,2} \cap \Omega_{I,3}$,
\[
\sup_{\bm x_I \in A_I} \Delta_I^2(\bm x_I)
\lesssim 
h^2 + r  + \frac{r^2}{h} + 
\frac{r^4}{h^2}
+ \frac1{h^2k} \Big\{ B_{n,k}^2(L_I ;  A_I^{\oplus \kappa_L}) + r \log\Big(\frac{1}{\delta r}\Big) \Big\}.
\]
Moreover, $\Prob(\Omega_{I,1} \cap \Omega_{I,2} \cap \Omega_{I,3}) \ge 1-(6m+7)\delta$. The assertion regarding the maximum over $I \in \mathcal I$ then follows from the union bound.
\end{proof}

\begin{lemma} \label{lem:chi_derivative}
Let $L$ be a $d$-variate stable tail dependence function and let $\bm x \in (0,\infty)^d$. Assume there exists an $\eps >0$ such that on the set $\bar B_\eps(\bm x)=\{ \bm y \in (0,\infty)^d: \| \bm x - \bm y \|_\infty \le \eps\}$, the partial derivatives $\pdL{j}$ exist and are Lipschitz-continuous with constant $\Kl$. Then, for any $0<h< \eps \wedge (\min_{j \in [d]}x_j)$, we have
\begin{align*}
\max_{j \in [d]} \big|\pdLhat{j}(\bm x)  - \pdL{j}(\bm x) \big| 
\le 
\Kl h &+ \frac1{h\sqrt{k}}\sup_{\bm y \in \bar B_{h}(\bm x)} \big|\Lb_{n}(\bm y) - \widebar \Lb_{n}(\bm y)\big| \\
&+ \Kl \frac{d}{\sqrt k} \max_{j \in [d]}\sup_{y_j\in [x_j-h,x_j+h]} \big|\widetilde \Lb_{nj}(y_j)\big|\\
&+ \frac1{h\sqrt k} \omega_{\widetilde \Lb_{n}}(2h;\bar B_h(\bm x)).
\end{align*}
\end{lemma}

\begin{proof} 
Note that $|\min(a,1) - b| \le |a-b|$ for $a \in \R, b \in [0,1]$. Together with the triangle inequality this yields
\begin{align} \label{eq:pd-bound}
    |\pdLhat{j}(\bm x)  - \pdL{j}(\bm x) | 
    &\le 
    \Big|\frac{\Lhat_n(\bm x + h \bm e_j) - L(\bm x + h \bm e_j)}{2h} - \frac{\Lhat_n(\bm x - h \bm e_j) - L(\bm x - h \bm e_j)}{2h}\Big| 
    \nonumber\\ 
    &\hspace{5cm} 
    + \Big|\frac{L(\bm x + h \bm e_j) - L(\bm x - h \bm e_j)}{2h} - \pdL{j}(\bm x)\Big| 
    \nonumber\\ 
    &= 
    \Big|\frac{\Lb_n(\bm x + h \bm e_j) - \Lb_n(\bm x - h \bm e_j)}{2h\sqrt{k}} \Big| + \Big|\frac{L(\bm x + h \bm e_j) - L(\bm x - h \bm e_j)}{2h} - \pdL{j}(\bm x)\Big|.
\end{align}    
We start with the second term on the right hand side. By the mean value theorem, there exists some $t\in (-1,1)$ such that 
\[
    \frac{L(\bm x + h \bm e_j) - L(\bm x - h \bm e_j)}{2h} = \pdL{j}(\bm x + t h \bm e_j).
\]
Using the Lipschitz continuity of $ \pdL{j}$, we obtain
\[
\Big| \frac{L(\bm x + h \bm e_j) - L(\bm x - h \bm e_j)}{2h} - \pdL{j}(\bm x) \Big| 
\le 
\Kl |t| h \le \Kl h. 
\]
For the first term on the right hand side of \eqref{eq:pd-bound}, again using the triangle inequality, we have
\begin{align*}
    &\phantom{{}={}} \big| \Lb_n(\bm x + h \bm e_j) - \Lb_n(\bm x - h \bm e_j) \big| \\
    &\le \big| \Lb_n(\bm x + h \bm e_j) - \widebar\Lb_n(\bm x + h \bm e_j) \big|
    +
    \big| \widebar\Lb_n(\bm x + h \bm e_j) - \widebar\Lb_n(\bm x - h \bm e_j) \big|
    \\&\hspace{8cm} +
    \big| \widebar\Lb_n(\bm x - h \bm e_j) - \Lb_n(\bm x - h \bm e_j) \big|
    \\&
    \le 2\sup_{\bm y \in \bar B_{h}(\bm x)} \big| \Lb_n(\bm y) - \widebar \Lb_n(\bm y) \big|  +  \big| \widebar\Lb_n(\bm x + h \bm e_j) - \widebar\Lb_n(\bm x - h \bm e_j) \big|.
\end{align*}
It remains to show that 
\[
\big| \widebar\Lb_n(\bm x + h \bm e_j) - \widebar\Lb_n(\bm x - h \bm e_j) \big|
\le 
{2\Kl dh} \max_{j \in [d]} \sup_{y_j\in [x_j-h,x_j]} \big|\widetilde \Lb_{nj}(y_j)\big|
+ 2 \omega_{\widetilde \Lb_{n}}(2h;\bar B_{h}(\bm x))
\]
By definition of $\widebar \Lb_{n}$, for any $\bm y, \bm y' \in \bar B_\eps(\bm x)$, we have
\begin{align*}
    &\phantom{{}={}} 
    \big|\widebar \Lb_{n}(\bm y) - \widebar \Lb_{n}(\bm y')\big| \\
    &\le \big|\widetilde \Lb_{n}(\bm y) - \widetilde \Lb_{n}(\bm y')\big| 
    + 
    \sum_{\ell \in [d]} \big|\pdL{\ell}(\bm y)\widetilde \Lb_{n\ell}(y_\ell) -  \pdL{\ell}(\bm y')\widetilde \Lb_{n\ell}(y_\ell')\big| \\
    &\le  
    \big|\widetilde \Lb_{n}(\bm y) - \widetilde \Lb_{n}(\bm y')\big|  
    + 
    \sum_{\ell \in [d]} \Big\{ \big|\pdL{\ell}(\bm y) \big|\times\big|\widetilde \Lb_{n\ell}(y_\ell) - \widetilde \Lb_{n\ell}(y_\ell')\big| 
    \\
    &\hspace{7cm} + \big|\widetilde \Lb_{n\ell}(y_\ell')\big|\times\big|\pdL{\ell}(\bm y) - \pdL{\ell}(\bm y')\big|\Big\}
     \\
    &\le 
     \big|\widetilde \Lb_{n}(\bm y) - \widetilde \Lb_{n}(\bm y')\big| 
     + \sum_{\ell \in [d]} \Big\{\big|\widetilde \Lb_{n\ell}(y_\ell) - \widetilde \Lb_{n\ell}(y_\ell')\big|  + \big|\widetilde \Lb_{n\ell}(y_\ell')\big|\times \Kl \big\|\bm y -\bm y'\big\|_\infty\Big\},
\end{align*} 
where we used $|\pdL{\ell}|\le 1$ and Lipschitz-continuity of the partial derivatives.
For $\bm y = \bm x + h \bm e_j$ and $\bm y' = \bm x - h \bm e_j$, we obtain \[
\big|{\widetilde \Lb}_{n}(\bm y) - {\widetilde \Lb}_{n}(\bm y')\big| \le \omega_{{\widetilde \Lb}_{n}}(2h;\bar B_{h}(\bm x)).
\]
The term $|\widetilde \Lb_{n\ell}(y_\ell) - \widetilde \Lb_{n\ell}(y_\ell')|$ equals zero for $\ell \ne j$ and is bounded by $\omega_{{\widetilde \Lb}_{n}}(2h;\bar B_{h}(\bm x))$ for $\ell=j$. 
Finally, it holds that 
\[
|\widetilde \Lb_{n\ell}(y_\ell')| \le \sup_{y_\ell \in [x_\ell-h, x_\ell]} |\widetilde \Lb_{n\ell}(y_\ell)|
\] 
and $\big\|\bm y -\bm y'\big\|_\infty = 2h.$ Combining the previous results yields the assertion.
\end{proof}

\subsection{Proofs for Section~\ref{sec:applications}}

\subsubsection{Proofs for Section~\ref{subsec:isotropy}}

\begin{proof}[Proof of Theorem~\ref{theo:level-control-isotropy}]
Without loss of generality, we can assume that $\log^5(pn) / k \le 1 $; otherwise, the result is trivial. Under $H(\rho)$, we can rewrite
\begin{align*}
T_n^{(\rho)} 
&=  
\max_{(t,\bm s_1, \bm s_2,\bm s_1',\bm s_2') \in D(\rho)} \Db_{n, (\bm s_1, \bm s_2), (\bm s_1', \bm s_2')}(t) 
\end{align*}
where $D(\rho)$ is from \eqref{eq:Drho} and where
\begin{align*}
    \Db_{n, (\bm s_1, \bm s_2), (\bm s_1', \bm s_2')}(t) 
    &= 
    \Lb_{n,(\bm s_1, \bm s_2)}(1-t,t) - \Lb_{n,(\bm s_1', \bm s_2')}(1-t,t)
\end{align*}
with
\[
\Lb_{n,(\bm s_1, \bm s_2)}(x_1, x_2) =  \sqrt k \{ \hat L_{(\bm s_1, \bm s_2)}(x_1, x_2) -  L_{(\bm s_1, \bm s_2)}(x_1, x_2) \}.
\]
Likewise, we can rewrite
\[
T_n^{(\rho),*}  = \max_{(t,\bm s_1, \bm s_2,\bm s_1',\bm s_2') \in D(\rho)}  \widebar \Db^*_{n, (\bm s_1, \bm s_2), (\bm s_1', \bm s_2')}(t)
\]
where 
\begin{align}
\label{eq:dbstar}
\widebar \Db^*_{n, (\bm s_1, \bm s_2), (\bm s_1', \bm s_2')}(t)
= 
\sum_{i=1}^n e_{i} \big\{ \hat Y_{i,(\bm s_1, \bm s_2)}(1-t,t)  - \hat Y_{i,(\bm s_1', \bm s_2')}(1-t,t)  \big\}
\end{align}
with $\hat Y_{i, (\bm s_1, \bm s_2)}$ from \eqref{eq:hatY-isotropy}. 

Next, recall $D= D(1, \sqrt 2) = D(1) \cup D(\sqrt 2)$ with $p = |D|$, and consider the stacked vectors 
\begin{align*}
    \widetilde{\bm S}_n 
    &= 
    \big(\Db_{n, (\bm s_1, \bm s_2), (\bm s_1', \bm s_2')}(t)\big)_{(t, \bm s_1, \bm s_2, \bm s_1', \bm s_2') \in D} \in \R^p,
\\
    \widetilde{\bm T}_n 
    &= 
    \big(\widebar \Db_{n, (\bm s_1, \bm s_2), (\bm s_1', \bm s_2')}(t)\big)_{(t, \bm s_1, \bm s_2, \bm s_1', \bm s_2') \in D} \in \R^p,
\\
    \widetilde{\bm S}_n^* 
    &= 
    \big(\widebar \Db^*_{n, (\bm s_1, \bm s_2), (\bm s_1', \bm s_2')}(t)\big)_{(t, \bm s_1, \bm s_2, \bm s_1', \bm s_2') \in D} \in \R^p,
\end{align*}
where $\widebar \Db^*_{n, (\bm s_1, \bm s_2), (\bm s_1', \bm s_2')}(t)$ is from \eqref{eq:dbstar} and where
\begin{align} \label{eq:bardbn}
    \widebar \Db_{n, (\bm s_1, \bm s_2), (\bm s_1', \bm s_2')}(t) 
    &= 
    \widebar \Lb_{n,(\bm s_1, \bm s_2)}(1-t,t) - \widebar \Lb_{n,(\bm s_1', \bm s_2')}(1-t,t)
\end{align}
with
\begin{align*}
\widebar \Lb_{n,(\bm s_1, \bm s_2)}(x_1, x_2) 
=
\widetilde \Lb_{n,(\bm s_1, \bm s_2)}(x_1, x_2) 
& - \sum_{j \in [2]}\partial_j L_{(\bm s_1, \bm s_2)}(x_1, x_2) \widetilde \Lb_{n,\bm s_j}(x_j),
\end{align*}
and, with $V_i(\bm s) = 1- F_{\bm s}(X_i(\bm s))$, 
\begin{align*}
\widetilde \Lb_{n,(\bm s_1, \bm s_2)}(x_1, x_2) &=
\frac1{\sqrt k} \sum_{i=1}^n \Big\{ \bm 1\Big(\exists j \in \{1,2\}: V_i(\bm s_j) \le \frac{k}n x_j\Big) \\ 
&\hspace{4cm} - \Prob\Big( \exists j \in \{1,2\}: V_i(\bm s_j) \le \frac{k}n x_j\Big)\Big\},
\\
\widetilde \Lb_{n,\bm s}(x)  &= \frac1{\sqrt k} \sum_{i=1}^n \Big\{ \bm 1\Big( V_{i}(\bm s) \le \frac{k}n x\Big) - \Prob\Big( V_{i}(\bm s) \le \frac{k}n x\Big)\Big\}.
\end{align*}
Further, let 
\[
\widetilde{\bm G}_n \sim \mathcal N_{p}(\bm 0, \Var(\widetilde{\bm T}_n)).
\]
Note that
\[
T_n^{(\rho)} = 
\max_{q \in P(\rho)} \widetilde S_{nq},
\qquad 
T_n^{(\rho),*} = 
\max_{q \in P(\rho)} \widetilde S_{nq}^*
\]
and define
\begin{align*}  
T_n^{(\rho),\gapp} 
=
\max_{q \in P(\rho)}  \widetilde G_{nq}^{(\rho)}, 
\end{align*}
where $P(\rho)$ corresponds to all coordinate indices of $\widetilde {\bm S}_n$ for which ${(t, \bm s_1, \bm s_2, \bm s_1', \bm s_2') \in D(\rho)}$. 
Define bivariate random vectors 
\[
\bm Y_n = (T_n^{(1)}, T_n^{(\sqrt 2)}), \quad 
\bm Y_n^{\gapp} = (T_n^{(1),\gapp}, T_n^{(\sqrt 2),\gapp}), \quad
\bm Y_n^* = (T_n^{(1),*}, T_n^{(\sqrt 2),*}),
\]
and note that 
\begin{align*}  
\Prob\big(Y_{n1} \le x, Y_{n2} \le y\big) 
&= 
\Prob \big( \widetilde {\bm S}_n \le \bm s_{x,y} \big),
\\
\Prob\big( Y_{n1}^{\gapp} \le x, Y_{n2}^{\gapp} \le y) 
&= 
\Prob \big(\widetilde {\bm G}_n \le \bm s_{x,y} \big),
\\
\Prob\big( Y_{n1}^*\le x, Y_{n2}^* \le y \mid \mathrm{data}\big) 
&= 
\Prob \big( \widetilde {\bm S}_n^* \le \bm s_{x,y} \mid \mathrm{data} \big),
\end{align*}  
where $\bm s_{x,y} \in \R^p$ is the vector with coordinates $s_j = x$ for $j \in P(1)$ and $s_j = y$ for $j \in P(\sqrt 2)$. Hence,
\begin{align} \label{eq:bound-K-dist}
d_K\Big( \bm Y_n, \bm Y_n^{\gapp} \Big)
&\le d_K(\widetilde {\bm S}_n, \widetilde {\bm G}_n),  \qquad
d_K\Big( \mathcal L\big( \bm Y_n^* \mid \mathrm{data} \big),  \bm Y_n^{\gapp}  \Big)
\le d_K( \mathcal L(\widetilde {\bm S}_n^*  \mid \mathrm{data} ), \widetilde {\bm G}_n),
\end{align}
which will eventually allow to apply Proposition~\ref{prop:p-value-combination}.

In the following, let 
\begin{align} \label{eq:deltan}
\delta_n = \Big( \frac{\log^5(pn)}{k} \Big)^{1/4} .
\end{align}
We will show below that there exist constants $c_1, c_2, c_3$ only depending on $K_L, \sigma_{\min}^2c_h, c_h'$ such that
\begin{align} \label{eq:Kdist1}
    d_K(\widetilde {\bm S}_n, \widetilde {\bm G}_n)
    &\le 
    c_1  \big[ \delta_n + \sqrt{\log p} B_{n,k}
 \big], \\
 \label{eq:Kdist2}
    d_K( \mathcal L(\widetilde {\bm S}_n^*  \mid \mathrm{data} ), \widetilde {\bm G}_n)
    &\le c_2  \big[\delta_n + \sqrt{\log(p+k)}\,  B_{n,k} \big]
\end{align}
the latter holding with probability at least $1-c_3\delta_n$.
In view of \eqref{eq:bound-K-dist}, an application of Proposition~\ref{prop:p-value-combination} implies that, for some constant $c_0$ depending on $c_1, c_2, c_3$, 
\[
\Big|\Prob( C_n \le \hat q_{n,\alpha}^* ) - \alpha \Big|
\le 
c_0 \big[\delta_n + \sqrt{\log(p+k)}\,  B_{n,k} \big]
\]
as asserted.

It remains to show \eqref{eq:Kdist1} and \eqref{eq:Kdist2}. We start with the former and begin by observing that
\begin{align} \label{eq:bound-dn1}
    \max_{(t,\bm s_1, \bm s_2,\bm s_1',\bm s_2') \in D(1,\sqrt 2)} & \big| \Db_{n, (\bm s_1, \bm s_2), (\bm s_1', \bm s_2')}(t) - \widebar \Db_{n, (\bm s_1, \bm s_2), (\bm s_1', \bm s_2')}(t) \big|
     \nonumber\\&\le 
     2 \max_{(\bm s_1, \bm s_2) \in \mathcal P_d(1, \sqrt 2)} \max_{t \in A} \big|  \Lb_{n,(\bm s_1, \bm s_2)}(1-t,t) - \widebar \Lb_{n,(\bm s_1, \bm s_2)}(1-t,t) \big|.
\end{align}
An application of Theorem~\ref{theo:linearization-hoelder} and the union bound implies that there exist constants $D_1=D_1(K_L)$ and $D_2$ such that, for $\delta>0$ specified below, 
\begin{align} \label{eq:bound-dn2}
&\max_{(\bm s_1, \bm s_2) \in \mathcal P_d(\rho)} \max_{t \in A} \big|  \Lb_{n,(\bm s_1, \bm s_2)}(1-t,t) - \widebar \Lb_{n,(\bm s_1, \bm s_2)}(1-t,t) \big|
\nonumber \\&\hspace{5cm}\le
B_{n,k} +  \frac2{\sqrt k} + D_1 \sqrt{r \log \Big( \frac{D_2}{\delta r} \Big)} =: \lambda_{n,k}(\delta)
\end{align}
with probability at least $1-17|\mathcal P_d(1, \sqrt 2)| \delta$. Here, we note that the last term in the final bound of~\ref{theo:linearization-hoelder} can be absorbed into the second-to last term at the cost of possibly increasing the constant since $\alpha_L = 1$.

We now proceed as in the proof of Theorem \ref{theo:clt}, and obtain that, for any $\lambda>0$,
\begin{equation*}
    d_K(\widetilde{\bm S}_n,\widetilde{\bm{G}}_n) 
    \le \Prob(\| \widetilde{\bm S}_n -\widetilde{\bm T}_n \|_\infty \ge \lambda) +  \frac{8\lambda}{{\sigma}^2_{\min}}\sqrt{\log p} + 3d_K(\widetilde{\bm T}_n , \widetilde{\bm G}_n);
\end{equation*}
see the derivations in~\eqref{eq:d-s-t} and~\eqref{eq:bound-dK}. 
With $\lambda = \lambda_{n,k}(\delta)$ from \eqref{eq:bound-dn2}, we obtain that, for $\delta$ chosen below,
\begin{equation*}
    d_K(\widetilde{\bm S}_n,\widetilde{\bm{G}}_n) \le 17|\mathcal P_d(1, \sqrt 2)| \delta +  \frac{8\lambda_{n,k}(\delta)}{{\sigma}^2_{\min}}\sqrt{\log p} + 3d_K(\widetilde{\bm T}_n , \widetilde{\bm G}_n). 
\end{equation*} 
We proceed by bounding $d_K(\widetilde{\bm T}_n , \widetilde{\bm G}_n)$. The coordinates of $\widetilde{\bm T}_n$ are of the form 
\begin{align} \label{eq:yitssss}
\sum_{i=1}^n Y_{i,(\bm s_1, \bm s_2)}(1-t,t) - Y_{i,(\bm s_1', \bm s_2')}(1-t,t)
=: 
\sum_{i=1}^n Y_{i,(t,\bm s_1, \bm s_2, \bm s_1', \bm s_2')}
\end{align}
where
\begin{align*}
Y_{i,(\bm s_1, \bm s_2)}(x_1, x_2) 
&= 
\frac1{\sqrt k} \Big[ \bm 1(\exists j \in [2]: V_{i}(\bm s_j) < kx_j/n) - (k/n) L_{(\bm s_1, \bm s_2)}(\bm x)
     \\ & \hspace{3cm}-\sum_{j \in [2]}\pdL{j}_{(\bm s_1, \bm s_2)}(x_1, x_2) \big\{\bm 1( V_{i}(\bm s_j) < kx_j/n) - kx_j/n \big\} \Big].
\end{align*}
Note the resemblance between $Y_{i,(\bm s_1, \bm s_2)}(x_1, x_2)$ and  $Y_{I}(\bm x_I)$ from \eqref{eq:YiI}, which allows us to make use of results in the proof of Theorem \ref{theo:clt}. For instance, using that $(a-b)^4 \le 8 (a^4 + b^4)$, we have
\[
\sum_{i=1}^n \Exp[|Y_{i,(t,\bm s_1, \bm s_2, \bm s_1', \bm s_2')}|^4] \le 16 \max_{(\bm s_1, \bm s_2) \in \mathcal P_d(1, \sqrt 2)}  \max_{t \in A} \sum_{i=1}^n \Exp[|Y_{i,(\bm s_1, \bm s_2)}(1-t, t)|^4 ],
\] 
and the right-hand side has been bounded in the proof of Theorem \ref{theo:clt} by some constant times $n^{-1}$; see \eqref{eq:bound-yni4}. Similar elementary computations bound $|Y_{i,(\bm s_1, \bm s_2)}(x_1, x_2)|$, and the result follows by applying Theorem~\ref{theo:clt-new} to obtain 
$
d_K(\widetilde{\bm T}_n , \widetilde{\bm G}_n) \le \tilde c_1\delta_n,
$ 
where $\tilde c_1$ depends only on $\sigma_{\min}^2$ and where $\delta_n$ is from \eqref{eq:deltan}.

Finally, letting $\delta = |\mathcal P_d(1, \sqrt 2)|^{-1} \delta_n \le 3^{-1} < e^{-1}$, we will verify that the conditions of Theorem~\ref{theo:linearization-hoelder} hold. Specifically, using (ii) and (iii) we see
\begin{align*}
\log(2/\delta) &\le \log (2 |\mathcal P_d(1, \sqrt 2)| k^{1/4}) \le 2k/7, 
\\
r := \sqrt{k^{-1} \log(\delta^{-1})}&\le \sqrt{k^{-1}\log (|\mathcal P_d(1, \sqrt 2)|k^{1/4})} \le \kappa_L/\neunzig.
\end{align*}
Finally, combining the bounds $r \ge k^{-1/2}, |\mathcal P_d(1, \sqrt 2)| \ge 3 \gtrsim D_2$ and $|\mathcal P_d(1, \sqrt 2)| \le p$, elementary computations show that
\[
\sqrt{\log p} \sqrt{r \log\Big(\frac{D_2}{\delta r}\Big)}
\le 
D_2' \delta_n
\]
for some constant $D_2'$ depending on $D_2$. Assembling terms, this implies \eqref{eq:Kdist1}.

It remains to show \eqref{eq:Kdist2}, for which we proceed as in the proof of Theorem~\ref{theo:niceboot-new}, which itself is a consequence of Proposition \ref{prop:boot1}. The proof of Proposition \ref{prop:boot1} is based on Lemmas \ref{lem:genboundboot}-\ref{lem:chi_derivative}. We will now discuss how these lemmas and their proofs can be adapted to the present setting.

Let
\begin{align*}
\widetilde{\bm S}_n^\circ = \big(\widebar\Db^\circ_{n, (\bm s_1, \bm s_2), (\bm s_1', \bm s_2')}(t)\big)_{(t, \bm s_1, \bm s_2, \bm s_1', \bm s_2') \in D} \in \R^p, 
\end{align*}
where
\begin{align*}
    \widebar\Db^\circ_{n, (\bm s_1, \bm s_2), (\bm s_1', \bm s_2')}(t)
    =
    \sum_{i=1}^n e_i \Big\{Y_{i,(t,\bm s_1, \bm s_2, \bm s_1', \bm s_2')} - \frac1n \sum_{i'=1}^n Y_{i',(t,\bm s_1, \bm s_2, \bm s_1', \bm s_2')} \Big\},
\end{align*}
with $Y_{i,(t,\bm s_1, \bm s_2, \bm s_1', \bm s_2')}$ as defined in \eqref{eq:yitssss}. A careful inspection of the proof of Lemma \ref{lem:genboundboot} shows that it continues to holds with $\bm S^*_n, \bm S_n^\circ$ and $\bm G_n$ replaced by the respective tilde-versions. More specifically, we have
\begin{equation*}
d_K( \mathcal L(\widetilde{\bm S}_n^* \mid \mathrm{data}), \widetilde{\bm G}_n)
\lesssim \frac1k + \frac{\widetilde \Delta \cdot \log(p+k)}{\sigma_{\mathrm{min}}^2} + d_K( \mathcal L(\widetilde{\bm S}_n^\circ \mid \mathrm{data}), \widetilde{\bm G}_n),   
\end{equation*} 
where the constant in $\lesssim$ is universal and where 
\begin{align*}
\widetilde \Delta^2 := \max_{(t, \bm s_1, \bm s_2, \bm s_1', \bm s_2') \in D} \widetilde \Delta^2_{t, \bm s_1, \bm s_2, \bm s_1', \bm s_2'}, 
\end{align*}
with
\begin{align*}
\widetilde \Delta^2_{t, \bm s_1, \bm s_2, \bm s_1', \bm s_2'} 
= 
\sum_{i=1}^n \Big\{\hat Y_{i,(t,\bm s_1, \bm s_2, \bm s_1', \bm s_2')} - Y_{i,(t,\bm s_1, \bm s_2, \bm s_1', \bm s_2')} - \frac1n \sum_{i'=1}^n Y_{i',(t,\bm s_1, \bm s_2, \bm s_1', \bm s_2')} \Big\}^2
\end{align*}
and
\[
\hat Y_{i,(t,\bm s_1, \bm s_2, \bm s_1', \bm s_2')} 
= 
\hat Y_{i,(\bm s_1, \bm s_2)}(1-t,t) - \hat Y_{i,(\bm s_1', \bm s_2')}(1-t,t).
\]

Next, note that Lemma~\ref{lem:chi_derivative} can be applied as is, after proper identification of the notation. Moreover, in view of the inequality $(a+b)^2 \le 2(a^2+b^2)$ and the fact that $Y_{i,(\bm s_1, \bm s_2)}(x_1, x_2)$ essentially corresponds to $Y_{I}(\bm x_I)$ from \eqref{eq:YiI}, the bounds in Lemma \ref{lem:non-stochastic-bound-on-si} and Lemma~\ref{lem:stochastic-bound-on-si} continue to hold, with the hidden universal constant multiplied by $4$ and 2, respectively. More specifically, the tilde-version of Lemma~\ref{lem:stochastic-bound-on-si} is as follows:
\[
\widetilde \Delta 
\lesssim 
h + \sqrt{r} + 
\frac{r}{\sqrt{h}}
+
\frac{r^2}{h} 
+ \frac1{h\sqrt{k}} \Big\{ B_{n,k} + \sqrt{r \log\Big(\frac{1}{\delta r}\Big)} \Big\}
\]
with probability at least $1-19 |\mathcal P_d(1, \sqrt 2)| \delta$. 
The rest of the proof follows the arguments in the proofs of Proposition \ref{prop:boot1} and Theorem \ref{theo:niceboot-new}.
\end{proof}

\subsubsection{Proofs for Section~\ref{sec:Mest}}

The main purpose of this section is to prove Theorem~\ref{theo:Mestlin3}. Along the way, we also establish two intermediate results; the following one is useful for proving consistency.

\begin{proposition}\label{prop: bound on hat theta - theta0}
Suppose that the tuple $(L, \{L(\cdot; \theta): \theta \in \Theta\}, \bg, \mu)$ satisfies the following: there exists some $\theta_0 \in \Theta$ such that for every \(\eps>0\), we have that
    \begin{equation*}
        \fmin(\eps) \coloneqq \inf_{\theta \in \Theta :  \| \theta-\theta_0\|_2 \geq \eps} \Big\{ {Q_L(\theta)} - {Q_L(\theta_0)}\Big\}>0,
    \end{equation*}
where the infimum over an empty set is defined to be infinity. Let \(\nearmin>0\). 
Then, for any estimator $\mestimator$ that is a near minimizer of $\theta \mapsto Q_n(\theta)$ in the sense that \(Q_n(\mestimator)-\inf_{\theta \in \Theta} Q_n(\theta)  < \nearmin\), we have 
\begin{equation*}
        \big\| \mestimator -\theta_0 \big\|_2 
        \leq 
        \leftinv{\fmin}\Big(\nearmin + 2 C_g \sup_{\bx \in [0,1]^d} \big|\Lhat_n(\bx)-L(\bx)\big|\Big),
    \end{equation*}
    where $\leftinv{\fmin}$ denotes the generalized inverse of $\fmin$ defined in \eqref{eq:generalized-inverse} and where $C_g$ is from \eqref{eq:defCg}.
\end{proposition}

Note that Proposition~\ref{prop: bound on hat theta - theta0} is formulated in a general, non-stochastic framework that does not put any assumptions on the observations. Such assumptions will be needed to control the order of $\sup_{\bx \in [0,1]^d} \big|\Lhat(\bx)-L(\bx)\big|$ which appears in the upper bound. The proposition also provides a key step in the proof of the following result.

\begin{theorem}\label{theo:Mestlin1}
Suppose that Assumption~\ref{assumption: smoothness} is met.
For $\nearmin>0 $, let \(\mestimator\) be an estimator that satisfies
$
Q_n(\mestimator)-\inf_{\theta \in \Theta} Q_n(\theta) < \nearmin.
$
For $\beta >0$, consider the event 
\begin{equation} \label{eq:Omega1}
\Omega_1(n, \beta) \coloneqq \Big\{\sup_{\bx \in [0,1]^d} k^{-\frac{1}{2}} \abs{\L_{n}(\bx)} \leq \beta \Big\}.
\end{equation}
There exist constants $\tilde{C}_{r1}, \tilde C_{r2}>0$ and $\tilde{C}_\beta, \tilde{C}_\nearmin \in (0,1]$  
only depending on $d,s,q$, the constant $C_g$ from \eqref{eq:defCg}, the three parameters $\thetainterior,\Cholder,\gammaholder$ from Assumption~\ref{assumption: smoothness} and the four constants defined in \eqref{eq:C_partial} and \eqref{eq:lambda-min}
such that, for any $\beta \in (0, \tilde C_\beta)$ and $\nearmin \in (0, \tilde C_\nearmin)$, we have, on the event $\Omega_1(n, \beta)$, 
\begin{equation}\label{eq:linMest1}
        \sqrt{k} \big( \mestimator - \theta_0 \big)  = 2V_{\theta_0}^{-1} J_{\theta_0}^\top \int_{[0,1]^d} \bg(\bx) \Lb_n(\bx) \, \diff \mu(\bx) + \sqrt{k} \bm r_{n,1}(\beta, \nearmin) 
\end{equation}
where \(\norm{\bm r_{n,1}(\beta, \nearmin)}_2^2 \leq \tilde C_{r1}\big( \beta^{2+\gammaholder }+\nearmin\big) \).
Moreover, for any measurable set $A \subset [0,1]^d$ such that $\widebar{\Lb}_n$ is defined on $[0,1]^d \setminus A$ 
\begin{equation*}
\sqrt{k} \big( \mestimator - \theta_0 \big)  = 2V_{\theta_0}^{-1} J_{\theta_0}^\top \int_{[0,1]^d\setminus A} \bg(\bx) \widebar{\Lb}_n(\bx) \,\diff \mu(\bx) + \sqrt{k} \bm r_{n,1}(\beta, \nearmin) + \bm r_{n,2}(A)
\end{equation*}
where 
\[
\|\bm r_{n,2}(A)\|_2 \le \tilde C_{r2} \Big\{ \sup_{\bx \in [0,1]^d\setminus A} \big|\widebar{\Lb}_n(\bx) - {\Lb}_n(\bx) \big| + \sqrt{k}\beta\int_{A} \|\bg(\bx)\|_2 \, \diff \mu(\bx) \Big\}.
\]
\end{theorem}

In the following, we successively prove Proposition~\ref{prop: bound on hat theta - theta0}, Theorem~\ref{theo:Mestlin1} and then Theorem~\ref{theo:Mestlin3}.

\begin{proof}[Proof of Proposition~\ref{prop: bound on hat theta - theta0}] Throughout, we write $Q=Q_L$.
By definition of the generalized inverse, it suffices to prove that
\begin{equation}\label{eq:bound-f}
\fmin\big(\big\| \mestimator-\theta_0 \big\|_2\big) 
< \nearmin +2C_g \sup_{\bx \in [0,1]^d} \big| \Lhat_{n}(\bx)-L(\bx) \big|. 
\end{equation}
Note that, by the definition of $\mestimator$ and  $\nearmin$, 
\begin{align*}
\nearmin & >  Q_n(\mestimator) - Q_n(\theta_0) \\
&=  \big(Q(\mestimator) - Q(\theta_0)\big) 
    - \big(Q(\mestimator) - Q_n(\mestimator)\big) - \big(Q_n(\theta_0) - Q(\theta_0)\big) \\
&\geq  \big(Q(\mestimator) - Q(\theta_0)\big) - \big|Q(\mestimator) - Q_n(\mestimator)\big| - \big|Q_n(\theta_0) - Q(\theta_0)\big|.
\end{align*}
Thus 
\[
\fmin\big(\big\| \hat \theta_n-\theta_0 \big\|_2\big) 
\le 
Q(\mestimator) - Q(\theta_0) < 2 \sup_{\theta \in \Theta} \big| Q_n(\theta) - Q(\theta)\big|+ \nearmin.
\]
For each \(\theta \in \Theta\), the reverse triangle inequality implies that 
\begin{align*}
    &\big| Q_n(\theta)-Q(\theta) \big|_2
    \\
    & = \bigg| \Big\| \int_{[0,1]^d} \bg(\bx)\big( L(\bx;\theta)-\Lhat_n(\bx) \big) \diff \mu(\bx) \Big\|_2 - \Big\| \int_{[0,1]^d} \bg(\bx)\big( L(\bx;\theta)- L(\bx) \big) \diff \mu(\bx) \Big\|_2 \bigg| 
    \\& \leq 
    \Big\| \int_{[0,1]^d} \bg(\bx)\big( L(\bx;\theta)-\Lhat_n(\bx) \big) \diff\mu(\bx) - \int_{[0,1]^d} \bg(\bx)\big( L(\bx;\theta)- L(\bx) \big) \diff \mu(\bx) \Big\|_2 \\
    &= \Big\| \int_{[0,1]^d} \bg(\bx)\big( L(\bx) - \Lhat_n(\bx)\big) \diff \mu(\bx) \Big\|_2. 
\end{align*}
By the triangle inequality for integrals,
\begin{align} \label{eq:qn-q-hoelder}
         \Big\| \int_{[0,1]^d} \bg(\bx)\big( L(\bx) - \Lhat_n(\bx)\big) \diff \mu(\bx)\Big\|_2 
         &\leq 
         \sup_{\bx \in [0,1]^d} \big| \Lhat_n(\bx)-L(\bx) \big| \times \int_{[0,1]^d} \| \bg(\bx) \|_2\, \diff \mu(\bx). 
\end{align}
Combining the last three displayed formulas establishes~\eqref{eq:bound-f} and completes the proof. \end{proof}

\begin{proof}[Proof of Theorem~\ref{theo:Mestlin1}]
Throughout, we write $Q=Q_L$ and utilize the following additional notation
\[
\bpsi := \int_{[0,1]^d } \bg(\bx) L(\bx) \diff\mu(\bx), \quad \widehat \bpsi := \int_{[0,1]^d } \bg(\bx)\Lhat_{n}(\bx) \diff\mu(\bx).
\] 
For a matrix \(A\), let \(\matnorm{A}_2\) denote the spectral norm of $A$, that is, $\matnorm{A}_2$ is largest singular value of \(A\). Further, $\matnorm{A}_1$ is the maximum of the absolute column sums of $A$, while $\matnorm{A}_\infty$ is the maximum of the absolute row sums of $A$; note that $\matnorm{A}_2^2 \leq \matnorm{A}_1 \cdot  \matnorm{A}_\infty$.
For either a vector or a matrix, \(\norm{\cdot}_\infty\) refers to the absolute maximum entry; note that the previous inequality then yields \(\matnorm{A}_2 \leq \sqrt{sq} \|A\|_\infty\) for $A \in \R^{s \times q}$. Further, $\|A\bm b\|_2 \le \matnorm{A}_2 \|\bm b\|_2$  for $A\in\R^{s \times q}$ and $\bm b \in\R^{q}$. 
Finally, if $A$ is a square matrix and $\bm b$ a vector, we have $|\bm b^\top A \bm b| \le \matnorm{A}_2 \| \bm b\|_2^2$.

In what follows, we will without loss of generality assume that $\thetainterior \leq 1$. Moreover, we will choose $\tilde C_\beta$ and $\tilde C_\nearmin$ not larger than $1$, which implies $\beta, \nearmin \le 1$.

Let \(\theta \in B_{\thetainterior}(\theta_0)\) and define \(\Delta_{\theta} \coloneqq  \theta-\theta_0\). Under Assumption~\ref{assumption: smoothness}, we have the Taylor expansion
\begin{align} \label{eq:expansion-qn2}
Q_n^2(\theta)-Q_n^2(\theta_0)  \nonumber
&= 
\left[\nabla Q_n^2(\theta_0)\right]^{\top}\Delta_{\theta}+\frac{1}{2}\Delta_{\theta}^{\top} V_{n, \tilde{\theta}} \Delta_{\theta} \\
&=
\frac{1}{2}\Delta_{\theta}^{\top} V_{\theta_0} \Delta_{\theta}
+ r_{n,1}(\theta) +r_{n,2}(\theta) +r_{n,3}(\theta), 
\end{align}
where \(\tilde{\theta}\) is a convex combination of \(\theta\) and \(\theta_0\) and where
\begin{align*} 
r_{n,1}(\theta) 
&:= 
\left[\nabla Q_n^2(\theta_0)\right]^{\top}\Delta_{\theta}, 
\\ 
r_{n,2}(\theta) 
&:= 
\frac{1}{2}\Delta_{\theta}^\top(V_{n,\theta_0} -  V_{\theta_0})\Delta_{\theta}, 
\\ 
r_{n,3}(\theta) &:= \frac{1}{2}\Delta_{\theta}^{\top} (V_{n, \tilde{\theta}}-V_{n,\theta_0}) \Delta_{\theta} .
\end{align*}
We will show below that, on the event \(\Omega_1(n,\beta)\),
\begin{align}\label{eq:bound-rn}
r_{n,1}(\theta) 
&\le 
C_1 \beta \norm{\Delta_{\theta}}_2,
\qquad
r_{n,2}(\theta) 
\le
C_2  
\beta \norm{\Delta_{\theta}}_2^2,
\qquad
r_{n,3}(\theta) 
\le
C_3(\beta) \norm{\Delta_{\theta}}_2^{2+\gammaholder },
\end{align}
where $C_1=2 \sqrt{sq}C_\partial C_g$, $C_2=s\sqrt q  C_g C_{\partial^2}$, and $C_3(\beta) := 3sq^{3/2}C_\partial C_{\partial^2} + sq\Cholder(C_g \beta + d_{\theta_0})$ with $d_{\theta_0} := \max_{p \in [q]} |\varphi_p(\theta_0) - \psi_p|$; note that $C_3(\beta)$ is increasing in $\beta$ and hence bounded by $C_3(1)$. Note that by using 
$0\le L(\cdot),L(\cdot;\theta)\le d$, we have $d_{\theta_0} \le dC_g$ which is an upper bound that does not depend on $L$. 

Regarding $r_{n,1}(\theta)$, recall that $\theta_0$ is the global minimizer of $\theta \mapsto Q^2(\theta) = \|\bm \varphi(\theta)-\bm \psi \|_2^2$ and so 
\[
0 = \nabla Q^2(\theta_0) = 2 \sum_{p \in [q]} \big(\varphi_p(\theta_0)-\psi_p\big)\nabla \varphi_p(\theta_0).
\]
Thus
\begin{align} \label{eq:nableqn2}
\nabla Q_n^2(\theta_0) 
&= \nonumber
2\sum_{p \in [q]}  (\varphi_p(\theta_0) - \widehat \psi_p)\nabla \varphi_p(\theta_0)
\\
&= 
2\sum_{p \in [q]}  \big( \psi_p - \widehat \psi_p \big)\nabla \varphi_p(\theta_0)
=
-\frac2{\sqrt k} J_{\theta_0}^\top \int_{[0,1]^d} \bg(\bx) \L_n(\bx) \diff \mu(\bx).
\end{align}
As a consequence, on the event \(\Omega_1(n,\beta)\), recalling the definition of $C_g$ and $C_\partial$ in~\eqref{eq:defCg} and \eqref{eq:C_partial}, respectively, we have the bound 
\begin{align*}
    \big| r_{n,1}(\theta)  \big| &= \frac{2}{\sqrt k} \Big| \Big( J_{\theta_0}^\top \int_{[0,1]^d} \bg(\bx) \L_n(\bx) \diff \mu(\bx)\Big)^\top \Delta_{\theta} \Big|  
    \\&\leq 
    \frac{2}{\sqrt k} \Big\| J_{\theta_0}^\top \int_{[0,1]^d} \bg(\bx) \L_n(\bx) \diff \mu(\bx)\Big\|_2 \norm{\Delta_{\theta}}_2 
    \\ &\leq
    \frac{2}{\sqrt k}\matnorm{J_{\theta_0}^\top}_2 \times \Big\|\int_{[0,1]^d} \bg(\bx) \L_n(\bx) \diff \mu(\bx) \Big\|_2 \norm{\Delta_\theta}_2 
    \\ &\leq 
    \frac{2}{\sqrt k} \sqrt {sq}C_\partial C_g \Big(\sup_{\bx \in [0,1]^d} \abs{\L_n(\bx)}\Big) \norm{\Delta_\theta}_2 \\
    & \le 2 \sqrt{sq}C_\partial 
C_g \beta \norm{\Delta_{\theta}}_2,
\end{align*}
as claimed in \eqref{eq:bound-rn}.

Next, regarding $r_{n,2}(\theta)$, note that the \((j, \ell)\)-entry of \(V_{n, \theta_0}-V_{\theta_0}\in \R^{s \times s}\) is given by  
\begin{align*}
[V_{n,\theta_0}-V_{\theta_0}]_{j\ell}  
&= 2 \sum_{p \in [q]} \Big((\varphi_p(\theta_0) - \widehat \psi_p)\partial_{j\ell} \varphi_p(\theta_0) + \partial_j \varphi_p(\theta_0) \partial_\ell \varphi_p(\theta_0)\Big)
\\
&\quad \quad  - 2 \sum_{p \in [q]} \Big((\varphi_p(\theta_0) -\psi_p)\partial_{j\ell} \varphi_p(\theta_0) + \partial_j \varphi_p(\theta_0) \partial_\ell \varphi_p(\theta_0)\Big)
\\ 
&= -2 (\widehat \bpsi- \bpsi)^\top \partial_{j\ell} \bm \varphi(\theta_0)
=
- \frac2{\sqrt k} \Big[\int_{[0,1]^d} \bg(\bx) \L_n(\bx) \diff\mu(\bx)\Big]^\top \partial_{j\ell} \bm \varphi(\theta_0).  
\end{align*}
Hence, on the event \(\Omega_{1}(n, \beta)\)
\begin{equation*}
    \norm{V_{n, \theta_0}-V_{\theta_0}}_\infty  \leq 2 \sqrt q C_g C_{\partial^2} \beta,
\end{equation*}
which in turn implies
\begin{align}\label{eq: V_n - V_theta bound}
\big|r_{n,2}(\theta) \big|
=
\abs{\frac{1}{2} \Delta_{\theta}^\top (V_{n,\theta_0}-V_{\theta_0}) \Delta_{\theta}} 
&\leq 
\frac{1}{2} \matnorm{V_{n, \theta_0}-V_{\theta_0}}_2\norm{\Delta_{\theta}}^2_2
\leq  
s\sqrt q  C_g C_{\partial^2}  
\beta \norm{\Delta_{\theta}}_2^2
\end{align}
as claimed in \eqref{eq:bound-rn}.

Finally, regarding $r_{n,3}(\theta)$, a similar calculation shows that the \((j, \ell)\)-entry of \(V_{n, \tilde{\theta}}- V_{n, \theta_0}\) can be written as
\begin{align*}
[V_{n, \tilde{\theta}}- V_{n, \theta_0}]_{j\ell} 
&=
2 \sum_{p \in [q]} \Big(\partial_j \varphi_p(\tilde \theta) \partial_\ell \varphi_p(\tilde \theta) - \partial_j \varphi_p(\theta_0) \partial_\ell \varphi_p(\theta_0)\Big)
\\
&\quad \quad  + 2 \sum_{p \in [q]} \Big((\varphi_p(\tilde \theta) - \widehat \psi_p)\partial_{j\ell} \varphi_p(\tilde \theta) - (\varphi_p(\theta_0) -\widehat \psi_p)\partial_{j\ell} \varphi_p(\theta_0) \Big).
\end{align*}
First, since $|ab-cd| \le |a||b-d|+|d||a-c|$ and $\tilde \theta \in B_{\thetainterior}(\theta_0)$,
\begin{align*}
        &
        \abs{\partial_j \varphi_p(\tilde \theta) \partial_\ell \varphi_p(\tilde \theta) - \partial_j \varphi_p(\theta_0) \partial_\ell \varphi_p(\theta_0)} \\
        &\leq 
        \abs{\partial_j \varphi_p(\tilde{\theta})} \abs{\partial_\ell \varphi_p(\tilde{\theta})-\partial_\ell \varphi_p(\theta_0)} + \abs{\partial_\ell \varphi_p(\theta_0)} \abs{\partial_j \varphi_p(\tilde{\theta})-\partial_j \varphi_p(\theta_0)} \\
        & \leq 
        C_{\partial}  \Big(\abs{\partial_\ell \varphi_p(\tilde{\theta})-\partial_\ell \varphi_p(\theta_0)} + \abs{\partial_j \varphi_p(\tilde{\theta})-\partial_j \varphi_p(\theta_0)}\Big) \\
        & \leq 2 \sqrt{q}C_{\partial} 
        C_{\partial^2}\|\tilde{\theta}-\theta_0\|_2,
\end{align*}
where we have used that, by the mean value inequality and the fact that the partial derivatives of $\theta \mapsto \partial_j \varphi_m(\theta)$ are bounded by $C_{\partial^2}$ on $B_{\thetainterior}(\theta_0)$,
\begin{align} \label{eq:bound-pds-phi}
\abs{\partial_j \varphi_p(\tilde{\theta})-\partial_{j} \varphi_p(\theta_0)}
&\le \nonumber
\sup_{t \in (0,1)} \Big| \frac{d}{dt} \partial_j \varphi_p(\theta_0 + t(\tilde{\theta}-\theta_0)) \Big|
\\&\le
\sup_{t \in (0,1)} 
\big\|\nabla[\partial_j \varphi_p](\theta_0 + t(\tilde{\theta}-\theta_0))\big\|_2 \big\|\tilde \theta - \theta_0\big\|_2 
\leq 
\sqrt{q}C_{\partial^2} \big\|\tilde \theta - \theta_0\big\|_2.
\end{align}
Second, recalling $d_{\theta_0} := \max_{p \in [q]} |\varphi_p(\theta_0) - \psi_p|$,
\begin{align*}
    &\Big| (\varphi_p(\tilde \theta) - \widehat \psi_p)\partial_{j\ell} \varphi_p(\tilde \theta) - (\varphi_p(\theta_0) -\widehat \psi_p)\partial_{j\ell} \varphi_p(\theta_0) \Big|
    \\
    &\leq 
    \big| \partial_{j\ell}\varphi_p(\tilde{\theta}) \big| 
    \abs{\varphi_p(\tilde{\theta})-\varphi_p(\theta_0)} + 
    \big| \varphi_p(\theta_0)-\widehat \psi_p \big| 
    \abs{\partial_{j\ell} \varphi_p(\tilde \theta) - \partial_{j\ell} \varphi_p(\theta_0)}
    \\&\leq
    \sqrt{q} C_\partial C_{\partial^2} \big\| \tilde{\theta}-\theta_0 \big\|_2
    + \Cholder \big\| \tilde{\theta}-\theta_0 \big\|_2^{\gammaholder }\Big(\big| \varphi_p(\theta_0)- \psi_p \big| + \big| \widehat \psi_p - \psi_p \big|  \Big)
    \\  &\leq 
    [ \sqrt{q} C_\partial C_{\partial^2} + \Cholder(C_g \beta + d_{\theta_0})] \times \norm{\Delta_{\theta}}_2^{\gammaholder },
\end{align*}
where we used that $\|\tilde \theta- \theta_0\|_2 \le \|\theta- \theta_0\|_2=  \norm{\Delta_{\theta}}_2 \le \thetainterior\le 1$, and that
\[
\hat \psi_p - \psi_p = \frac{1}{\sqrt k} \int_{[0,1]^d}g_p(\bm x) \Lb_n(\bm x) \diff \mu(\bm x)
\]
is bounded by $C_g \beta$ on the event $\Omega_1(n, \beta)$, and that  $|\varphi_p(\tilde \theta) - \varphi_p(\theta_0)| \le \sqrt{q} C_{\partial} \big\|\tilde \theta - \theta_0\big\|_2$,
which follows from the same arguments that were used in \eqref{eq:bound-pds-phi}. Combining the bounds so far we obtain
\begin{equation}\label{eq: bound for V_n - V_n theta}
\matnorm{V_{n, \tilde{\theta}}-V_{n,\theta_0}}_2\leq s \norm{V_{n, \tilde{\theta}}-V_{n,\theta_0}}_\infty \leq 2C_3(\beta) \norm{\Delta_{\theta}}_2^{\gammaholder },
\end{equation}
where $C_3(\beta) = 3sq^{3/2}C_\partial C_{\partial^2} + sq\Cholder(C_g \beta + d_{\theta_0})$, which in turn implies
\begin{align*}
\big|r_{n,3}(\theta) \big|
=
\abs{\frac{1}{2} \Delta_{\theta}^\top (V_{n,\tilde \theta}-V_{\theta_0}) \Delta_{\theta}} 
&\leq 
\frac{1}{2} \matnorm{V_{n, \tilde \theta}-V_{\theta_0}}_2\norm{\Delta_{\theta}}^2_2
\leq  
C_3(\beta) \norm{\Delta_{\theta}}_2^{2+\gammaholder }
\end{align*}
as claimed in \eqref{eq:bound-rn}.

Next, we will show that
\begin{equation}\label{eq:bounddiffQn^2}
\forall \theta \in \Theta: \qquad Q_n^2(\hat \theta_n) - Q_n^2(\theta)
<
2dC_g \nearmin =:C_4 \nearmin.
\end{equation} 
For that purpose, note that our assumption on $\mestimator$ yields
$Q_n(\mestimator)-Q_n(\theta) < \nearmin$ for any $\theta\in \Theta$.
Moreover, by a similar calculation as in \eqref{eq:qn-q-hoelder}, we have for any $\theta \in \Theta$ (in particular, for  $\theta=\mestimator$)
\[
0 \le Q_n(\theta) 
\le 
C_g \sup_{\bm x \in [0,1]^d}\Big|  \Lhat_n(\bm x) - L(\bm x;\theta) \Big|
\le 
dC_g,
\]
where we used that $\Lhat(\bm x),L(\bm x;\theta) \le \| \bm x\|_1$.
As a consequence
\begin{equation*}
Q_n^2(\hat \theta_n) - Q_n^2(\theta)
=
\big(Q_n(\hat \theta_n)-Q_n(\theta)\big)\big(Q_n(\hat \theta_n)+Q_n(\theta)\big) 
<
2dC_g \nearmin 
\end{equation*} 
as asserted in \eqref{eq:bounddiffQn^2}.

We will next apply Proposition~\ref{prop: bound on hat theta - theta0}, and for that purpose, we need to check that $\fmin(\eps)>0$ for all $\eps>0$. In fact, for later purposes, we will need a precise lower bound on that function, and more specifically, we will now show that there exists a constant $C_f>0$ depending on $d,C_g,C_V$ and $C_Q$ only such that
\begin{align} \label{eq:fql-quadratic-bound}
    \fmin(\eps) \equiv \inf_{\theta: \|\theta-\theta_0\|_2 \ge \eps} \big\{ Q(\theta) - Q(\theta_0)\big\} \ge C_f (\eps^2 \wedge 1).
\end{align}
Indeed, note that
$
Q(\theta) \le d \int \| \bg \|_2 \, \diff \mu \le dC_g
$ for all $\theta \in \Theta$, 
whence 
\[
Q^2(\theta) - Q^2(\theta_0) 
= 
(Q(\theta) - Q(\theta_0))(Q(\theta) + Q(\theta_0))
\le 
2d C_g (Q(\theta) - Q(\theta_0)).
\]
As a consequence, with $\tilde C_f = (2d C_g)^{-1}$ and with $d_Q$ as defined in Assumption~\ref{assumption: smoothness}(v).
\[
Q(\theta) - Q(\theta_0) \ge \tilde C_f \big\{ Q^2(\theta) - Q^2(\theta_0) \big\} = \tilde C_f d_Q(\theta).
\]
Hence,
\begin{align*}
\fmin(\eps) &\ge \tilde C_f \inf_{\theta: \|\theta-\theta_0\|_2 \ge \eps} d_Q(\theta)
\\&=
\tilde C_f \min\Big\{  \inf_{\theta: \thetainterior \ge \|\theta-\theta_0\|_2 \ge \eps} d_Q(\theta),  \inf_{\theta: \|\theta-\theta_0\|_2 > \thetainterior} d_Q(\theta) \Big\}
\ge 
\tilde C_f \min\Big( \frac{C_V}4 \eps^2, C_Q\Big),
\end{align*}
where we have used Assumption~\ref{assumption: smoothness}(v). This implies \eqref{eq:fql-quadratic-bound} with $C_f = \tilde C_f\min(C_V/4, C_Q)$.

Next, note that $\leftinv{\fmin}(u) \le \sqrt{u/C_f}$ for $0 < u \le C_f$ by~\eqref{eq:fql-quadratic-bound}.
Choosing $\tilde C_\nearmin \le (\thetainterior^2\wedge 1)C_f/2$ and $\tilde C_\beta \le (\thetainterior^2\wedge 1)C_f/(4 C_g)$ ensures $\eta + 2C_g \beta \le (\kappa^2 \wedge 1) C_f \le C_f$ for  $\beta \in (0,\tilde C_\beta)$ and $\nearmin \in (0,\tilde C_\nearmin)$, and so 
\[
\big\| \mestimator-\theta_0\big\|_2 
\leq 
\leftinv{\fmin}\Big(\nearmin+2 C_g \beta \Big)
\leq
\Big(\frac{\nearmin+2 C_g \beta}{C_f} \Big)^{1/2} 
\le 
(\kappa^2 \wedge 1)^{1/2} \le \kappa. 
\]
by Proposition~\ref{prop: bound on hat theta - theta0}.

As a consequence, we can apply \eqref{eq:expansion-qn2} and \eqref{eq:bound-rn} with $\hat \Delta_n=\Delta_{\mestimator}=\mestimator-\theta_0$ to obtain that
\begin{equation} \label{eq:qn2-mest}
    Q_n^2(\mestimator) - Q_n^2(\theta_0) = \frac{1}{2} \hat \Delta_n^{\top} V_{\theta_0} \hat \Delta_n + r_{n,1}(\mestimator) + r_{n,2}(\mestimator) + r_{n,3}(\mestimator),
\end{equation}
with the three error terms satisfying
\begin{align} \label{eq:error-mest}
|r_{n,1}(\mestimator)| \le C_1 \beta \big\| \hat \Delta_n \big\|_2, \qquad
|r_{n,2}(\mestimator)| + |r_{n,3}(\mestimator) |
\leq
C_5(\beta, \nearmin) \big\| \hat \Delta_n \big\|_2^2,
\end{align}
with $C_5(\beta, \nearmin) := C_2\beta +  C_3(\beta) \{ (\nearmin+2 C_g \beta)/{C_f}\}^{\gammaholder/2}$.
Combining \eqref{eq:bounddiffQn^2} (with $\theta=\theta_0$) with \eqref{eq:qn2-mest} and \eqref{eq:error-mest}, we obtain that
\begin{align*}
    C_{4} \nearmin 
    & \ge 
    \frac{1}{2} \hat \Delta_n^{\top} V_{\theta_0} \hat \Delta_n + r_{n,1}(\mestimator) + r_{n,2}(\mestimator) + r_{n,3}(\mestimator) \\
    &> \frac{1}{2} \lambda_{\text{min}}(V_{\theta_0}) \big\| \hat \Delta_n \big\|_2^2 - C_1 \beta \big\| \hat \Delta_n \big\|_2 - C_5(\beta, \nearmin) \big\| \hat \Delta_n \big\|_2^2.
\end{align*}
Decreasing $\tilde C_\beta$ and $\tilde C_\nearmin$ if necessary, we can guarantee that $C_{5}(\beta, \nearmin) \le \lambda_{\min}(V_{\theta_0})/4$ for any $\beta \in (0,\tilde C_\beta)$ and $\nearmin \in (0,\tilde C_\nearmin)$. Hence, 
\begin{equation*}
    \big\| \hat \Delta_n \big\|_2^2 
    < 
    \frac{4}{\lambda_{\min}(V_{\theta_0})} \big( C_4 \nearmin + C_1 \beta \big\|  \hat \Delta_n \big\|_2 \big).
\end{equation*}
For $a,b>0$ and $x \ge 0$, we have that $x^2 \le ax+b$ implies $x \le a+\sqrt b$; indeed, if $x>a+\sqrt b$, we have $x^2 > x(a+\sqrt b) >ax+(a +\sqrt b)\sqrt b > ax+b$. Thus, 
\begin{equation}\label{eq:boundhatdelta1}
\big\| \hat \Delta_n \big\|_2 
\le 
\frac{2\sqrt{C_4 \nearmin}}{\sqrt{\lambda_{\min}(V_{\theta_0})}} + \frac{4C_1 \beta}{\lambda_{\min}(V_{\theta_0})}.
\end{equation}
As a consequence, $\| \hat \Delta_n \|_2^2 \le C_6 \big( \nearmin + \beta^2\big)$ with 
$C_6 = \{8C_4/\lambda_{\min}(V_{\theta_0}) \} \vee \{32C_1^2/\lambda_{\min}^2(V_{\theta_0})\}$, 
which, using \eqref{eq:bound-rn} with $\theta=\hat \theta_n$, yields
\begin{align} \label{eq:bound-errors-hat}
|r_{n,2}(\hat\theta_n)| + |r_{n,3}(\hat\theta_n) |
&\le \nonumber
C_2\beta \big\| \hat \Delta_{n} \big\|_2^{2} + C_3(\beta)\big\| \hat \Delta_{n} \big\|_2^{2+\gammaholder }
\\&\le\nonumber
\Big( C_2C_6 \frac{\beta }{(\nearmin+ \beta^2)^{\gammaholder /2}}+C_3(\beta)C_6^{1+\gammaholder /2} \Big) (\nearmin + \beta^2)^{1+\gammaholder /2}
\\&\le 
C_{7}(\beta) (\nearmin + \beta^2)^{1+\gammaholder /2},
\end{align}
where $C_7(\beta)= C_2C_6\beta^{1-\gammaholder } + C_3(\beta) C_6^{1+\gammaholder /2}$.

Next, let
\[
\widetilde \Delta_n = 2k^{-\frac{1}{2}}V_{\theta_0}^{-1} J_{\theta_0}^{\top} \int_{[0,1]^d} \bg(\bx) \L_n(\bx) \diff\mu(\bx) = - V_{\theta_0}^{-1} \nabla Q_n^2(\theta_0)
\]
where the second equality follows from \eqref{eq:nableqn2}. Note that we need to find $\tilde C_r>0$ such that $\| \hat \Delta_n - \widetilde \Delta_n \|_2^2 \le \tilde C_r(\nearmin + \beta^{2+\gammaholder })$. On $\Omega_1(n, \beta)$, we have
\begin{align} \label{eq:boundhatdelta2}
\big\| \widetilde \Delta_n \big\|_2
\le
2C_g \matnorm{ V_{\theta_0}^{-1} J_{\theta_0}^{\top} }_2 \beta
\le
2C_g \matnorm{ V_{\theta_0}^{-1}}_2 \matnorm{ J_{\theta_0}^{\top} }_2 \beta
\le
2C_g C_V^{-1}\sqrt{sq} C_{\partial} \beta
=: C_{8} \beta,
\end{align}  
where we have used \eqref{eq:qn-q-hoelder}. Further decreasing $\tilde C_\beta$ if necessary, the right hand-side is bounded by $\thetainterior$ for all $\beta \in (0, \tilde C_\beta)$, which implies that $\widetilde \theta_n := \theta_0 + \widetilde \Delta_n  \in B_\thetainterior(\theta_0)$. We can hence apply the expansions and bounds derived at the beginning of this proof, specifically \eqref{eq:expansion-qn2}, with $\theta= \widetilde \theta_n$ and $\Delta_{\widetilde \theta_n}= \widetilde \Delta_n$ to deduce that
\begin{equation} \label{eq:qn2-mest-tilde}
    Q_n^2(\widetilde \theta_n)-Q_n^2(\theta_0) = 
    \frac{1}{2} \widetilde \Delta_{n}^\top V_{\theta_0} \widetilde \Delta_{n}+ r_{n,1}(\widetilde\theta_n) + r_{n,2}(\widetilde\theta_n) + r_{n,3}(\widetilde\theta_n),
\end{equation}
where, using \eqref{eq:bound-rn} and \eqref{eq:boundhatdelta2},
\begin{align}\label{eq:bound-errors-tilde}
|r_{n,2}(\widetilde\theta_n)| + |r_{n,3}(\widetilde\theta_n) |
\le 
C_2\beta \big\| \widetilde \Delta_{n} \big\|_2^{2} + C_3(\beta)\big\| \widetilde \Delta_{n} \big\|_2^{2+\gammaholder }
\le
C_{9}(\beta) \beta^{2+\gammaholder },
\end{align}
where $C_{9}(\beta) = C_{8}^2\{ C_2\beta^{1-\gammaholder } + C_8^{\gammaholder}C_3(\beta)\}$. 
Overall, from \eqref{eq:bounddiffQn^2} applied with $\theta= \widetilde \theta_n$ and \eqref{eq:qn2-mest} and \eqref{eq:qn2-mest-tilde}, we find that
\begin{align*}
    C_{4} \nearmin
    &> 
    Q_n^2(\hat \theta_n)-Q_n^2(\widetilde \theta_n) 
= 
    \big( Q_n^2(\hat \theta_n)-Q_n^2(\theta_0)\big)-\big(Q_n^2(\widetilde \theta_n)-Q_n^2(\theta_0)\big) 
    =
    M_n + \tilde r_n
\end{align*}
where
\begin{align*}
    M_n 
    &= \frac{1}{2} \hat \Delta_n^{\top} V_{\theta_0} \hat \Delta_n - \frac{1}{2} \widetilde \Delta_n^{\top} V_{\theta_0} \widetilde \Delta_n + \big[ \nabla Q_n^2(\theta_0)\big]^\top \big(\hat \Delta_n-  \widetilde \Delta_n\big), \\
    \tilde r_n
    &= r_{n,2}(\hat \theta_n)-r_{n,2}(\widetilde \theta_n) + r_{n,3}(\hat \theta_n)-r_{n,3}(\widetilde \theta_n).
\end{align*}
In view of \eqref{eq:bound-errors-hat} and \eqref{eq:bound-errors-tilde}, the remainder term satisfies
\[
|\tilde r_n| \le C_{7}(\beta) (\nearmin + \beta^2)^{1+\gammaholder /2} + C_{9}(\beta) \beta^{2+\gammaholder }
\le 
C_{10} (\nearmin+\beta^2)^{1+\gammaholder /2}
\]
with $C_{10}=C_7(\tilde C_\beta) + C_{9}(\tilde C_\beta)$. Moreover, since $\nabla Q_n^2(\theta_0) =  - V_{\theta_0}\widetilde \Delta_n$, we find that
\begin{align*}
    M_n 
    =
    \frac{1}{2} \hat \Delta_n^{\top} V_{\theta_0} \hat \Delta_n
    + \frac{1}{2} \widetilde \Delta_n^{\top} V_{\theta_0} \widetilde \Delta_n - \widetilde \Delta_n^{\top} V_{\theta_0} \hat \Delta_n
    &=
    \frac12 \Big\| V_{\theta_0}^{1/2}(\hat \Delta_n-\widetilde \Delta_n) \Big\|_2^2
    \\&\ge 
    \frac12 \lambda_{\min}(V_{\theta_0}) \big\|\hat \Delta_n-\widetilde \Delta_n\big\|_2^2.
\end{align*}
Overall,
\[
C_4 \nearmin > \frac12 \lambda_{\min}(V_{\theta_0}) \big\|\hat \Delta_n-\widetilde \Delta_n\big\|_2^2 - C_{10}(\nearmin +\beta^2)^{1+\gammaholder /2}.
\]
Convexity of $x \mapsto x^{1+\gammaholder /2}$ and the fact that $\nearmin \le 1$ yields
\[
\big\|\hat \Delta_n-\widetilde \Delta_n\big\|_2^2 
\le
\frac2{\lambda_{\min}(V_{\theta_0})} \Big[ (C_4 + 2^{\gammaholder /2}C_{10}) \nearmin + 2^{\gammaholder /2}C_{10} \beta^{2+\gammaholder } \Big].
\]
This proves~\eqref{eq:linMest1} with $\tilde C_{r1} = 2 (C_4 + 2^{\gammaholder /2}C_{10})/ {\lambda_{\min}(V_{\theta_0})}$.

To prove the second half of the theorem, note that
\begin{align*}
& \Big\|\int_{[0,1]^d} V_{\theta_0}^{-1} J_{\theta_0}^\top\bg(\bx) \Lb_n(\bx) \, \diff \mu(\bx) -
\int_{[0,1]^d\setminus A} V_{\theta_0}^{-1} J_{\theta_0}^\top\bg(\bx) \widebar{\Lb}_n(\bx) \, \diff \mu(\bx)\Big\|_2
\\
\le~& \int_{[0,1]^d\setminus A} \|V_{\theta_0}^{-1} J_{\theta_0}^\top\bg(\bx)\|_2 \cdot \big| \widebar{\Lb}_n(\bx) - {\Lb}_n(\bx) \big| \, \diff \mu(\bx) 
+ \int_{A} \|V_{\theta_0}^{-1} J_{\theta_0}^\top\bg(\bx)\|_2 \cdot |{\Lb}_n(\bx)| \, \diff \mu(\bx)
\\
\le~& \Big(\sup_{\bx \in [0,1]^d\setminus A} \big|\widebar{\Lb}_n(\bx) - {\Lb}_n(\bx) \big|\Big) \times \int_{[0,1]^d} \|V_{\theta_0}^{-1} J_{\theta_0}^\top\bg(\bx)\|_2 \, \diff \mu(\bx) \\
&\hspace{7cm}+ \sqrt{k} \beta \int_{A} \|V_{\theta_0}^{-1} J_{\theta_0}^\top\bg(\bx)\|_2 \, \diff \mu(\bx).
\end{align*}
The bound 
\begin{align}
\label{eq:bound-vjg}
\|V_{\theta_0}^{-1} J_{\theta_0}^\top\bg(\bx)\|_2 
\le 
\matnorm{V_{\theta_0}^{-1}J_{\theta_0}^\top}_2 \| \bg(\bm x) \|_2
&\le
\matnorm{V_{\theta_0}^{-1}}_2 \matnorm{J_{\theta_0}^\top}_2\| \bg(\bm x) \|_2
\le 
C_V^{-1} \cdot \sqrt{sq} C_{\partial} \cdot \| \bm g (x)\|_2
\end{align}
completes the proof of Theorem~\ref{theo:Mestlin1}, with $\tilde C_{r2} = C_V^{-1} \sqrt{sq} C_{\partial}(C_g \vee 1)$.
\end{proof}

\begin{proof}[Proof of Theorem~\ref{theo:Mestlin3}]
First, all assumptions of Theorem~\ref{theo:linearization-good} are satisfied, and an application of that theorem implies that there exist constants $D_1 = D_1(d,\Kl)$ and $D_2 = D_2(d,\Kl)$ and an event $\Omega_2$ that has probability at least $1-(6d+5)\delta$ on which 
\[
\sup_{\bm x \in [0,1]^d \setminus (\mathfrak B^{\oplus \neunzig r})} \big| \mathbb L_n(\bm x) - \widebar {\mathbb L}_n(\bm x) \big| 
\le \zeta_{n,2} := 
B_{n,k}(L; [0,1+\neunzig r]^d) +
\frac{d}{\sqrt{k}}
+
D_{1} \sqrt{r\log\Big(\frac{D_{2}}{\delta r}\Big)}.
\]

On the same event, by~\eqref{eq:Snj-uniform-bound},
\begin{align*} 
\max_{j\in[d]} \sup_{x_j \in [0,1]} |S_{nj}(x_j) - x_j| \le \neunzig  r,
\end{align*}
and in view of the decomposition 
\begin{align*}
\Lb_n 
= \widetilde \Lb_n \circ S_n + \sqrt k (L \circ S_n - L) + B_n \circ S_n
\end{align*}
from \eqref{eq:decomposition1-Lbn}, we obtain that
\begin{align*}
\sup_{\bx \in [0,1]^d}|\Lb_n(\bx)| 
&\le \sup_{\bx \in [0,1+\neunzig r]^d}|\widetilde \Lb_n(\bx)| + \neunzig d r \sqrt{k} + B_{n,k}(L;[0,1+\neunzig r]^d)
\\&\le
\sup_{\bx \in [0,2]^d}|\widetilde \Lb_n(\bx)| + \neunzig d r \sqrt{k} + B_{n,k}(L;[0,2]^d)
\end{align*}
by Lipschitz continuity of $L$ and using that $\neunzig r \le 1$ by assumption.

The current choice of $\delta$ also satisfies the conditions of Lemma~\ref{lem:bnd-Ltilde} with $T=2$. Hence there exists an event $\Omega_3$ with probability at least $1-\delta$ on which 
\[
\sup_{\bx \in [0,2]^d} |\widetilde \Lb_n(\bx)| \le (188/3) \cdot d \cdot \sqrt{2 \log(1/\delta)} = (188\sqrt 2/3) \cdot d r  \sqrt{k}.
\]
Combining the above, we find that on $\Omega_2\cap\Omega_3$
\[
\sup_{\bx \in [0,1]^d} \big|\Lb_n(\bx) \big| \le (\neunzig+188\sqrt 2/3)dr
\sqrt{k}  + B_{n,k}(L;[0,2]^d) =\sqrt k \zeta_{n,1}.
\] 
As a consequence,
$\Omega_2\cap\Omega_3 \subset \Omega_1(n,\zeta_{n,1})$ 
with $\Omega_1(\cdot,\cdot)$ from \eqref{eq:Omega1}.
By an application of the second part of Theorem~\ref{theo:Mestlin1} with $A=\mathfrak B^{\oplus \neunzig r}$ we obtain  
\begin{equation*}
\sqrt{k} \big( \mestimator - \theta_0 \big)  = 2V_{\theta_0}^{-1} J_{\theta_0}^\top \int_{[0,1]^d\setminus \mathfrak B^{\oplus \neunzig r}} \bg(\bx) \widebar{\Lb}_n(\bx) \, \diff \mu(\bx) + \sqrt{k} \bm r_{n,1}(\beta,\nearmin) + \bm r_{n,2}(\mathfrak B^{\oplus \neunzig r})
\end{equation*}
where 
\[
\|r_{n,2}(\mathfrak B^{\oplus \neunzig r})\|_2 \le \tilde C_{r2} \Big( \zeta_{n,2} + \sqrt{k} \zeta_{n,1}\int_{\mathfrak B^{\oplus \neunzig r}} \|\bg(\bx)\|_2 \, \diff \mu(\bx) \Big).
\] 
In the following, with a slight abuse of notation, we extend the definition of $\widebar \Lb_n$ to $[0,1]^d$ by replacing the partial derivatives of $L$ by the their right-hand side counterparts as described in the paragraph before Theorem~\ref{theo:Mestlin3}. 
Then, by an application of Lemma~\ref{lem:bnd-Ltilde} we have, on an event $\Omega_4$ that has probability at least $1-(d+1)\delta$,
\begin{align*}
\sup_{\bx \in [0,1]^d}|\widebar{\Lb}_n(\bx)| \le \sup_{\bx \in [0,1]^d}|\widetilde{\Lb}_n(\bx)| + \sum_{j \in [d]}\sup_{x \in [0,1]} |\widetilde{\Lb}_{nj}(\bx)| \le 2\cdot(188/3)\cdot dr \cdot \sqrt{k}.     
\end{align*}
Thus, on the event $\Omega_2\cap\Omega_3\cap\Omega_4$ and using \eqref{eq:bound-vjg}, we have 
\begin{align*}
&\Big\|2V_{\theta_0}^{-1} J_{\theta_0}^\top \int_{[0,1]^d\setminus \mathfrak B^{\oplus \neunzig r}} \bg(\bx) \widebar{\Lb}_n(\bx) \, \diff \mu(\bx) - 2V_{\theta_0}^{-1} J_{\theta_0}^\top \int_{[0,1]^d} \bg(\bx) \widebar{\Lb}_n(\bx) \, \diff \mu(\bx)\Big\|_2
\\
& \le 4 \cdot(188/3)\cdot dr \cdot \sqrt{k} \cdot \int_{\mathfrak B^{\oplus \neunzig r}} \|V_{\theta_0}^{-1} J_{\theta_0}^\top \bm g(\bx)\|_2  \, \diff \mu(\bx)
\\
& \le 4 \sqrt{k} \zeta_{n,1} \int_{\mathfrak B^{\oplus \neunzig r}} \|V_{\theta_0}^{-1} J_{\theta_0}^\top \bm g(\bx)\|_2  \, \diff \mu(\bx)
\\&\le 
4C_V^{-1} \sqrt{sq} C_{\partial} \cdot \sqrt{k} \zeta_{n,1}\int_{\mathfrak B^{\oplus \neunzig r}} \| \bm g(\bx)\|_2  \, \diff \mu(\bx) .
\end{align*}
Noting that $\Omega_2\cap\Omega_3\cap\Omega_4$ has probability at least $1 - 7(d+1)\delta$ and that
\[
2V_{\theta_0}^{-1} J_{\theta_0}^\top \int_{[0,1]^d} \bg(\bx) \widebar{\Lb}_n(\bx) \, \diff \mu(\bx)
= \frac{1}{\sqrt{k}} \sum_{i=1}^n \big( Z_{i,n} - \Eb[Z_{i,n}] \big) 
\]
by definition of $Z_{i,n}$ in \eqref{eq:defZjn}
completes the proof, after increasing $\tilde C_{r2}$. 
\end{proof}

\begin{proof}[Proof of Theorem~\ref{thm:Mest-hd-CLT}]
Without loss of generality, 
\begin{align}\label{eq:helpMestBmax}
\frac{\log^5(sn)}{k} \leq 1, \quad B_{n,k}^{\mathcal I} \le 1
\end{align} 
as otherwise the right-hand side in the theorem is greater than 1. We start by bounding $d_K(\bm S_n, \bm T_n)$. By Lemma~\ref{lem:ks1},  for any $\lambda>0$,
\begin{align*}
d_K(\bm S_n , \bm T_n) 
&\le 
\Prob\big(\|\bm S_n-\bm T_n\|_\infty \ge \lambda\big)
    + \sup_{\bm x \in \R^{s}} \Prob( \bm T_n \le \bm x+\lambda \bm1 ) - \Prob( \bm T_n \le \bm x-\lambda \bm 1) \\
    &\leq \sum_{I \in \mathcal{I}} \Prob\big( \|\bm S_n^I - \bm T_n^I\|_2 \geq \lambda\big)+\sup_{\bm x \in \R^{s}} \Prob( \bm T_n \le \bm x+\lambda \bm1 ) - \Prob( \bm T_n \le \bm x-\lambda \bm 1),
\end{align*}
where we have used the union bound and the fact that $\| \cdot \|_\infty \le \| \cdot \|_2$. By the assumption that \(\sigma^2_{\min}>0\), with the same reasoning as in the proof of Theorem~\ref{theo:clt}, specifically~\eqref{eq:bound-dK}, 
\begin{align*}
    \sup_{\bm x \in \R^{s}} \Prob( \bm T_n \le \bm x+\lambda \bm1 ) - \Prob( \bm T_n \le \bm x-\lambda \bm 1) \leq \frac{8\lambda}{\sigma_{\min}^2} \sqrt{\log s}  + 2 d_K(\bm T_n, \bm G_n).
\end{align*}
Thus
\begin{align}
\label{eq:dk-sn-gn-m-estimators}
d_K(\bm S_n, \bm G_n) \leq 
\sum_{I \in \mathcal{I}} \Prob(\norm{\bm S_n^I - \bm T_n^I}_2 \geq \lambda) + \frac{8\lambda}{\sigma_{\min}^2} \sqrt{\log(s)}  + 3 d_K(\bm T_n, \bm G_n).
\end{align}
In the remaining proof, we will choose a suitable $\lambda$ to balance the second and third term and bound $d_K(\bm T_n, \bm G_n)$. 

\medskip
\noindent
\textbf{Bounding $d_K(\bm T_n, \bm G_n)$.} We will apply Theorem~\ref{theo:clt-new} to the random vector \(\bm Y_{i,n} \in \R^{s}\) with entries given by \(Y_{i,n,(I,t)} = k^{-1/2} (Z_{i,n}^{I,t} - \Exp[Z_{i,n}^{I,t}])\), enumerating over \(I \in \mathcal{I}\) and \(t \in [s^I]\). By assumption, \(\Exp[Y_{i,n,(I,t)}^2] \geq \sigma^2_{\min}\). 

Next, note that 
\[
|Z_{i,n}^{I,t}| \le \| \bm Z_{i,n}^I\|_2\le 2 \matnorm{V_{\theta_0^I}^{-1}J_{\theta_0^I}^\top}_2 \| \bm A^I_{i,n}\|_2
\le 2 (C_V^{\Ic})^{-1}\sqrt{s^{\Ic} q^{\Ic}} C_{\partial}^{\Ic} \cdot \| \bm A^I_{i,n}\|_2
\]
where $s^{\Ic} =\max_{I \in \Ic} s^I$ and  $q^{\Ic} = \max_{I \in \Ic} q^I$.
Further, by the same argumentation that lead to \eqref{eq:qn-q-hoelder}, we have
\begin{align*}
    \norm{\bm A^I_{i,n}}_2 &\leq C_g^{\Ic} \sup_{\bx \in [0,1]^I} \Big| \mathbf{1}\!\Big(\exists j\in I:\; V_{ij}<\frac{k}{n}x_j\Big)
- \sum_{j\in I} \partial_j \widetilde L_I(\bx_I)\,
\mathbf{1}\!\Big(V_{ij}<\frac{k}{n}x_j\Big) \Big| \\
&\leq C_g^{\Ic} \abs{I} \leq  C_g^{\Ic} m.
\end{align*}
Combining the previous two inequalities, we obtain that
\begin{align} \label{eq:deftildeC}
     \sup_{I \in \mathcal I, t \in [q^I]}\big| Z_{i,n}^{I,t}-\Exp[Z_{i,n}^{I,t}]\big| \leq  2 (C_V^{\Ic})^{-1}\sqrt{s^{\Ic} q^{\Ic}} C_{\partial}^{\Ic}C_g^{\Ic} m =: \tilde C.
\end{align}
Also, 
\begin{align*}
\norm{\bm A^I_{i,n}}_2 
&\leq  C_g^{\Ic} \Big\{ \sup_{\bx_I \in [0,1]^I} \mathbf{1}\Big(\exists j\in I:\; V_{ij}<\frac{k}{n}x_j\Big) + \sum_{j \in I} \sup_{\bx_I \in [0,1]^I} \mathbf{1}\Big(V_{ij}<\frac{k}{n}x_j\Big)   \Big\} \\
&\leq 2C_g^{\Ic} \sum_{j \in I} \sup_{\bx_I \in [0,1]^I} \mathbf{1}\Big(V_{ij}<\frac{k}{n}x_j\Big)  \\
&\leq 2C_g^{\Ic} \sum_{j \in I} \mathbf{1}\Big(V_{ij}<\frac{k}{n}\Big),
\end{align*}
which yields 
\begin{align*}
\Exp\big[\big| Z_{i,n}^{I,t} -\Exp[Z_{i,n}^{I,t}] \big| \big]
\le 2
\Exp\big| Z_{i,n}^{I,t}\big| 
&\le 
4(C_V^{\Ic})^{-1}\sqrt{s^{\Ic} q^{\Ic}} C_{\partial}^{\Ic} \cdot \Exp \| \bm A_{i,n}^{I} \|_2 
\\&\le 
8(C_V^{\Ic})^{-1}\sqrt{s^{\Ic} q^{\Ic}} C_{\partial}^{\Ic} C_g^{\Ic} m\cdot  \frac{k}{n}
=4 \tilde C \frac{k}{n}.
\end{align*}
Hence, 
\begin{align*}
    \sum_{i=1}^n \Exp \Big| k^{-1/2} \big(Z_{i,n}^{I,t}-\Exp[ Z^{I,t}_{i,n}] \big)\Big|^4 
    &\leq 
    \tilde C^3 \frac{n}{k^2} \Exp\Big[\big| Z_{1,n}^{I,t} -\Exp[Z_{1,n}^{I,t}] \big| \Big] 
    \leq   
    4\tilde C^4 \frac{1}k
    = \frac{b_2 B_n^2}n, 
\end{align*}
where $B_n = (\log 2)^{-1}\tilde C\sqrt{n/k}$ and $b_2 = 4(\log 2)^2 \tilde C^{2}$. With these choices, we also have
\[
\frac{\sqrt{n}\big| k^{-1/2}(Z_{i,n}^{I,t}-\Exp[Z_{i,n}^{I,t}])\big|   }{B_n} 
\le 
\frac{\sqrt{n/k} \tilde C}{B_n} = \log2,
\]
and hence the conditions of Theorem~\ref{theo:clt-new} hold with $b_1 = \sigma^2_{\min}$. An application of the theorem yields
\begin{align}
\label{eq:dk-tn-gn-m-estimators}
    d_K(\bm T_n, \bm G_n) \leq C_1 \Big( \frac{\log^5(s n)}{k}\Big)^{1/4},
\end{align}
with $C_1$ depending only on $\sigma^2_{\min}$ and $\tilde C$.

\medskip
\noindent
\textbf{Choosing $\lambda$.}
Let
\begin{align} \label{eq:helpMestdefdelta}
\delta = \frac{1}{m |\mathcal{I}|}\Big(\frac{\log^5(sn)}{k}\Big)^{1/4}
\end{align}
and recall $r = r(\delta, 1,k) = \sqrt{k^{-1}\log(1/\delta)}$ from \eqref{eq:definition-r}. By our assumption in \eqref{eq:helpMestBmax} from the beginning of the proof, we have $\delta \le 1/(m|\mathcal I|) < e^{-1}$,
and we will later verify that this $\delta$ also satisfies the conditions $\log(m/\delta)\le 2k/7$ and $\neunzig r \le 1$. We may therefore apply Theorem~\ref{theo:Mestlin3} for each tuple $(L_I, \{L_I(\cdot; \theta^I): \theta^I \in \Theta^I\}, \bm g^I, \mu^I)$, which yields the existence of certain constants $D_1^I, D_2^I>0, \tilde C_\beta^I, \tilde C_\nearmin^I \in (0,1]$ and $\tilde C_{r1}^I, \tilde C_{r2}^I>0$ such that certain claims hold with probability at least $1-7(|I|+1)\delta \ge 1-7(m+1)\delta$. In the following, we will make use of these claims.

Let $\tilde C_\beta^{\Ic} = \min_{I \in \Ic} \tilde C_\beta^I$ and
\begin{align}
\label{eq:definition-zeta_n1-mest}
\zeta_{n,1}^{\mathcal I} := \max_{I \in \Ic} \zeta_{n,1}^I, 
\qquad
\zeta_{n,1}^I :=
\Big( k^{-1/2} \sup_{\bm{x}_I \in [0,2]^{I}} \abs{B_n^I(\bm{x}_I)}
+(\neunzig+188\sqrt 2/3)\cdot |I| r  \Big).
\end{align}
We will later verify that 
$
\zeta_{n,1}^{\mathcal I} \le \tilde C_\beta^{\Ic},
$
which in turn implies that $\zeta_{n,1}^I \le C_\beta^{I}$ for each $I \in \Ic$. Theorem~\ref{theo:Mestlin3} therefore guarantees that, for any $\nearmin < \tilde C_{\nearmin}^{\Ic} := \min_{I \in \Ic} \tilde C_\nearmin^I$, 
\begin{align}
\label{eq:dk-sni-tni-m-estimators}
\sum_{I \in \mathcal{I}} \Prob\big( \|\bm S_n^I - \bm T_n^I\|_2 \geq \lambda_{n,k}(\delta) \big)  
\leq 7|\mathcal{I}|(m+1)\delta \leq 14|\mathcal{I}|m\delta = 14\cdot \Big(\frac{\log^5(sn)}{k}\Big)^{1/4}
\end{align}
where $\lambda_{n,k}(\delta) := \zeta_{n,2}^{\Ic} + \zeta_{n,3}^{\Ic}$ with 
\begin{align*}
\zeta_{n,2}^{\Ic} & := \max_{I \in \Ic} \zeta_{n,2}^I, 
\qquad
\zeta_{n,3}^{\Ic}  := \max_{I \in \Ic} \zeta_{n,3}^I, 
\end{align*}
where, recalling that $\gammaholder =1$ by assumption,
\begin{align*}
\zeta_{n,2}^I &= \sqrt{\tilde C_{r1}^I k \{ (\zeta_{n,1}^{I})^{3} + \nearmin\} },
\\
\zeta_{n,3}^I &= \tilde C_{r2}^I \Big(\sup_{\bm{x}_I \in [0,2T]^{I}} |B_n^I(\bm{x}_I)| +
\frac{\abs{I}}{\sqrt{k}}
+
D_{1}^I \sqrt{r\log\Big(\frac{D_{2}^I}{\delta r}\Big)} \Big)
\\ 
&\hspace{6cm} +  \sqrt k \zeta_{n,1}^I \int_{\mathfrak B_I^{\oplus \neunzig r}} \|{\bg}^I(\bx_I)\|_2 \, \diff \mu^I(\bx_I) \Big).
\end{align*}
In summary, from \eqref{eq:dk-sn-gn-m-estimators}, \eqref{eq:dk-tn-gn-m-estimators} and \eqref{eq:dk-sni-tni-m-estimators},
\begin{align}\label{eq:bounddKprelim1}
d_K(\bm S_n, \bm G_n) \lesssim 
(\zeta_{n,2}^{\Ic} + \zeta_{n,3}^{\Ic}) \sqrt{\log(s)}  + \Big(\frac{\log^5(sn)}{k}\Big)^{1/4},
\end{align}
where the constant in $\lesssim$ depends on $\sigma_{\min}$ and $ \tilde C$ from~\eqref{eq:deftildeC}. 

We now bound $\zeta_{n,2}^{\Ic}$. First, by our choice of $\delta$ in \eqref{eq:helpMestdefdelta}, 
\begin{align}\label{eq:boundhelpr_1}
r = \sqrt{\frac{1}{k} \log\Big( \frac{m|\mathcal I|k^{1/4}}{\log^{5/4}(sn)}\Big)} \leq \sqrt{\frac{1}{k} \log(m s k^{1/4})} \leq C_2  \sqrt{\frac{\log(sk)}{k}}
\end{align}
with $C_2 = \sqrt{1+\log m}$ as 
\[
\log(msk^{1/4}) \leq \log(m)+\log(sk) 
\leq \log(sk) \br{1 + \log m} = C_2^2 \log(sk),
\] 
using that $\log(sk) \geq \log(6) \ge 1$. Recall our assumption $B_{n,k}^{\mathcal I} \le 1$ from the beginning of the proof. Moreover, since $\delta \le e^{-1}$, we have $k^{-1/2} \le r \le mr$ and thus, by \eqref{eq:definition-zeta_n1-mest} and \eqref{eq:boundhelpr_1}, 
\begin{align}\label{eq:helpboundzetan1}
\zeta_{n,1}^{\Ic} \le  (1+\neunzig+188\sqrt 2/3) \cdot mr \le C_3 \sqrt{\frac{\log(sk)}{k}} 
\end{align} 
where $C_3 := C_2 (\neunzig+1+188\sqrt 2/3)m$. Hence, by subadditivity of $x \mapsto \sqrt x$ on $[0,\infty)$,  
\begin{align} 
    \zeta_{n,2}^{\Ic} &\le \sqrt{\tilde C_{r1}^{\Ic} k} \times \big\{ (\zeta_{n,1}^{\Ic})^{3/2} + \sqrt{\nearmin} \big\}  
    \leq C_4 \Big\{ \Big(\frac{\log^3(sk)}{k}\Big)^{1/4} + \sqrt{k\nearmin}\Big\}, \label{eq:boundzetan2}
\end{align}
where $\tilde C_{r1}^{\Ic} = \max_{I \in \Ic}\tilde C_{r1}^{I}$ and where $C_4 = (\tilde{C}_{r1}^{\Ic})^{1/2}C_3^{3/2}$.

Next, we bound $\zeta_{n,3}^{\Ic}$. First, using that 
\[
\delta = \frac{1}{m|\mathcal I|} \Big( \frac{\log^5(sn)}{k} \Big)^{1/4} \ge \frac{1}{m|\mathcal I|k^{1/4}}
\ge 
\frac{1}{msk^{1/4}}
\] 
and $r \geq k^{-1/2}$, we obtain, recalling \eqref{eq:boundhelpr_1},
\begin{align*}
    \max_{I \in \Ic} D_1^I\sqrt{r\log\Big(\frac{D_{2}^I}{\delta r}\Big)} 
    &\leq 
    D_1^{\Ic}C_2^{1/2} \Big(\frac{\log(sk)}{k}\Big)^{1/4} \sqrt{\log(D_2^{\mathcal I}  ms k^{3/4})} 
    \leq 
    C_5 \Big( \frac{\log^3(sk)}{k}\Big)^{1/4}
\end{align*}
where $D_j^{\mathcal I} = \max_{I \in \Ic} D_j^I$ and $C_5 = D_1^{\Ic}C_2^{1/2} \{1 + \log(D_2^{\Ic} m)\}^{1/2}$ and we used $sk\ge 3$ so that 
\[
\log(D_2^{\mathcal I}  ms k^{3/4}) 
= 
\log(D_2^{\mathcal I}m) + \log(s k^{3/4}) \leq \log(sk)\{ 1 + \log(D_2^{\Ic} m) \}. 
\] 
Together with \eqref{eq:helpboundzetan1} and $\neunzig r \le \zeta_n$ 
by \eqref{eq:boundhelpr_1} with $\zeta_n$ from the formulation of the theorem, we obtain that
\begin{align} \label{eq:boundzetan3}
\zeta_{n,3}^{\Ic}
\le\tilde C_{r2}^{\mathcal I} \Big\{ B_{n,k}^{\Ic} + (C_5 +m)\Big( \frac{\log^3(sk)}{k}\Big)^{1/4} +  C_3 \sqrt{\log(sk)} \int_{\mathfrak B_I^{\oplus \zeta_n}} \|{\bg}^I(\bx_I)\|_2 \, \diff \mu^I(\bx_I) \Big\}
\end{align}

Combining the bounds in~\eqref{eq:bounddKprelim1},~\eqref{eq:boundzetan2} and~\eqref{eq:boundzetan3} we obtain
\begin{align*}
d_K(\bm S_n, \bm G_n) &\lesssim \Big(\frac{\log^5(sn)}{k}\Big)^{1/4} + 
\sqrt{\log s}\Big\{ B_{n,k}^{\Ic}  + \Big(\frac{\log^3(sk)}{k}\Big)^{1/4} \\
& \hspace{4cm} + \sqrt{k\nearmin} + \sqrt{\log(sk)} \int_{\mathfrak B_I^{\oplus \zeta_n}} \|{\bg}^I(\bx_I)\|_2 \, \diff \mu^I(\bx_I)\Big\},
\end{align*}
which implies the assertion.

It remains to verify that $\delta$ as defined in~\eqref{eq:helpMestdefdelta} satisfies $\log(m/\delta)\le 2k/7$ and $\neunzig r \le 1$ and $\zeta_{n,1}^{\mathcal I} \le \tilde C_\beta^{\Ic}$. First,
\begin{align*}
    \log\br{\frac{m}{\delta}} = \log\Big( \frac{m^2 \abs{\mathcal{I}} k^{1/4}}{\log^{5/4}(sn)} \Big) 
    \leq \log(m^2 \abs{\mathcal{I}} k^{1/4}) \leq 2k/7
\end{align*}
by assumption (ii). Next, by assumption (iii),
\begin{align*}
r = \sqrt{\frac{1}{k}\log\Big( \frac1\delta\Big)} 
=  \sqrt{\frac{1}{k}\log\Big( \frac{m|\mathcal I|k^{1/4}}{\log^{5/4}(sn)} \Big)} 
\le 
\sqrt{\frac{1}{k}\log( m|\mathcal I|k^{1/4})}
\le \frac{1}{\neunzig}
\end{align*}
Finally, from \eqref{eq:boundhelpr_1} and \eqref{eq:helpboundzetan1},
\[
\zeta_{n,1}^{\Ic} \le m(1+\neunzig + 188\sqrt{2}/3) \sqrt{ \frac{1}k\log(m|\Ic|k^{1/4})}.
\]
The right-hand side is upper bounded by $\tilde C_\beta^{\Ic}$ if we choose $\tilde C_{k}^{\mathcal I} = [\tilde C_\beta^{\Ic} / \{m(1+\neunzig + 188\sqrt{2}/3)\}]^2$. This completes the proof.
\end{proof}

\section{Auxiliary results}
\label{sec:auxiliary-results}

The following lemma is a version of the argument on page 7 in \cite{GoiSabCle15}, with the precise constant $188/3$ deduced from \cite{ClemenconJalalzaiLhautSabourinSegers2023}.

\begin{lemma}\label{lem:bnd-Ltilde}
Let $n \in \N,k \in [n],d \in \N, T >0,\delta \in (0,e^{-1})$ and $\emptyset \ne I \subset [d]$ satisfy  $\log(1/\delta)\le |I|^2 T k$. Then
\[
\sup_{\bm x \in [0,T]^I} | \widetilde \Lb_{n,I}(\bm x)| \le (188/3) \cdot |I| \cdot \sqrt{T \log(1/\delta)}
\] 
with probability at least $1-\delta.$
\end{lemma}

\begin{proof} 
Fix $I \subset [d]$, write $m=|I|$ and define $\mu_{n,I} = \frac 1n \sum_{i=1}^n \delta_{V_{i,I}}$ and let $\mu_I$ denote the distribution of $V_{i,I}$. Then we can write  
    \[
     \sup_{\bm x \in [0,T]^I} | \widetilde \Lb_{n,I}(\bm x)| = \frac{n}{\sqrt k} \sup_{A \in \mathcal A} | \mu_{n,I}(A) - \mu_I(A)|
     \]
where $\mathcal A$ contains all sets of the form $A_{\bm x} = \{ \bm z \in [0,\infty)^I \mid \exists j \in I: z_j < (k/n) x_j \}$ with $\bm x \in [0, T]^I$. Let $\mathbb A := \bigcup_{A\in \mathcal A} A$, with $p = \mu(\mathbb A) =  \Prob(\exists j \in I : V_{ij} \le \frac kn T) \le mTk/n$.
By Theorem A.1 in \cite{ClemenconJalalzaiLhautSabourinSegers2023} we have, with probability at least $1-\delta$,
\[
\sup_{A \in \mathcal A} | \mu_n(A) - \mu(A)|
\le 
\frac2{3n} \log(1/\delta) + \sqrt{\frac{mTk}{n^2}} \Big\{ 2 \sqrt{\log(1/\delta)}+ 60 \sqrt m\Big\},
\]
where we have used that the VC-dimension of $\mathcal A$ is $m$. Since  $1 \le \log(1/\delta) \le m^2 T k $, we get the upper bound
\begin{align*}
\sup_{\bm x \in [0,T]^I} | \tilde \Lb_{n,I}(\bm x)|  
&\le 
\frac2{3 \sqrt k} \log(1/\delta) + \sqrt{mT} \Big\{ 2 \sqrt{\log(1/\delta)}+ 60 \sqrt m\Big\}
\\&\le
\frac{2}{3\sqrt k} m\sqrt{Tk \log(1/\delta)} + \sqrt{mT} \Big\{2\sqrt{\log(1/\delta)} + 60\sqrt{m}\Big\} 
\\& \le 
m\sqrt{T\log(1/\delta)}\Big\{  \frac23 + \frac2{\sqrt m} + 60 \Big\} 
\le 
(188/3) m \sqrt{T \log(1/\delta)} 
\end{align*}
with probability at least $1-\delta$.
\end{proof}

Recall $S_{nj}(x_j) = (n/k) \cdot V_{\lceil kx_j \rceil, j} \cdot \bm 1(x_j>0)$ from \eqref{eq:snj}. The following lemma is akin to Lemma 9 in \cite{GoiSabCle15}.

\begin{lemma}[Bound on order statistics] \label{lem:bound-on-order-statistics}
Let $\neunzig=188\sqrt{2}/3  + \sqrt{1-\log2} \approx 89.18$. For any $n,d,k,T \in \N$ and $\delta \in (0, e^{-1})$ with $k\in[n]$ and $\log(d/ \delta) \le (1-\log 2) kT  \approx 0.31\cdot kT$ we have
\begin{align} 
\label{eq:bound-on-order-statistics-1}
\max_{j\in[d]} \sup_{x_j \in [0,T]} S_{nj}(x_j) \le 2T
\end{align}
with probability larger than $1-\delta$. Moreover, we have
\begin{align} \label{eq:bound-on-order-statistics-2}
\max_{j\in[d]} \sup_{x_j \in [0,T]} |S_{nj}(x_j) - x_j| \le \neunzig \sqrt{\frac{T}k \log\Big(\frac1\delta\Big)}
\end{align}
with probability larger than $1-(d+1)\delta$, and on the latter event where \eqref{eq:bound-on-order-statistics-2} is met we also have \eqref{eq:bound-on-order-statistics-1}. 
\end{lemma}

\begin{proof}[Proof of Lemma~\ref{lem:bound-on-order-statistics}] 
First, note that $\sup_{x_j \in [0,T]} S_{nj}(x_j) = (n/k) \cdot V_{kT:n, j}$ by monotonicity. Moreover, writing $G_{nj}(v_j)=n^{-1}\sum_{i=1}^n \bm 1(V_{ij} \le v_j)$, we have $V_{\ell:n} \le x$ iff $G_{nj}(x) \ge \ell/n$ for all $\ell \in[n]$ and $x \in \R$, which implies
\[
\frac{n}k V_{kT:n, j} \le 2T 
\quad \Longleftrightarrow \quad
G_{nj}\Big( 2\frac{kT}n  \Big) \ge \frac{kT}n.
\]
As a consequence, by the union bound, 
\begin{align*}
\Prob\Big(\max_{j\in[d]} \sup_{x_j \in [0,T]} S_{nj}(x_j) > 2 T  \Big)
\le 
d \cdot \Prob\Big( G_{nj}\Big( 2\frac{kT}n  \Big) < \frac{kT }n \Big)
&\le 
d \cdot \big(\sqrt 2 e^{-1/2}\big)^{2kT}
\\&=
d \cdot \exp\big( - (1-\log2) kT \big\},
\end{align*}
where the second inequality follows from the multiplicative Chernoff bound; see, for instance, Exercise 2.11 in \cite{Boucheron2013}.
By our assumption $\log(d/\delta) \le (1-\log2) kT$, the upper bound in the previous display is smaller than $\delta$.  This proves~\eqref{eq:bound-on-order-statistics-1}.

We may now proceed analogously to the proof of Lemma 9 in \cite{GoiSabCle15} to show that
\begin{align} \label{eq:bound-gauss}
\max_{j\in[d]} \sup_{x_j \in [0,T]} \Big|S_{nj}(x_j) - \frac{\lceil kx_j\rceil}k\Big| 
\le 
(188\sqrt{2}/3) \sqrt{\frac{T}k \log\Big(\frac1\delta\Big)}
\end{align}
with probability at least $1-(d+1)\delta$. Indeed, by the definition of $S_{nj}$ in \eqref{eq:snj}, we have, on the event in~\eqref{eq:bound-on-order-statistics-1},
\begin{align*}
\sup_{x_j \in [0,T]} \Big|S_{nj}(x_j) - \frac{\lceil kx_j\rceil}k\Big| 
&=
\sup_{x_j \in (0,T]} \Big| S_{nj}(x_j) - \frac{n}k G_{nj}\big(V_{\lceil kx_j\rceil:n, j}\big)   \Big|
\\&=
\frac{n}k \sup_{x_j \in (0,T]} \Big| \frac{k}n S_{nj}(x_j) - G_{nj}\Big( \frac{k}n S_{nj}(x_j) \Big) \Big|
\\&\le
\frac{n}k \sup_{x_j \in [0, 2T]} \Big| \frac{k}n x_j -  G_{nj}\Big( \frac{k}n x_j \Big) \Big|
\\&=
\sup_{x_j \in [0, 2T]} \Big| x_j -  \Loracle_{nj}( x_j ) \Big|
=\frac 1{\sqrt k}\sup_{x_j \in [0,2T]} | \tilde \Lb_{nj}(x_j)|
\end{align*} 
where we used that $\frac nk G_{nj}(\frac kn x_j - ) = \Loracle_{nj}(x_j)$. As a result, since $\log(1/\delta) \le \log(d/\delta) \le ( 1-\log2)  kT \le 2Tk$, the assertion in \eqref{eq:bound-gauss} follows from Lemma \ref{lem:bnd-Ltilde}, applied with $T$ replaced by $2T$, and the union bound. Finally, the result in \eqref{eq:bound-on-order-statistics-2} follows from the triangular inequality, observing that
\[
\sup_{x_j \in [0,T]} \Big| \frac{\lceil kx_j\rceil}k - x_j\Big| 
\le \frac1k \le \sqrt{1-\log2} \sqrt{\frac{T}k \log \Big(\frac1\delta\Big)},
\]
again using that $\log(1/\delta) \le ( 1-\log2)  kT$.
\end{proof}

Recall that $\bm V_1, \bm V_2, \dots$ are 
iid random vectors in $[0,1]^d$ with standard uniform margins. For $\bm u \in \R^d$, the interesting points being $\bm u \in [0,1]^d$, let
\begin{align}
\alpha_{n}(\bm u) 
&= \label{eq:alphan}
\frac1{\sqrt n} \sum_{i=1}^n \big[ \bm1(\forall j \in [d]: V_{ij} < u_j) - \Prob(\forall j \in [d]: V_{ij} < u_j) \big], \\
\beta_{n}(\bm u) 
&= \label{eq:betan}
\frac1{\sqrt n} \sum_{i=1}^n \big[ \bm1(\exists j \in [d]: V_{ij} < u_j) - \Prob(\exists j \in [d]: V_{ij} < u_j) \big].
\end{align}

\begin{lemma}\label{lem:modulus_newnewnew}  
Fix $d \in \N$, $0 \le a_j < b_j \le 1$ for $j \in [d]$, $\eps \in (0, \min_{j \in [d]}(b_j-a_j)]$, and $\delta \in (0,e^{-1})$. Then, for any $n\in\N$, there exists an event $\Omega$ of probability at least $1-\delta$ such that, on $\Omega$,
\begin{align} \label{eq:modulus_beta1}
\omega_{\alpha_{n}}(\eps; [\bm a, \bm b]) 
&\le  \nonumber
2d \Big[ \frac 2{3\sqrt n} \log\Big( \frac{2 \| \bm b - \bm a \|_1}{\eps\delta}\Big) +\Big\{2 \sqrt{\eps\log\Big( \frac{2\| \bm b - \bm a \|_1}{\eps\delta}\Big)}+ 60\sqrt{2d\eps} \Big\} \Big]
\\&\le 
\kappa \sqrt{\eps\log\Big( \frac{2\| \bm b - \bm a \|_1}{\eps\delta}\Big)},
\end{align}
where $\omega_{\alpha_n}$ is the modulus of continuity defined in \eqref{eq:definition-modulus} and where
\[
\kappa = 2d \Big[ \sqrt{\frac4{9n\eps} \log\Big( \frac{2\| \bm b - \bm a \|_1}{\eps\delta}\Big)} + 2 + 60 \sqrt{2d}  \Big].
\]
The same inequality holds with $\alpha_{n}$ replaced by $\beta_{n}$, also with probability at least $1-\delta$.
\end{lemma}

\begin{proof}
The proof is largely inspired by \cite[Inequality 5.3]{einmahl1987}.
For $j \in [d]$ and $k \in K_j := \{ 1, \dots, \lceil (b_j-a_j)/\eps \rceil \}$ define
\[
\mathcal A_{j,k} = \Big\{ [\bm x,\bm y) \subset [0,1)^d : \ a_j+\eps(k - 1) \le  x_j < y_j \le a_j+\eps k \Big\},
\] 
which has VC-dimension $2d$.
Next, let $\mathbb A_{j,k} = \bigcup_{A \in \mathcal A_{j,k}} A$, and note that for all $j \in [d], k \in K_j$  we have $\Prob(\bm V \in \Ab_{j,k}) \le \Prob(V_j \in [a_j+\eps(k-1), a_j+\eps k]) \le \eps$.

Let $\tilde \delta >0$. Then, by Theorem A.1 in \cite{ClemenconJalalzaiLhautSabourinSegers2023}, applied with $B = \mathbb A_{j,k}$, there exists an event $\Omega_{j,k}$ with probability at least $1-\tilde\delta$ such that, on $\Omega_{j,k}$,
\[
\sup_{A\in \mathcal A_{j,k}}  |\mu_{n}(A) - \mu(A) | 
\le \frac 2{3n} \log(1/{\tilde \delta}) + \sqrt{\frac{\eps}{n}}\Big\{2 \sqrt{\log(1/{\tilde \delta})}+ 60\sqrt{2d} \Big\},
\]
where $\mu_{n} = n^{-1} \sum_{i=1}^n \delta_{\bm V_{i}}$ and where $\mu$ is the distribution of $\bm V_{i}$. Note that $|K_j| = \lceil (b_j-a_j)/\eps \rceil \le (b_j-a_j)/\eps + 1 \le 2(b_j-a_j)/\eps$.
On the intersection set $\Omega_1 = \bigcap_{j \in [d]} \bigcap_{k \in K_j} \Omega_{j,k}$, which has probability at least $1- \sum_{j \in [d]}|K_j| \tilde \delta \ge 1-2 \| \bm b - \bm a \|_1 \tilde \delta /\eps$, we obtain that
\[
\max_{j \in [d]} \max_{k \in K_j }\sup_{A\in \mathcal A_{j,k}}  |\mu_{n}(A) - \mu(A) | 
\le 
\Big[ \frac 2{3n} \log(1/{\tilde \delta}) + \sqrt{\frac{\eps}{n}}\Big\{2 \sqrt{\log(1/{\tilde \delta})}+ 60\sqrt{2d} \Big\} \Big].
\]
Let
\[
\Ac :=  \Big\{\ R_{k_1, \dots, k_d} := \bigtimes_{j = 1}^d \big[ a_j+(k_j - 1)\eps, (a_j+k_j\eps)\wedge b_j \big]  \ : \ k_j \in K_j \Big\}
\]
denote a cover of $[\bm a, \bm b]$ consisting of axis aligned hyper-rectangles $R_{k_1, \dots, k_d}$ with edge length at most $\eps$, and note that
\begin{align*}
\omega_{\alpha_{n}}(\eps, [\bm a, \bm b])
&=
\sup_{\|\bm x - \bm y \|_\infty \le \eps, \bm x,\bm y \in [\bm a, \bm b]} |\alpha_{n}(\bm x) - \alpha_{n}(\bm y)|
\le
2 \max_{R \in \mathcal A} \sup_{\bm x, \bm y \in R} |\alpha_{n}(\bm x) - \alpha_{n}(\bm y)|
\end{align*}
by the triangle inequality for the $\| \cdot \|_\infty$-norm.\footnote{In \cite[page 72]{einmahl1987}, the constant in front of the max-sup is $2^d$, but it can be replaced by $2$. Indeed, note that if $\bm x, \bm y \in [\bm a, \bm b]$ with $\|\bm x - \bm y\|_\infty \leq \eps$ then there must exists rectangles $R, \tilde R \in \mathcal A$ with a non-empty intersection such that $\bm x \in R, \bm y \in \tilde R$. Since each rectangle has diameter $\eps$ with respect to the sup norm, the claim follows from the triangle inequality.}

Next, for fixed $R=R_{k_1 \dots k_d}$ and $\bm x, \bm y \in R = R_{k_1,\dots,k_d} \subset [0,1]^d$ we have 
\begin{align*}
    \alpha_n(\bm x) -\alpha_n(\bm y)  
    &=
    \alpha_n(x_1, \dots, x_d) \pm \alpha(y_1, x_2, \dots, x_d) \pm \alpha_n(y_1, y_2, x_3, \dots, x_d) 
    \\& \hspace{3cm} \pm \dots \pm \alpha_n(y_1, \dots, y_{d-1}, x_d) - \alpha_n(y_1, \dots, y_d)
    \\&=
    \sum_{j \in [d]} \alpha_n(y_{1:j-1},x_{j:d}) - \alpha_n(y_{1:j},x_{j+1:d})
\end{align*}
where $x_{i:j}=(x_i, \dots, x_j)$ for $i \le j$, and where $x_{i:j}$ should be interpreted as `not being there' for $i>j$. 
In what follows, with a slight abuse of notation, write $\alpha_n(A) = \sqrt n\{\mu_n(A)-\mu(A)\}$ for Borel sets $A$. This defines a finite signed measure.
Fix $j \in [d]$. First consider the case $x_j > y_j$. Then 
\begin{align*}
T_{nj}(\bm x, \bm y) 
:=&\ \alpha_n(y_{1:j-1},x_{j:d}) - \alpha_n(y_{1:j},x_{j+1:d})
\\=&\
\alpha_n(y_{1:j-1}, x_j, x_{j+1:d}) - \alpha_n(y_{1:j-1},y_j, x_{j+1:d})
\\=&\
\alpha_n(A_{j>,\bm x, \bm y}),
\end{align*}
with
\[
A_{j>,\bm x, \bm y} := [0,y_1) \times \dots \times [0,y_{j-1}) \times [y_j, x_j) \times [0,x_{j+1}) \times \dots \times [0,x_{d}) \in \mathcal A_{j,k_j}.
\]
Likewise, if $x_j < y_j$, we have
\[
T_{nj}(\bm x, \bm y)  = - \alpha_n( A_{j<,\bm x, \bm y}),
\]
where 
\[
A_{j<,\bm x, \bm y} := [0,y_1) \times \dots \times [0,y_{j-1}) \times [x_j, y_j) \times [0,x_{j+1}) \times \dots \times [0,x_{d}) \in \mathcal A_{j,k_j},
\]
and if $x_j=y_j$, we have $T_{nj}(\bm x, \bm y) =0$.
Overall, $|T_{nj}(\bm x, \bm y) | \le \sup_{A \in \mathcal A_{j,k_j}} | \alpha_n(A)|$, which implies
\begin{align*}
\sup_{\bm x, \bm y \in  R} |\alpha_{n}(\bm x) - \alpha_{n}(\bm y)|
\le
\sum_{j \in [d]} \sup_{A \in \Ac_{j,k_j}} | \alpha_{n}(A) |
\le
d \max_{j \in [d]}\max_{k_j \in K_j} \sup_{A \in \Ac_{j,k_j}} | \alpha_{n}(A) |.
\end{align*}
Hence,
\begin{align*}
\omega_{\alpha_{n}}(\eps, [\bm a, \bm b])
\le 2d \max_{j \in [d]} \max_{k_j \in K_j}\sup_{A\in \mathcal A_{j,k_j}} | \alpha_{n}(A) |,
\end{align*}
and thus, with probability at least $1-2 \| \bm b - \bm a \|_1 \tilde \delta /\eps$, 
\[
\omega_{\alpha_{n}}(\eps, [\bm a, \bm b])
\le
2d \sqrt n\Big[ \frac 2{3n} \log(1/{\tilde \delta}) + \sqrt{\frac{\eps}{n}}\Big\{2 \sqrt{\log(1/{\tilde \delta})}+ 60\sqrt{2d} \Big\} \Big].
\]
With $\tilde \delta = \eps\delta / ( 2 \| \bm b - \bm a \|_1 )$, the upper bound can be rewritten as
\[
2d \Big[ \frac 2{3\sqrt n} \log\Big( \frac{2\| \bm b - \bm a \|_1}{\eps\delta}\Big) +\Big\{2 \sqrt{\eps\log\Big( \frac{2\| \bm b - \bm a \|_1}{\eps\delta}\Big)}+ 60\sqrt{2d\eps} \Big\} \Big],
\]
which is the first statement of the lemma. 

Regarding the second statement concerning $\beta_{n}$, note that the events of interest in its definition satisfy
\[
\big\{ \exists j \in [d]: V_{ij}<u_j \big\}
=
\big\{ \forall j \in [d]: V_{ij} \ge u_j \big\}^c
=
\big\{ \forall j \in [d]: U_{ij} \le 1-u_j \big\}^c
\]
where $U_{ij}=1-V_{ij}$.
As a consequence, 
\[
\beta_{n}(\bm u) = - \tilde \alpha_{n}^\circ(\bm 1-\bm u)
\]
where 
\[
\tilde \alpha_{n}^\circ(\bm u)
=
\frac1{\sqrt n} \sum_{i=1}^n \big[ \bm1(\forall j \in [d]: U_{ij}  \le u_j) - \Prob(\forall j \in [d]: U_{ij} \le u_j) \big].
\]
Hence, $\omega_{\beta_{n}}(\eps; [\bm a, \bm b]) = \omega_{\tilde \alpha_{n}^\circ}(\eps; [\bm 1- \bm b, \bm 1 - \bm a])$.
Define $\alpha_{n}^\circ$ as in \eqref{eq:alphan}, but with $\bm V_i$ replaced by $\bm U_i$, and note that the derived probability bound holds for $\alpha_n^\circ$. Further note that $\tilde \alpha_n^\circ(\bm u) = \lim_{\eta\downarrow 0} \alpha_n^\circ(\bm u + \eta \bm 1)$ for any $\bm u \in [0,1)^d$, so that $\omega_{\tilde \alpha_{n}^\circ}(\eps; (\bm 1- \bm b, \bm 1 - \bm a)) = \omega_{ \alpha_{n}^\circ}(\eps; (\bm 1- \bm b, \bm 1 - \bm a))$. 
Moreover, for fixed $\bm a, \bm b$ we have with probability one $\tilde \alpha_{n}^\circ(\bm u) = \alpha_{n}^\circ(\bm u)$ for all $\bm u$ on the boundary of the set $[\bm 1- \bm b, \bm 1 - \bm a]$, so that in fact $\omega_{\tilde \alpha_{n}^\circ}(\eps; (\bm 1- \bm b, \bm 1 - \bm a)) = \omega_{ \alpha_{n}^\circ}(\eps; [\bm 1- \bm b, \bm 1 - \bm a])$ with probability one. The assertion for $\omega_{\beta_{n}}$ now follows from the probability bound on $\omega_{\alpha_{n}^\circ}$. 
\end{proof}

\begin{lemma}  \label{lem:continuity-lprime-good}
Let $L$ be an $d$-variate stable tail dependence function satisfying \ref{cond:smoothness-good}, and let $j \in [d]$. Then, for any $\bm y, \bm z \in E_j$ such that the rectangle $[\bm y, \bm z] = \{ \bm x \in [0,\infty)^d: y_\ell \le x_\ell \le z_\ell \text{ for all } \ell \in [d] \}$ is contained in $G_j :=G_j^{(1)} \cap \bigcap_{\ell \in [d]} G^{(2)}_{j\ell}$, 
we have
\[
|\pdL{j}(\bm y) - \pdL{j}(\bm z)| \le \Kl \max\Big\{\frac1{y_j},\frac1{z_j}\Big\} \|\bm y - \bm z\|_1.
\]
\end{lemma}

\begin{proof}[Proof of Lemma~\ref{lem:continuity-lprime-good}] 
For $t\in[0,1]$, let $\bm x(t) = \bm y + t (\bm z - \bm y)$ denote the line segment connecting $\bm y$ and $\bm z$. Note that $x_j(t)>0$.
Since $\bm x(t) \in [\bm y, \bm z] \subset G_j$ by assumption, the function $f(t) = \pdL{j}(\bm x(t))$ is well-defined, continuous on $[0,1]$ and continuously differentiable on $(0,1)$ with derivative
\[
f'(t) = \sum_{\ell \in [d] : y_\ell>0 \text{ or } z_\ell>0 } (z_\ell-y_\ell) \pdL{j\ell}( \bm x(t)).
\]
By the mean-value theorem, there exists some $t^* \in (0,1)$ such that
\[
\pdL{j}(\bm z) - \pdL{j}(\bm y) = f(1) - f(0) = f'(t^*) = \sum_{\ell \in [d] : y_\ell>0 \text{ or } z_\ell>0 } (z_\ell-y_\ell) \pdL{j\ell}( \bm x(t^*)).
\]
Hence, by Condition~\ref{cond:smoothness-good},
\begin{align*}
|\pdL{j}(\bm y) - \pdL{j}(\bm z)| 
&\le 
\max_{\ell \in [d] : y_\ell>0 \text{ or } z_\ell>0 } \sup_{t \in (0,1)} |\pdL{j\ell}( \bm x(t))| \times \sum_{\ell \in [d] : y_\ell>0 \text{ or } z_\ell>0 } |y_\ell-z_\ell|
\\&\le 
\Kl \Big( \sup_{t \in (0,1)} \frac{1}{x_j(t)}\Big)  \times \sum_{\ell \in [d]} |y_\ell - z_\ell|
\end{align*}
Since the denominator in the supremum on the right-hand side is an affine linear function, the supremum must be attained at one of the boundary points 0 or 1, with $1/x_j(0)=1/y_j$ and $1/x_j(1) = 1/z_j$. 
As a consequence,  $\sup_{t \in (0,1)} 1/x_j(t) =\max(1/y_j, 1/z_j)$, which yields the assertion.
\end{proof}

\begin{lemma} \label{lem:ks1}
    Suppose $\bm X, \bm Y$ are $d$-variate random vectors defined on the same probability space. Then, for all $\delta>0$,
    \[
    \sup_{\bm x \in \R^d} \big| \Prob(\bm X \le \bm x) - \Prob(\bm Y\le \bm x) \big| 
    \le 
    \Prob\big(\|\bm X-\bm Y\|_\infty \ge \delta\big)
    + \sup_{\bm x \in \R^d} \Prob( \bm Y \le \bm x+\delta \bm1 ) - \Prob( \bm Y \le \bm x-\delta \bm 1),
    \]
    where $\| \cdot \|_\infty$ is the maximum norm on $\R^d$.
\end{lemma}

\begin{proof}[Proof of Lemma~\ref{lem:ks1}]
Let $\Delta = \big\{\|\bm X- \bm Y\|_\infty \geq \delta\big\}$. Then, for any $\bm x \in \R^d$,
\begin{align*}
\Prob\big(\bm X \le \bm x \big) 
\ge 
\Prob\big(\bm X \le \bm x, \Delta^c\big)
&\ge
\Prob\big( \bm Y \le \bm x - \delta \bm 1, \Delta^c\big) 
\\ &=
\Prob\big( \bm Y\le \bm x-\delta\bm1 \big) - \Prob\big(\bm Y \le \bm x-\delta \bm 1, \Delta\big)
\\& \ge 
\Prob\big(\bm Y\le \bm x-\delta\bm 1 \big) - \Prob(\Delta).
\end{align*}
As a consequence,
\begin{align*}
\Prob\big(\bm Y \leq \bm x\big) -\Prob\big(\bm X \leq \bm x \big)
& \le 
\Prob\big(\bm Y \leq \bm x\big)  - \Prob\big(\bm Y  \le \bm x - \delta \bm 1 \big) + \Prob(\Delta)
\\ &\le
\Prob\big(\bm Y \leq \bm x+\delta \bm 1\big)  - \Prob\big(\bm Y  \le \bm x - \delta \bm 1 \big) + \Prob(\Delta).
\end{align*}
Likewise,
\begin{align*}
\Prob\big(\bm X \le \bm x, \Delta^c\big)
\le 
\Prob\big(\bm Y \le \bm x + \delta\bm 1, \Delta^c \big)
\le
\Prob\big(\bm Y \le \bm x + \delta \bm 1),
\end{align*}
which implies
\begin{align*}
\Prob\big(\bm X \leq \bm x\big) -\Prob\big(\bm Y \leq \bm x \big)
&=
\Prob\big(\bm X \leq \bm x, \Delta\big) + \Prob\big(\bm X \leq \bm x, \Delta^c\big)  -\Prob\big(\bm Y \leq \bm x \big)
\\&\le 
\Prob\big(\bm X \leq \bm x, \Delta\big)  + \Prob\big(\bm Y \le \bm x + \delta\bm 1) - \Prob(\bm Y \le \bm x)
\\&\le\Prob(\Delta) + \Prob\big(\bm Y \leq \bm x+\delta \bm 1\big)  - \Prob\big(\bm Y  \le \bm x - \delta \bm 1 \big).
\end{align*}
This concludes the proof.    
\end{proof}

\begin{theorem}[Nazarov] \label{theo:nazarov}
Suppose $\bm Z \sim \mathcal N_d(\bm 0, \Sigma)$ such that $\min_{j=1}^d \Var(Z_j) \ge \sigma_{\min}^2>0$. Then, for every $\delta>0$,
\[
\sup_{\bm x \in \R^d} \Prob( \bm Z \le \bm x+\delta \bm 1 ) - \Prob( \bm Z \le \bm x-\delta \bm 1 ) \le \frac{2\delta}{\sigma_{\min}}\big( 2+ \sqrt{2\log d} \big).
\]
\end{theorem}

\begin{proof}
This is Nazarov’s inequality, see \cite{Che17-nazarov}. 
\end{proof}

\begin{theorem}[\citealp{CheCheKat23}] \label{theo:clt-new}
Let $\bm S_n = \sum_{i=1}^n \bm Y_{i,n}$ with $\bm Y_{1,n}, \dots, \bm Y_{n,n}$ independent and with $\Exp[\bm Y_{i,n}]=0, \Exp[Y_{i,n,j}^2]<\infty$, where $\bm Y_{i,n}=(Y_{i,n,1}, \dots, Y_{i,n,p})^\top$. Further suppose that $b_1,b_2>0$ and $B_n\ge 1$ are constants such that
\begin{compactenum}
    \item $\sum_{i=1}^n \Exp[Y_{i,n,j}^2] \ge b_1$ for all $j \in [p]$.
    \item $\sum_{i=1}^n \Exp[|Y_{i,n,j}|^{4}]  \le b_2 B_n^2/n$ for all $j \in [p]$.
    \item $\Exp[\exp(\sqrt n|Y_{i,n,j}|/B_n)] \le 2$ for all $i \in [n], j \in [p]$.
\end{compactenum}
Let $\Sigma_n=\Var(\bm S_n)$ and $\bm Z_n \sim \Nc_p(\bm 0, \Sigma_n)$. Then there exists a constant $\CCCK$ 
only depending on $b_1$ and $b_2$ such that
\[
\sup_{\bm x \in \R^p} \big| \Prob(\bm S_n \le \bm x) - \Prob(\bm Z_n \le \bm x) \big| 
\le 
\CCCK \Big( \frac{B_n^2 \log^5(pn)}{n} \Big)^{1/4}.
\]
\end{theorem}

\begin{proof}
    This is Theorem 1 in \cite{CheCheKat23}, with their $X_{i}$ equal to our $\sqrt n \bm Y_{i,n}$. 
\end{proof}

\begin{lemma} \label{lem:convex-differentiable}
Let $U\subset \R^d$ be an open convex set and $f:U \to \R$ a convex function. If for some $\bm x\in U$ all partial derivatives $\partial_i f(\bm x)$ exist, then $f$ is (totally) differentiable at $\bm x$.
\end{lemma}

\begin{proof} 
Since $U$ is an open set, there exists an $\eps>0$ such that $\mathcal B_\eps(\bm x) \subset U$. 
For $\bm h\in \R^d$ with $\|\bm h\|\le \eps$, define $\varphi(\bm h) = f(\bm x + \bm h) - f(\bm x) - \langle \nabla f(\bm x) , \bm h \rangle.$ Convexity of $f$ implies that $\varphi$ is convex as well. Denote by $\bm e_i$ the standard basis vectors of $\R^d$ so that $\bm h\in  \R^d$ can be written as $\bm h = h_1 \bm e_1 + \dots + h_d \bm e_d$. Then,
\[
   {\varphi(\bm h)} = {\varphi\Big(\frac 1d \sum_{i=1}^d dh_i \bm e_i\Big)} \le \frac 1d \sum_{i=1}^d \varphi(d h_i \bm e_i) \le \frac 1d \sum_{i=1}^d |\varphi(d h_i \bm e_i)|
\] 
and as a result, using $\|\bm h\| \ge |h_i|$,
\[
\frac{\varphi(\bm h)}{\|\bm h\|} \le \frac 1d \sum_{i=1}^d \frac{|\varphi(d h_i \bm e_i)|}{\|\bm h\|} \le \frac 1d \sum_{i=1}^d \frac{|\varphi(d h_i \bm e_i)|}{|h_i|}.
\]
Next, $\varphi(\bm 0) = 0$ together with the convexity of $\varphi$ implies $0=\varphi(\bm h/2 - \bm h/2)\le (\varphi(\bm h) + \varphi(-\bm h))/2$ and thus $-\varphi(\bm h) \le \varphi(-\bm h)$. 
It follows that 
\[
-\frac{\varphi(\bm h)}{\|\bm h\|} \le \frac{\varphi(-\bm h)}{\|-\bm h\|} \le \frac 1d \sum_{i=1}^d \frac{|\varphi(-d h_i \bm e_i)|}{|-h_i|}.
\] 
All that remains to show is that $|\varphi(d h_i e_i)|/{| dh_i|}$ converges to $0$ for $h_i \to 0$, for each $i\in [d]$. We have 
\[
\frac{|\varphi(d h_i e_i)|}{d|h_i|} =\Big| \frac{f(x + de_i h_i)-f(x) -  \partial_i f(x) dh_i}{dh_i}\Big| = \Big| \frac{f(x + de_i h_i)-f(x)} {dh_i} -  \partial_i f (x) \Big| \to 0
\] 
for $h_i \to 0$ by definition of the partial derivatives. 
\end{proof}

The following result provides finite sample guarantees on the level of a test that is obtained by combining dependent p-values. The set-up is as follows:
suppose $\bm Y_n \in \R^p$ is an observable vector of test statistics (it is instructive to consider each coordinate $Y_{nq}, q \in [p]$, as a test statistic for which large value provide evidence against some hypothesis $H_q$), $\bm Y_n^{\gapp} \in \R^p$ is some unobservable random vector, and $\bm Y_n^* \in \R^p$ is some observable bootstrap vector, to be thought of as approximating $\bm Y_n$. For $q \in [p]$, let
\[
    \hat p_{nq} = 1-F_{nq}^*(Y_{nq}),
    \quad
    \hat p_{nq}^* = 1-F_{nq}^*(Y_{nq}^*),
\]
with $F_{nq}^*$ the conditional cdf of $Y_{nq}^*$ given the data. Let
\[
    C_n = \min_{q \in [p]} \hat p_{nq}, 
    \quad
    C_n^* = \min_{q \in [p]} \hat p_{nq}^*.
\]
It is instructive to think of small values of $C_n$ providing evidence against some intersection hypothesis $H_1 \cap \dots \cap H_p$.

\begin{proposition}
\label{prop:p-value-combination}
Let $\lambda, \lambda^*, \delta >0$, and suppose that $\bm Y_n^{\gapp}$ has a continuous cdf and that
\begin{align}
\label{eq:gauss-approx-p-values}
    d_K(\bm Y_n, \bm Y_n^{\gapp}) \le \lambda, 
    \qquad
    d_K(\mathcal L(\bm Y_n^* \mid \mathrm{data}), \bm Y_n^{\gapp}) \le \lambda^* , 
\end{align} 
the latter holding on a set $\Omega_n^*$ with $\Prob(\Omega_n^*) \ge 1 -\delta$. Then, for every $\alpha \in (0,1)$,
\[
    \Big| \Prob(C_n \le \hat q_{n,\alpha}^*) - \alpha \Big| 
    \le   
    2\delta + \lambda + (2p+1)\lambda^* ,
\]
where $q_{n,\alpha}^*$ is the $\alpha$-quantile of $\Prob(C_n^* \in \cdot \mid \mathrm{data})$.
\end{proposition}

\begin{proof}[Proof of Proposition~\ref{prop:p-value-combination}] 
Throughout, we denote the cdf of $\bm Y_n$ by $F_n$, the cdf of $\bm Y_n^{\gapp}$ by $F_n^{\gapp}$ and the conditional cdf of $\bm Y_n^*$ given the data by $F_n^*$. Moreover, we define
\[
C_n^{\gapp} = \min_{q \in [p]} p_{nq}^{\gapp}, \quad \text{ with } \quad p_{nq}^{\gapp} = 1 - F_{nq}^{\gapp}(Y_{nq}^{\gapp}).
\]

We start by noting that, if $F$ and $G$ are cdfs on the real line satisfying $\sup_{x \in \R} |F(x) - G(x) | \le \lambda$, then
\begin{equation} \label{eq:marginal-quantiles2}
    \forall \alpha \in (0,1) : \qquad F^-(\alpha - \lambda) \le  G^-(\alpha) \le  F^-(\alpha + \lambda);
\end{equation} 
here, $F^{-}(v) = \inf\{ u \in \R: F(u) \ge v\}$ for $v \in (0,1]$ and 
$F^-(v)=-\infty$ for $v \le 0$ and $F^-(v)=+\infty$ for $v>1$; note the slight difference to the generalized inverse used in  \eqref{eq:generalized-inverse}.

We will show below that
\begin{align}
\label{eq:approx-cdf-cn}
\sup_{t \in  \R}  \Big| \Prob(C_n \le t) 
-  \Prob(C_n^{\gapp} \le t) \Big| 
\le  \delta + \lambda + p \lambda^*, \\
\label{eq:approx-cdf-cn*}
\text{On } \Omega_n^*: \qquad \sup_{t \in  \R}  \Big| \Prob(C_n^* \le t \mid \mathrm{data}) 
-  
\Prob( C_n^{\gapp} \le t ) \Big| 
\le  (p+1) \lambda^*.
\end{align}
Combined with \eqref{eq:marginal-quantiles2}, the bound in \eqref{eq:approx-cdf-cn*}  then implies that, on $\Omega_n^*$,   
\[
F_{C_n^{\gapp}}^{-} \big(\alpha - (p+1)\lambda^*\big)
\le 
q_{n,\alpha}^* 
\le 
F_{C_n^{\gapp}}^{-} \big(\alpha + (p+1)\lambda^*\big),
\]
where $F_{C_n^{\gapp}}$ is the cdf of $C_n^{\gapp}$. As a consequence, by  \eqref{eq:approx-cdf-cn} 
\begin{align*}
    \Prob(C_n \le q^*_{n,\alpha}) 
    &\le  
    \Prob\big(C_n \le F_{C_n^{\gapp}}^{-} \big(\alpha + (p+1)\lambda^*\big) \big) 
    + \Prob\big((\Omega_n^*)^c\big)
    \\&\le
    \Prob\big(C_n^{\gapp} \le F_{C_n^{\gapp}}^{-} \big(\alpha + (p+1)\lambda^*\big) \big) 
    + 2\delta + \lambda + p \lambda^*
    \le
    \alpha + 2\delta + \lambda + (2p+1)\lambda^* .
\end{align*}
A lower bound can be obtain by similar arguments, and this yields the claim.

It remains to prove \eqref{eq:approx-cdf-cn} and \eqref{eq:approx-cdf-cn*}. We start with the former, fix $t\in \R$ and note that
\begin{align*}
    \Prob(C_n \le t) 
    = 
    1 - \Prob\Big( \forall q \in [p]: 1-F_{nq}^*(Y_{nq}) > t \Big) 
    &= 
    1- \Prob\Big(\forall q \in [p]: Y_{nq} < [F_{nq}^*]^-(1-t) \Big).
\end{align*}
Our assumptions imply that, on the event $\Omega_n^*$, we have 
$
\max_{p \in [q]}\sup_{x \in \R} |F_{nq}^*(x) - F_{nq}^{\gapp}(x)| \le \lambda^*,
$
where $F_{nq}^{\gapp}$ is the cdf of $Y_{nq}^{\gapp}$. As a consequence, on the same event and by \eqref{eq:marginal-quantiles2},
\begin{align} \label{eq:fn*-fng-quantiles}
    \forall q \in [p]: \qquad 
    [F_{nq}^*]^-(1-t) \ge [F_{nq}^{\gapp}]^-(1-t-\lambda^*)
\end{align}
Combining the previous inequalities and using that $\Prob((\Omega_n^*)^c) \le \delta$, we obtain that 
\begin{align*}
\Prob(C_n \le t) 
&\le 
1 - \Prob\Big(\forall q \in [p]: Y_{nq} < [F_{nq}^{\gapp}]^-(1-t-\lambda^*)\Big)  + \Prob((\Omega_n^*)^c)
\\& \le
1-  \Prob\Big(\forall q \in [p]: Y_{nq}^{\gapp} < [F_{nq}^{\gapp}]^-(1-t-\lambda^*)\Big)  + \delta + \lambda
\\&= 
1-F_n^{\gapp}\big( [F_{n1}^{\gapp}]^-(1-t-\lambda^*), \dots, [F_{np}^{\gapp}]^-(1-t-\lambda^*) \big) + \delta + \lambda
\\&=
1 - \mathfrak C_n^{\gapp}(1-t-\lambda^*, \dots, 1-t-\lambda^*) + \delta + \lambda
\end{align*}
where $\mathfrak C_n^{\gapp}$ is the copula of $\bm Y_n^{\gapp}$ (considered as a function on $\R^p$) and where we have used \eqref{eq:gauss-approx-p-values} at the second inequality. Since copulas are Lipschitz-continuous, we obtain that 
\begin{align*}
\Prob(C_n \le t) 
&\le 
1- \mathfrak C_n^{\gapp}(1-t, \dots, 1-t) + \delta + \lambda + p \lambda^*
\\&=
\Prob( C_n^{\gapp} \le t ) + \delta + \lambda + p \lambda^*,
\end{align*}
where the last equality follows from a straightforward calculation similar to the ones done above.
With the same arguments, 
\[
\Prob(C_n \le t) 
\ge 
\Prob( C_n^{\gapp} \le t ) - \delta - \lambda - p \lambda^*.
\]
The previous two equations imply \eqref{eq:approx-cdf-cn}.

It remains to prove \eqref{eq:approx-cdf-cn*}.
By the same calculations as before, and on the event $\Omega_n^*$ and for each fixed $t \in \R$,
\begin{align*}
    \Prob(C_n^* \le t \mid \mathrm{data}) 
    &= 
    1 - \Prob\Big( \forall q \in [p]: Y_{nq}^* < [F_{nq}^*]^-(1-t) \, \Big| \, \mathrm{data}\Big) 
    \\&\le 
     1 - \Prob\Big( \forall q \in [p]: Y_{nq}^{\gapp} < [F_{nq}^*]^-(1-t) \, \Big| \, \mathrm{data}\Big) + \lambda^*
     \\
     &= 1 - \Prob\Big( \forall q \in [p]: Y_{nq}^{\gapp} \le [F_{nq}^*]^-(1-t) \, \Big| \, \mathrm{data}\Big) + \lambda^*
    \\&\le
    1 - \Prob\Big( \forall q \in [p]: Y_{nq}^{\gapp} \le [F_{nq}^{\gapp}]^-(1-t-\lambda^*) \, \Big| \, \mathrm{data}\Big) + \lambda^*
   \\&=
    1 - \mathfrak C_n^{\gapp}(1-t-\lambda^*, \dots, 1-t-\lambda^*) + \lambda^*
    \\&\le 
    1- \mathfrak C_n^{\gapp}(1-t, \dots, 1-t) + (p+1) \lambda^*
    \\&= \Prob( C_n^{\gapp} \le t ) + (p+1) \lambda^*,
\end{align*} 
where we have used \eqref{eq:fn*-fng-quantiles} at the second inequality. The respective lower bound can be deduced similarly, and \eqref{eq:approx-cdf-cn*} thus follows from $\Prob(\Omega_n^*) \ge 1- \delta$.
\end{proof}

\end{document}